\documentclass[11pt]{amsart}
\usepackage{amsmath,amsfonts,amsbsy,amsgen,amscd,mathrsfs,amssymb,amsthm}
\usepackage{enumerate,mathtools}
\usepackage{caption}
\usepackage{booktabs}
\usepackage{clipboard}
\newclipboard{myclipboard}
\usepackage{tablefootnote}
\usepackage{makecell}
\usepackage[flushleft]{threeparttable}
\usepackage{fullpage}
\usepackage{url}

\usepackage{mathrsfs}
\usepackage{thmtools, thm-restate}

\usepackage[colorlinks=true,linkcolor=blue]{hyperref} 
\makeatletter
\renewcommand*{\eqref}[1]{%
  \hyperref[{#1}]{\textup{\tagform@{\ref*{#1}}}}%
}
\makeatother

\numberwithin{equation}{section}

\newtheorem{theorem}{Theorem}[section]
\newtheorem{proposition}[theorem]{Proposition}

\newtheorem{corollary}[theorem]{Corollary}
\newtheorem{lemma}[theorem]{Lemma}

\theoremstyle{definition}
\newtheorem{definition}[theorem]{Definition}

\newtheorem{remark}[theorem]{Remark}

\newtheorem*{claim*}{Claim}

\newcommand{\R}{\mathbb R}
\newcommand{\D}{\mathcal D}
\newcommand{\PP}{\mathbb P}
\newcommand{\EE}{\mathbb E}

\newcommand{\h}{\mathsf{H}}
\newcommand{\supp}{\mathop{\mathrm{supp}}}
\newcommand{\inter}{\mathop{\mathrm{int}}}
\newcommand{\diam}{\mathop{\mathrm{diam}}}

\newcommand{\Ent}{\operatorname{Ent}}
\newcommand{\Var}{\operatorname{Var}}
\newcommand{\Id}{\operatorname{Id}}
\newcommand{\Tr}{\operatorname{Tr}}

\newcommand{\Cov}{\operatorname{Cov}}
\newcommand{\HH}{H^1}
\newcommand{\cc}{\kappa}
\newcommand{\rev}[1]{{\color{black} #1}}

\title{The Brownian transport map}

\author{Dan Mikulincer}
\address{Department of Mathematics, Massachusetts Institute of Technology, 
Cambridge, MA, USA}
\email{danmiku@mit.edu}

\author{Yair Shenfeld}
\address{Department of Mathematics, Massachusetts Institute of Technology, 
Cambridge, MA, USA}
\email{shenfeld@mit.edu}

\begin{document}
\maketitle
\begin{abstract}
Contraction properties of transport maps between probability measures play an important role in the theory of functional inequalities. The actual construction of such maps, however, is a non-trivial task and, so far, relies mostly on the theory of optimal transport. In this work, we take advantage of the infinite-dimensional nature of the Gaussian measure and  construct a new transport map, based on the F\"ollmer process, which pushes forward the Wiener measure onto probability measures on Euclidean spaces. Utilizing the tools of the Malliavin and stochastic calculus in Wiener space, we show that this Brownian transport map is a contraction in various settings where the analogous questions for optimal transport maps are open.  

The contraction properties of the Brownian transport map enable us to prove functional inequalities in Euclidean spaces, which are either completely new or improve on current results. Further and related applications of our contraction results are the existence of Stein kernels with desirable properties (which lead to new central limit theorems), as well as new insights into the Kannan--Lov\'asz--Simonovits conjecture. 

We go beyond the Euclidean setting and address the problem of contractions on the Wiener space itself. We show that optimal transport maps and causal optimal transport maps (which are related to Brownian transport maps)  between the Wiener measure and other target measures on Wiener space exhibit very different behaviors.
\end{abstract}
\section{Introduction}
One of the basic tools in the study of functional inequalities in Euclidean spaces is the use of Lipschitz maps $T:\R^d\to \R^d$ \cite{cordero2002some, kim2012generalization}. 
A good starting point for this discussion is Caffarelli's contraction theorem \cite{caffarelli2000monotonicity} (see also \cite{chewi2023entropic,fathi2020proof,kolesnikov2011mass,valdimarsson2007hessian} for other proofs): If $\gamma_d$ is the standard Gaussian measure on $\R^d$ and $p$ is a probability measure on $\R^d$ which is more log-concave than $\gamma_d$, then the optimal transport map of Brenier $T:\R^d\to \R^d$, which pushes forward $\gamma_d$ to $p$, is 1-Lipschitz. In other words, that $p$ is more log-concave than $\gamma_d$ is manifested by the contractive properties of the transport map $T$. With the existence of $T$ in hand, we can easily transfer to $p$ functional inequalities which are known to be true for $\gamma_d$. For example, the Poincar\'e inequality states that for $\eta:\R^d\to \R$ we have
\[
\Var_{\gamma_d}[\eta]\le \int |\nabla\eta|^2d\gamma_d.
\]
As $T$ is 1-Lipschitz, its derivative is bounded, $|DT|_{\text{op}}\le 1$, so
\begin{equation} \label{eq:transportproof}
	\Var_{p}[\eta]=\Var_{\gamma_d}[\eta\circ T]\le \int |\nabla(\eta\circ T)|^2d\gamma_d\le  \int |DT|_{\text{op}}^2\,(|\nabla\eta|\circ T)^2d\gamma_d\le \int |\nabla\eta|^2dp
\end{equation}
where we used that $p$ is the pushforward of $\gamma_d$ under $T$. We see that $p$ satisfies the Poincar\'e inequality  with the same constant as $\gamma_d$. 

This paper starts with the observation that since the Gaussian measure is infinite-dimensional in nature\footnote{The infinite-dimensional nature of the Gaussian is manifested by the fact that it enjoys several quantitative analytic properties, like the Poincar\'e inequality, which do not depend on the ambient dimension. In particular, the Wiener measure, which is the infinite-dimensional analogue of the Gaussian measure, satisfies the same analytic properties.}, the search for contractive transport maps from the Gaussian measure to some target measure should not be confined to Euclidean spaces, even if the target measure is a measure on $\R^d$. Specifically, we will take our source measure to be the Wiener measure (an infinite-dimensional Gaussian measure) which will allow us to take advantage of the Malliavin and stochastic calculus of the Wiener space. Given a target measure $p$ on $\R^d$, our construction relies on the F\"ollmer process: The solution $X=(X_t)$ to the stochastic differential equation
\begin{align}
\label{eq:Follmer}
dX_t=\nabla\log P_{1-t}\left(\frac{dp}{d\gamma_d}\right)(X_t)dt+dB_t, \quad t\in [0,1], \quad X_0=0
\end{align}
where $(B_t)$ is the standard Brownian motion in $\R^d$ and $(P_t)$ is the heat semigroup. This process can be seen as Brownian motion conditioned on being distributed like $p$ at time 1; i.e., $X_1\sim p$. Hence, we view the solution to \eqref{eq:Follmer} at time 1 as a transport map, which we call the \emph{Brownian transport map}, $X_1:\Omega\to \R^d$ which pushes forward the Wiener measure $\gamma$ on the Wiener space\footnote{The Wiener measure $\gamma$ is such that $\Omega\ni\omega\sim \gamma$ is a standard Brownian motion in $\R^d$ where $\Omega$ is the classical Wiener space of continuous paths in $\R^d$ parameterized by time $t\in [0,1]$.} $\Omega$  to the target measure $p$ on $\R^d$. 

In the remainder of the introduction we present our results on the contractive properties of the Brownian transport map, as well as applications to functional inequalities and to central limit theorems. We also study the behavior of the Brownian transport map when considered as a map from the Wiener space to itself. This point of view further elucidates the connection between our results and optimal transport theory.

\subsection{Almost-sure contraction}
\label{subsec:uniIntro}

Before presenting our first result we discuss the types of measures for which it is reasonable to expect that the Brownian transport map will be an almost-sure contraction (the reader is referred to section \ref{sec:prem} for the exact definition). The rough intuition is that if the measure $\gamma$ is squeezed into a more concentrated measure then the transport map should be a contraction. We focus on several mechanisms which in principle could facilitate such contractions. The first mechanism works by requiring that $S:=\text{diam}(\supp (p))$ is finite so that the entire mass of $p$ is confined into a region with finite volume. The second mechanism, inspired by Caffarelli's result, works by imposing convexity assumptions on $p$: We say that $p$ is $\cc$-log-concave for some $\cc\in \R$ if it has smooth density and
\[
-\nabla^2\log\left(\frac{dp}{dx}\right)(x)	\succeq\cc \Id_d \quad\forall\,x\in \inter(\supp (p)).
\]
Note that we allow $\cc$ to take negative values and that the case $\cc=0$ corresponds to $p$ being log-concave.  When $\cc\ge 1$ we see that $p$ is more log-concave than $\gamma_d$ ($\cc=1$ when $p=\gamma_d$ as $-\nabla^2\log\left(\frac{d\gamma_d}{dx}\right)=\Id_d$) so in that sense $p$ is more concentrated than $\gamma_d$ and we expect some type of contraction.

The following result shows that the Brownian transport map is an almost-sure contraction when the target measure satisfies either a convexity assumption or a finite-volume of support assumption.  For example, as will be clear from the subsequent discussion, Theorem \ref{thm:uniinformal} always improves on the analogous result of Caffarelli which states that when $p$ is $\cc$-log-concave, for $\cc>0$, the optimal transport map is $\frac{1}{\sqrt{\cc}}$-Lipschitz. In the remainder of the paper we refer to $\ell$-Lipschitz maps as \emph{contractions with constant $\ell$}.
\begin{theorem}
\label{thm:uniinformal}{\normalfont{(Almost-sure contraction)}}
Let $p$ be a $\cc$-log-concave measure for some $\cc\in \R$ and let $S:=\textnormal{diam}(\supp (p))$. \\

\begin{enumerate}[(i)]
\item If \rev{$\cc S^2\ge 4$} then the Brownian transport map between $\gamma$ and $p$ is an almost-sure contraction with constant $\frac{1}{\sqrt{\cc}}$. \\

\item If \rev{$\cc S^2<4$} then the Brownian transport map between $\gamma$ and $p$ is an almost-sure contraction with constant \rev{$\left(\frac{e^{2-\frac{\cc}{2} S^2}+1}{8}\right)^{1/2}S$}.
\end{enumerate}
\end{theorem}

To unpack Theorem \ref{thm:uniinformal} let us consider some of its important special cases.

\begin{itemize}

\item $S<\infty$ and $\cc=0$. This setting corresponds to the case where $p$ is log-concave with bounded convex support. It is an open question  \cite[Problem 4.3]{kolesnikov2011mass} whether the optimal transport map of Brenier between $\gamma_d$ and $p$ is a contraction with a dimension-free constant. On the other hand, Theorem \ref{thm:uniinformal} shows that the Brownian transport map between $\gamma$ and $p$ is in fact an almost-sure contraction with a dimension-free constant of the optimal dependence $O(S)$. \\

\item $\cc>0$. If \rev{$\cc S^2\ge 4$} then we obtain the exact  analogue of Caffarelli's result for the optimal transport map \cite[Theorem 2.2]{kolesnikov2011mass}. If \rev{$\cc S^2<4$} then part (ii) shows that the Brownian transport map is an almost-sure contraction with constant \rev{$\left(\frac{e^{2-\frac{\cc}{2} S^2}+1}{8}\right)^{1/2}S\le \frac{1}{\sqrt{\cc}}$}; the last inequality holds as \rev{$\cc S^2<4$}. Thus, we get an improvement on the analogue of Caffarelli's result. \\

\item $S<\infty$ and $\cc<0$. In this setting, only part (ii) applies and we see that the Brownian transport map is an almost-sure contraction with constant \rev{$\left(\frac{e^{2+\frac{|\cc|}{2} S^2}+1}{8}\right)^{1/2}S$}. There are no analogous results for other transport maps. \\

%

\item $S=\infty$ and $\cc\le 0$. The bounds provided by Theorem \ref{thm:uniinformal} are trivial in this case. This is unavoidable, as we explain in section \ref{subsec:kls}.
\end{itemize}
\vspace{0.1in}

\begin{remark}
\label{rem:rhoS}
The distinction between \rev{$\cc S^2\ge 4$ and $\cc S^2<4$} is not necessary and one could get more refined results; see Remark \ref{rem:rhoSBL}. The formulation of Theorem \ref{thm:uniinformal}, however, is the cleanest which is why we chose it. 
\end{remark}

Theorem \ref{thm:uniinformal} goes beyond the above  examples by capturing the effect of the interplay of convexity (including $\cc<0$) and support size on the contraction properties of the Brownian transport map.\footnote{Note that the quantity $\cc S^2$ is scale-invariant since  the effect of rescaling $\R^d$ by $r$ turns a $\cc$-log-concave measure into a $\frac{\cc}{r^2}$-log-concave measure while the size of the support of the measure rescales from $S$ to $rS$.}

The reason why we can prove the results in Theorem \ref{thm:uniinformal}, which are unknown for the Brenier map, is because the Malliavin calculus available in the Wiener space allows us to write a differential equation for the derivative of the Brownian transport map, which in turn shows that it is a contraction. This feature does not have an analogue in optimal transport maps (but see \cite[equation (1.8)]{kim2012generalization} for a different transport map). Moreover, as will be shown in section \ref{sec:opt}, trying to replace the Brownian transport map by the optimal transport map on the Wiener space (see section \ref{subsec:Wiener} for more details) is not possible since the optimal transport map on Wiener space will essentially reduce to the optimal transport map between $\gamma_d$ and $p$ for which the desired contraction properties are not known. 

In our second result we identify a third mechanism to promote the existence of contractive transport maps. In essence, the idea is that taking well-behaved mixtures of Gaussians will also have tame concentration profiles. Indeed, in the case where the mixing measure has bounded support we establish the Brownian transport map is a contraction.
\begin{theorem}{\normalfont{(Gaussian mixtures)}}
	\label{thm:Gaussmix}
	Let $p:=\gamma_d\star \nu$ be the convolution of the standard Gaussian measure $\gamma_d$ with a probability measure $\nu$ on $\R^d$ \rev{such that $R:=\diam(\supp(\nu))$}. Then, the Brownian transport map between $\gamma$ and $p$ is an almost-sure contraction with constant \rev{$\left(e^{\frac{1}{2}R^2}-1\right)^{1/2}\frac{\sqrt{2}}{R}$}.
\end{theorem}
In the one-dimensional case, an analogue of Theorem \ref{thm:Gaussmix} was established in \cite{zimmermann2016elementary}, for the Brenier map. The proof relied on the explicit expression for transport maps between measures on the real line. While it is unknown whether the Brenier map enjoys similar properties in higher dimensions, our analysis of the Brownian transport map affords Theorem \ref{thm:Gaussmix} as an extensive generalization to arbitrary dimensions. As a corollary we are able to deduce several new functional inequalities, as well as improve upon existing ones, for Gaussian mixtures (see section \ref{sec:funcineq} below).
\subsection{Functional inequalities} \label{sec:funcineq}
Once Theorems \ref{thm:uniinformal} and \ref{thm:Gaussmix} are established, the generality of the transport method towards functional inequalities opens the door  to the improvement of numerous functional inequalities. Section \ref{sec:funct} is dedicated to proving such results which include the isoperimetric inequality, $\Psi$-Sobolev inequalities; a generalization of the log-Sobolev inequality, and $q$-Poincar\'e inequalities, a generalization of the Poincar\'e inequality.
In Table 1 we summarize our results and note the ones which seem to be new. The definitions and exact statements are deferred to Section \ref{sec:funct}.
{\small
\begin{table}[h!]\label{tbl:results} 
	
	\centering

	\begin{threeparttable}
	\begin{tabular}{|l|c|c|c|c|c|c|} 
		\toprule
		\textbf{Measures}   &  \textbf{Poincar\'e} & \textbf{Log-Sobolev} & \textbf{Isoperimetric} & \textbf{q-Poincar\'e} & \textbf{$\Psi$-log-Sobolev} \\
		\midrule
		\midrule
		\makecell{$p$ is $\cc$-log-concave \\ $\cc > 0$} & $\frac{1}{\kappa}$ \cite{brascamp1976extensions}  & $\frac{1}{\kappa}$ \cite{bakry1985diffusions} &  $\frac{1}{\sqrt{\kappa}}$\cite{bakry1996levy} & $\frac{1}{\cc^\frac{q}{2}}$\cite{calderon2019functional} & $\frac{1}{\kappa}$ \cite{bakry1985diffusions} \\
		\midrule
		\makecell{$p$ is $\cc$-log-concave \\ $S:=\diam(\supp(p))$\\
			\rev{$\cc S^2 \leq 4$}} & \makecell{\rev{$\frac{(e^{2-\frac{\cc}{2} S^2}+1)S^2}{8}$} \\ \cite{PW60,bakry2000some}}  & \makecell{\rev{$\frac{(e^{2-\frac{\cc}{2} S^2}+1)S^2}{8}$} \\ \cite{frieze1999log,calderon2019functional}}  & \makecell{\rev{$\sqrt{\frac{(e^{2-\frac{\cc}{2} S^2}+1)S^2}{8}}$}\\ \cite{milman2015sharp}} & \makecell{\rev{$\sqrt{\frac{(e^{2-\frac{\cc}{2} S^2}+1)S^2}{8}}^q$}\\ Theorem \ref{thm:qpoinc}\tnote{a}}  & \makecell{\rev{$\frac{(e^{2-\frac{\cc}{2} S^2}+1)S^2}{8}$}\\ Theorem \ref{thm:PHIsob}} \\ 
		\midrule
		\makecell{$p = \gamma_d\star \nu$\\
		$R:=\diam(\supp(\nu))$} & \makecell{\rev{$\frac{2\left(e^{\frac{1}{2}R^2}-1\right)}{R^2}$}\\Theorem \ref{thm:qpoinc}\tnote{b}
	} & \makecell{\rev{$\frac{2\left(e^{\frac{1}{2}R^2}-1\right)}{R^2}$}\\Theorem \ref{thm:PHIsob}\tnote{c}} &  \makecell{\rev{$\sqrt{\frac{2\left(e^{\frac{1}{2}R^2}-1\right)}{R^2}}$} \\Theorem \ref{thm:isoper}} & \makecell{\rev{$\sqrt{\frac{2\left(e^{\frac{1}{2}R^2}-1\right)}{R^2}}^q$} \\Theorem \ref{thm:qpoinc}} & \makecell{\rev{$\frac{2\left(e^{\frac{1}{2}R^2}-1\right)}{R^2}$}\\Theorem \ref{thm:PHIsob}} \\ 
		\bottomrule
	\end{tabular}
	\begin{tablenotes}
		\item[a] Comparable results with more restrictive assumptions were proven in \cite{calderon2019functional}.
		\item[b] Comparable results with worse exponent were proven in \cite{wang2016functional,bardet2018functional}.
		\item[c] Comparable results with worse exponent were proven in \cite{chen2021dimension}.
	\end{tablenotes}	
	\caption{Summary of functional inequalities obtained from Theorem \ref{thm:uniinformal} and Theorem \ref{thm:Gaussmix}. For results which were previously known we supply references, otherwise the relevant theorem is noted.}
\end{threeparttable}
\end{table}
}
As can be seen from the table, for log-concave measures some of the results are not new. However, let us note that the proofs of the mentioned results, obtained gradually over the last several decades, utilized a myriad of different techniques. These techniques include, among others, localization methods, Bakry-\'Emery calculus, and Brunn-Minkowski theory, and often require ad-hoc arguments for the specific functional inequality in question. In contrast, our transportation approach provides a unifying framework to study such functional inequalities. As a result we are also able to obtain new, previously unknown, results such as $\Psi$-log-Sobolev and $q$-Poincar\'e inequalities for log-concave measures with bounded support. While it is likely that one could use other techniques to prove comparable results, the benefit of our approach is that no further arguments are needed, other than Theorem \ref{thm:uniinformal} and arguments similar to the one outlined in \eqref{eq:transportproof}.

The bottom row of Table 1 deals with Gaussian mixtures. The question of existence of functional inequalities for a mixture of distributions, given the existence of the corresponding inequalities for the individual components, has been investigated for some time. Only recently has it been settled for the Poincar\'e and log-Sobolev inequalities, \cite[Theorem 1]{chen2021dimension}. The result of \cite{chen2021dimension} is very general and applies to many families of mixture distributions but, on the other hand, the method of proof seems to be specialized to the Poincar\'e and log-Sobolev inequalities. In this case, the generality of the transport method allows to tackle inequalities which seem to lie outside of the scope of previous methods. In addition, the generality of the method of \cite{chen2021dimension} misses the special nature of mixtures of Gaussians. Indeed, our results improve on \cite[Corollaries 1,2]{chen2021dimension}.

\subsection{Log-concave measures}  
\label{subsec:kls}
As we saw in 
section \ref{sec:funcineq}, measures which are $\cc$-log-concave (with $\cc>0$) satisfy a Poincar\'e inequality with constant $\cc^{-1}$ (which in particular does not depend on the dimension $d$). When $\cc=0$, this constant blows up which leaves open the question of the existence of a Poincar\'e inequality for log-concave measures. The Kannan--Lov\'asz--Simonovits conjecture \cite{kannan1995isoperimetric}, in one of its formulations, states that there exists a constant $C_{\text{kls}}$, which does not depend on the dimension $d$, such that for any isotropic (i.e.,  centered with covariance equal to the identity matrix) log-concave measure $p$ on $\R^d$ we have
\[
\Var_{p}[\eta]\le C_{\text{kls}}\int |\nabla\eta|^2dp\quad\text{for }\eta:\R^d\to \R.
\]
In words, any isotropic log-concave measure on $\R^d$ satisfies a Poincar\'e inequality with a constant $C_{\text{kls}}$ which is dimension-free. In light of the above discussion, transport maps offer a natural route to proving the conjecture: Given an isotropic log-concave measure $p$ on $\R^d$ we would like to construct a transport map, from $\gamma_d$ or $\gamma$, to $p$ which is an almost-sure contraction with constant $C_{\text{kls}}$. Unfortunately, in general, such a map cannot exist:  Indeed, as seen in Table 1
, such a map will imply that $p$ satisfies a log-Sobolev inequality with a dimension-free constant.  But this is known to be false because the existence of a  log-Sobolev inequality is equivalent to sub-Gaussian concentration \cite[Theorem 5.3]{ledoux2001concentration}, which does not hold for the two-sided exponential measure even though it is isotropic and log-concave. Nonetheless, the transport approach towards the conjecture can still be made to work by using a weaker notion of contraction and using an important result of E. Milman. Indeed, consider the Brownian transport map and suppose that instead of having an almost-sure bound $|\mathcal DX_1|\le C_{\text{kls}}$ we only have a bound in expectation,
\[
\EE_{\gamma}[|\mathcal DX_1|^2]\le C,
\]
for some dimension-free constant $C$. Repeating the argument above, and using H\"older's inequality, we find that 
\[
\Var_{p}[\eta]\le C~ \text{Lip}^2(\eta)
\]
where $\text{Lip}(\eta):=\sup_{x\in\R^d}|\nabla\eta(x)|$. In principle, this bound is weaker than a Poincar\'e inequality because of the use of the $L^{\infty}$ norm on the gradient rather than the $L^2$ norm. However, as shown by E. Milman  \cite{milman2009role}, a Poincar\'e inequality is equivalent, up to a dimension-free constant, to first moment concentration which follows from the above $L^{\infty}$ bound. In conclusion, the Kannan--Lov\'asz--Simonovits conjecture is proven as soon as we can show that $\EE_{\gamma}[|\mathcal DX_1|^2]\le C$. 

Significant progress towards the resolution of the Kannan--Lov\'asz--Simonovits conjecture was made in a series of works \cite{eldan2013thin, leeVempala, chen2021almost,KlartagLehec,Klartag}. Building on these results and techniques,  we are able to make the, so far missing, connection between measure transportation and the Kannan--Lov\'asz--Simonovits conjecture. 

\begin{theorem}
\label{thm:klsintro}{\normalfont{(Contraction in expectation for log-concave measures)}}
\rev{There exists a universal constant $\zeta>0$, such that for any positive integer $m$, and any $p$, an isotropic log-concave measure on $\R^d$ with compact support\footnote{The assumption of compact support does not affect the connection to the Kannan--Lov\'asz--Simonovits conjecture \cite[section 2.6]{chen2021almost}. In particular, the bounds in the theorem are independent of the size of the support of $p$.}, if $X_1:\Omega\to \R^d$ is the Brownian transport map from the Wiener measure $\gamma$ to $p$ then},
\[
\EE_{\gamma}\left[|\D X_1|_{\mathcal L(H,\R^d)}^{2m}\right]\le \zeta^m (2m+1)! (\log d)^{12m},
\]
where $|\cdot|_{\mathcal L(H,\R^d)}$ is the operator norm of operators from the Cameron-Martin space $H$ to $\R^d$. 
\end{theorem}

Taking $m=1$ in Theorem \ref{thm:klsintro} we see that $C_{\text{kls}}$ is almost dimension-free, as would be expected from the Kannan--Lov\'asz--Simonovits conjecture. In fact, our transport perspective allows us to go beyond $m=1$, which is needed for the applications outlined below.

\begin{remark}
As explained in this section, an expectation bound of the form $\EE_{\gamma}[|DT|^2]<C$, where  $T$ is a transport map from either $\gamma$ or $\gamma_d$ to $p$, with a dimension-free universal constant $C$, would lead to a proof of the  Kannan--Lov\'asz--Simonovits conjecture. In fact, one of the novelties of the proof of Theorem \ref{thm:klsintro} is that it reveals that the reverse is also true, up to $\log d$ factors, when $T$ is the Brownian transport map. That is, assuming that the Kannan--Lov\'asz--Simonovits conjecture is true we would get, up to $\log d$ factors, an expectation bound $
\EE_{\gamma}\left[|\D X_1|_{\mathcal L(H,\R^d)}^{2m}\right]\le C_m$ for a dimension-free universal constant $C_m$ which depends only on $m$.
\end{remark}

\subsection{Stein kernels and central limit theorems}
Our proof of Theorem \ref{thm:klsintro} is based on known results concerning $C_{\text{kls}}$. Thus, Theorem \ref{thm:klsintro} does not supply any new information regarding the Kannan--Lov\'asz--Simonovits conjecture itself.
However, for isotropic log-concave measures, the transport approach is useful not only in the study of the conjecture but also in the theory of Stein kernels. As will become evident soon, these results go beyond the Poincar\'e inequality, and hence do not follow from the current results on the Kannan--Lov\'asz--Simonovits conjecture.

 Given a centered measure $p$ on $\R^d$, a matrix-valued map $\mathfrak{s}_p$ is called a \emph{Stein kernel} for $p$ if
\[
\EE_p[\eta(x)x]=\EE_p[\nabla \eta(x)\mathfrak{s}_p(x)]
\]
for a big-enough family of functions $\eta:\R^d\to \R$. Gaussian integration by parts shows that $p=\gamma_d$ if and only if the constant matrix $\Id_d$ is a Stein kernel for $p$. Hence, the distance of $\mathfrak{s}_p$ from $\Id_d$ controls the extent to which $p$ can be approximated by $\gamma_d$. Specifically, the \emph{Stein discrepancy} of $\mathfrak{s}_{p}$ is defined as
\[
S^2(\mathfrak{s}_p):=\EE_p[|\mathfrak{s}_p-\Id_d|_{\text{HS}}^2].
\]
The quantity $S(\mathfrak{s}_p)$ plays an important role in functional inequalities \cite{ledoux2017deficit, saumard2019weighted, nourdin2009second} and normal approximations \cite{nourdin2012normal, ledoux2015stein, fathi2019stein, courtade2019existence}. For applications, often it is enough to bound $\EE_p[|\mathfrak{s}_p|_{\text{HS}}^2]$. While in one dimension the Stein kernel of a given measure is unique and given by an explicit formula, in high-dimensions these kernels are non-unique and their construction is often non-trivial \cite{courtade2019existence}. It was observed in \cite{chatterjee2009fluctuations} that transport maps with certain properties are good candidates for constructing Stein kernels. We will follow this strategy and construct Stein kernels with small Hilbert-Schmidt norm, based on the Brownian transport map, for a very large class of measures.
\begin{restatable}{theorem}{stein}\label{thm:SteinIntro}
	{\normalfont{(Stein kernels)}} Let $p$ be an isotropic log-concave measure on $\R^d$ with compact support. Let $\chi:\R^d\to \R^k$ be a continuously differentiable function with bounded partial derivatives such that $\EE_p[\chi]=0$ and $\EE_p[|\nabla\chi|_{\textnormal{op}}^8]<\infty$. Then, the pushforward measure $q:=\chi_*p$ on $\R^k$ admits a Stein kernel $\tau_q$ satisfying 
	\[
	\EE_q[|\tau_q|_{\textnormal{HS}}^2]\le a d (\log d)^{24}\sqrt{\EE_p[|\nabla\chi|_{\textnormal{op}}^8]},
	\]
 for some universal constant $a>0$.
\end{restatable}

For example, taking $k=d$ and $\chi(x) := x$ shows that for $p$ isotropic and log-concave
\begin{align}
	\label{eq:steinbd}
	\EE_p[|\tau_p|_{\text{HS}}^2]\le a d (\log d)^{24}.
\end{align}
The analogous bound of \eqref{eq:steinbd}, with a better polylog and a different Stein kernel $\mathfrak{s}_p$, follows from \cite{courtade2019existence} that showed 
$\EE_p[|\mathfrak{s}_p|_{\text{HS}}^2]\le dC_p$, where $C_p$ is the Poincar\'e constant of $p$. Indeed, since $C_p\le c\log d$ by the result of \cite{Klartag}, the bound \eqref{eq:steinbd} holds for $\mathfrak{s}_p$.  However, unlike previous constructions, our construction is well-behaved with respect to compositions. It allows to consider general $\chi$, which leads to the existence of Stein kernels $\tau_q$ with bounded Hilbert-Schmidt norm, where now $q$ does not necessarily satisfy a Poincar\'e inequality.

As a concrete application we will use Theorem \ref{thm:SteinIntro} to deduce new central limit theorems with nearly optimal convergence rates. The best dimensional dependence in the convergence rate one could expect is of order $\sqrt{\frac{d}{n}}$, as can be seen by considering product measures. Most known results establish general rates of convergence, in various distances, which are not better than $\frac{d}{\sqrt{n}}$, and typically require a super-linear dependence in the dimension (see \cite{bentkus2004lypanunv,chaterjee2008multivariate} for some notable examples). To improve on such bounds, several recent works have shown that, by imposing strong structural assumptions on the common law of the summands, one can reduce the rate of convergence to $\sqrt{\frac{d}{n}}$. However, these works dealt with highly regular measures, such as log-concave measures \cite{eldan2020clt}, measures with small support \cite{zhai2018high}, or measures satisfying a Poincar\'e inequality \cite{fathi2019stein,courtade2019existence}. These assumptions can be restrictive for certain applications involving heavy-tailed measures, whose moment generating function may not be well-defined and hence do not have sub-exponential tails.

We will bypass the above restrictive assumptions by utilizing the Stein kernel approach to normal approximations combined with Theorem \ref{thm:SteinIntro}. 
The starting point is the inequality,
\begin{align}
\label{eq:WS}
W_2^2(p,\gamma_d)\le S^2(\mathfrak{s}_p),
\end{align}
valid for any Stein kernel $\mathfrak{s}_p$, where $W_2$ is the Wasserstein 2-distance \cite[Proposition 3.1]{ledoux2015stein}. 
 In particular, inequality \eqref{eq:WS} can be used  to prove central limit theorems: Suppose that $p$ is isotropic and let $\{Y_i\}$ be an i.i.d. sequence sampled from $p$. Then, as shown in \cite[section 2.5]{ledoux2015stein}, for every given Stein kernel $\mathfrak{s}_p$, there exists a Stein kernel $\mathfrak{s}_{p_n}$, where $\frac{1}{\sqrt{n}}\sum_{i=1}^nY_i\sim p_n$, such that  
\[
S^2(\mathfrak{s}_{p_n})\le \frac{S^2(\mathfrak{s}_p)}{n}.
\]
Combining \eqref{eq:WS} with the triangle inequality thus yields
\[
W_2^2\left(p_n,\gamma_d\right)\le \frac{2}{n}\left\{\EE_p[|\mathfrak{s}_p|_{\text{HS}}^2]+d\right\}.
\]
The upshot of this discussion is that if we can construct a Stein kernel $\mathfrak{s}_p$ with a small Hilbert-Schmidt norm we then obtain a central limit theorem with a good rate. In particular, whenever $\EE_p[|\mathfrak{s}_p|_{\text{HS}}^2]=O(d)$, we get an $\sqrt{\frac{d}{n}}$ rate of convergence. Using our Stein kernel $\tau_p$ and the bound in Theorem \ref{thm:SteinIntro} we obtain:

\begin{corollary}
\label{cor:CLTIntro}{\normalfont{(Central limit theorem)}}
Let $p, q$, and $\chi$ be as in Theorem \ref{thm:SteinIntro}, and further suppose that $q$ is isotropic. Then, if $\{Y_i\}$ are i.i.d. sampled from $q$ we have, with $\frac{1}{\sqrt{n}}\sum_{i=1}^nY_i\sim q_n$, 
\[
\rev{W_2^2(q_n,\gamma_k)\le a \frac{\sqrt{\EE_p[|\nabla\chi|_{\textnormal{op}}^8]}\,d(\log d)^{24}+k}{n},
}\]
for some universal constant $a>0$.
\end{corollary}
Corollary \ref{cor:CLTIntro} allows to significantly relax the regularity assumptions of the above mentioned works, while still maintaining a nearly optimal rate of convergence. For example, if $\chi$ has a quadratic growth then $q$ can have super-exponential tails, so it cannot satisfy a Poincar\'e inequality. In contrast, in such a setting, Corollary \ref{cor:CLTIntro} provides results which are comparable, up to the Sobolev norm of $\chi$ and $\log d$ terms, to the ones obtained for log-concave measures \cite{courtade2019existence, fathi2019stein}. Another appealing feature is the ability to treat singular measures when $k> d$. A particular case of interest is  $\chi(x) := x^{\otimes m}$ for some positive integer $m$. Even though $q$ may be heavy-tailed in this case, since $\EE_p[|\nabla\chi|_{\text{op}}^8]$ is finite when $p$ is log-concave we get a central limit theorem for sums of i.i.d. tensor powers. Proving such a result was a question posed in \cite[section 3]{Mikulincer} where it was also suggested that finding an appropriate transport map could prove useful. Using our construction, Corollary \ref{cor:CLTIntro} resolves this question.
\subsection{Contractions on Wiener space}
\label{subsec:Wiener}
In the previous sections we viewed the solution $X$ to \eqref{eq:Follmer} as a map $X_1:\Omega\to \R^d$. We can go however beyond the Euclidean setting by not restricting ourselves to the value of $X$ at time $t=1$. We then get a map $X:\Omega\to \Omega$ which transports the Wiener measure $\gamma$ to the measure $\mu$ on $\Omega$ given by $d\mu(\omega)= \frac{dp}{d\gamma_d}(\omega_1)d\gamma (\omega)$ for $\omega\in \Omega$ where $\omega_t$ is the value of $\omega$ at time $t\in [0,1]$. In words, $\mu$ is obtained from $\gamma$ by reweighting the probability of $\omega$ according to $\gamma$ by the value of $\frac{dp}{d\gamma_d}:\R^d\to \R^d$ at $\omega_1\in\R^d$. This leads to the following question: Is $X$ a contraction, in a suitable sense, from the Wiener measure $\gamma$ to $\mu$ under appropriate conditions on $\mu$? In fact, this question can be placed in the general context of transport maps on Wiener space as we now explain.

We start by recalling that the Wiener space $\Omega$ contains the important Cameron-Martin space $\HH$ whose significance lies in the fact that the law of $\omega+h$ is absolutely continuous with respect to the law of $\gamma$ whenever $h\in \HH$. The Cameron-Martin space is the continuous injection of the space $H:=L^2(\Omega,\R)$ under the anti-derivative map $\dot{h}\in H\mapsto h:=\int_0^{\cdot}\dot{h}\in \HH$ and it induces a cost on $\Omega$ by setting, for $x,y\in \Omega$, $|x-y|_{\HH}$ with the convention that $|x-y|_{H^1}=+\infty $ if $x-y\notin \HH$. Based on this cost two notions of optimal transport maps on $\Omega$ can be defined:

Given probability measures measure $\nu,\mu$ on $\Omega$ let $\Pi(\nu,\mu)$ be the set of probability measures  on $\Omega\times \Omega$ such that their two marginals are equal to $\nu$ and $\mu$, respectively. The (squared) Wasserstein 2-distance between $\nu$ and $\mu$ is defined as 
\[
W_2^2(\nu,\mu)=\inf_{\pi\in\Pi(\nu,\mu)}\int_{\Omega\times \Omega}|x-y|_{\HH}^2d\pi(x,y). 
\]
Assuming that $W_2^2(\nu,\mu)<\infty$, the \emph{optimal transport map} $O:\Omega\to \Omega$, when its exists, is a map which transports $\nu$ into $\mu$ satisfying
\[
\int_{\Omega}|\omega-O(\omega)|_{\HH}^2d\nu(\omega)=W_2^2(\nu,\mu).
\]
This definition is the generalization of the definition appearing in the classical optimal transport theory on Euclidean spaces \cite{villani2008optimal}. In Euclidean spaces, the existence of optimal transport maps and their regularity was proven by Brenier \cite[Theorem 2.12]{villani2008optimal} while the analogous result in Wiener space is due to Feyel and {\"U}st{\"u}nel \cite{feyel2004monge}. Since $W_2^2(\nu,\mu)<\infty$, we may write $O(\omega)=\omega+\xi(\omega)$ where $\xi:\Omega\to \HH$ so that 
\[
W_2^2(\nu,\mu):=\inf_{\xi} \EE_{\gamma}\left[|\xi(\omega)|_{\HH}^2\right]
\]
where the infimum is taken over all maps $\xi:\Omega\to \HH$ such that $\text{Law}(\omega+\xi(\omega))=\mu$ with $\omega\sim \nu$. Importantly, we do not make the requirement that $\xi(\omega)$ is an adapted\footnote{Colloquially, a process is adapted if it cannot anticipate the future; see section \ref{sec:prem} for the precise definition.} process. We now turn to the second notion of optimal transport: Given probability measures $\nu,\mu$ on $\Omega$, we define\footnote{This definition is the analogue of the Monge problem rather than the Wasserstein distance. There is a more general definition \cite{lassalle2013causal} of causal optimal transport corresponding to adapted  Wasserstein distance but this is not important for our work.} the \emph{causal optimal transport map} $A:\Omega\to \Omega$ to be the map which transports $\nu$ to $\mu$ satisfying 
\[
\int_{\Omega}|\omega-A(\omega)|_{\HH}^2d\nu(\omega)=\inf_{\xi \text{ adapted}} \EE_{\gamma}\left[|\xi(\omega)|_{\HH}^2\right]
\]
where the infimum is taken over all maps $\xi:\Omega\to \HH$ such that $\text{Law}(\omega+\xi(\omega))=\mu$ with $\omega\sim \nu$, and with the additional requirement that $\xi$ is an adapted process. This notion of optimality, sometimes referred to as \emph{adapted optimal transport}, has recently gained a lot of traction (e.g, \cite{bartl2021wasserstein} and references therein). 

The connection between these transport maps and the Brownian transport map follows from the work of Lassalle \cite{lassalle2013causal}. It turns out that when $d\mu(\omega)= \frac{dp}{d\gamma_d}(\omega_1)d\gamma (\omega)$ for $\omega\in\Omega$, the causal optimal transport map $A:\Omega\to\Omega$, which transports $\gamma$ to $\mu$, is precisely the F\"ollmer process $X:\Omega\to \Omega$. This is essentially a consequence of Girsanov's theorem as well as the entropy-minimization property of the F\"ollmer process \cite{follmer1988random}.

Once a notion of optimal (causal or non-causal) transport map in Wiener space is established, the question of contraction arises: Given a measure $\mu$ on $\Omega$ which is more log-concave than the Wiener measure $\gamma$, is either $O$ or $A$ a contraction? We are not aware of any such results in the current literature. To make this question precise we need a notion of convexity on $\Omega$ as well as a notion of contraction. We postpone the precise definitions to section \ref{sec:opt} and for now denote such a notion of contraction as \emph{Cameron-Martin contraction}. Let us state some of our results in this direction.

\begin{theorem}{\normalfont{(Cameron-Martin contraction)}}
\label{thm:strongInformal} 
\vspace{0.1in}

\begin{itemize}
\item Let $p$ be any 1-log-concave measure on $\R^d$ and let $\mu$ be a measure on the Wiener space given by $d\mu(\omega)=\frac{dp}{d\gamma_d}(\omega_1)d\gamma(\omega)$ for $\omega\in \Omega$. Then, the optimal transport map $O$ from the Wiener measure $\gamma$ to $\mu$ is a Cameron-Martin contraction with constant 1.
\\

\item There exists a 1-log-concave measure $p$ on $\R^d$ such that the causal transport map $A$ between $\gamma$ to $\mu$, where $d\mu(\omega)=\frac{dp}{d\gamma_d}(\omega_1)d\gamma(\omega)$ for $\omega\in \Omega$, is not a Cameron-Martin contraction with any constant. 
\end{itemize}
\end{theorem}
The first part of the theorem is a straightforward consequence of Caffareli's contraction theorem. The second part requires some work by constructing a non-trivial example (see Remark \ref{rem:nontrivcexam}) where the causal optimal transport map fails to be a Cameron-Martin contraction.

\subsection*{Organization of the paper} 
Section \ref{sec:prem} contains the preliminaries necessary for this work including the definition of the Brownian transport map based on the F\"ollmer process. Section \ref{sec:causal} contains the construction of the F\"ollmer process and the analysis of its properties which then leads to the almost-sure contraction properties of the Brownian transport map. The main results in this section are contained in Theorem \ref{thm:causalmain}. Section \ref{sec:contrctlc} focuses on the setting where the target measure is log-concave with compact convex support and shows that, in this setting, we can bound the moments of the derivative of the Brownian transport map; the main result is Theorem \ref{thm:KLS}. In addition, section \ref{sec:contrctlc} contains a short explanation of the connection between stochastic localization and the F\"ollmer process.  In section \ref{sec:funct} we use the almost-sure contraction established in Theorem \ref{thm:causalmain} to prove new functional inequalities. In addition, section \ref{sec:funct} contains our results on Stein kernels and their applications to central limit theorems. In section \ref{sec:CMcontract} we set up the preliminaries necessary for the study of contraction properties of transport maps on the Wiener space itself. In section \ref{sec:causal_strong} we show that causal optimal transport maps are not Cameron-Martin contractions even when the target measure is $\cc$-log-concave, for any $\cc$.  Finally, section \ref{sec:opt} is devoted to optimal transport on the Wiener space.

\subsection*{Acknowledgments} We gratefully acknowledge the contributions of Ramon van Handel to this work: The idea of viewing the F\"ollmer process as a transport map from the Wiener space to Euclidean spaces  is due to him and, in addition, he was involved in the initial stage of this work and pointed out Remark \ref{rem:thetaToTheta}. Further, the connection to the Kannan--Lov\'asz--Simonovits conjecture is due to him. We thank him for many helpful comments which improved the quality of this manuscript. We are grateful to Sasha Kolesnikov for illuminating for us the subtleties of the Wiener space and for other discussions about this work. We also thank Ronen Eldan, Alexandros Eskenazis, Bo'az Klartag, and Emanuel Milman, and Yunan Yang for useful conversations about the topics of this paper. Finally, we thank the anonymous referees for their valuable suggestions and corrections. \rev{ChatGPT 5.6 Sol was used to fix a mistake in an earlier version of the proof of Theorem \ref{thm:SteinIntro}. It also helped sharpen the bounds in Section 3.}

The authors gratefully acknowledge the program Geometric Functional Analysis and Applications that took place in 2017 at the Mathematical Sciences Research Institute (MSRI) where their collaboration on this project was initiated. Dan Mikulincer is partially supported by a European Research Council grant no. 803084. Yair Shenfeld was partially funded by NSF grant DMS-1811735 and the Simons Collaboration on Algorithms \& Geometry; this material is based upon work supported by the National Science Foundation under Award Number 2002022.

\section{Preliminaries} 
\label{sec:prem}
For the rest of the paper we fix a dimension $d$ and let $f:\R^d\to \R_{\ge 0}$ be a function such that $\int_{\R^d}fd\gamma_d=1$ where $\gamma_d$ is the standard Gaussian measure on $\R^d$. We denote the probability measure $p(x)dx:=f(x)d\gamma_d(x)$ and further assume that the relative entropy $\h(p|\gamma_d):=\int_{\R^d} \log\left(\frac{dp}{d\gamma_d}\right)dp<+\infty$. We set $S:=\diam (\supp(p))$. We write $\langle \cdot,\cdot\rangle$ for the Euclidean inner product and $|\cdot|$ for the corresponding norm. Our notion of convexity is the following:
\begin{definition}
 A probability measure $p$ is \emph{$\cc$-log-concave} for some $\cc\in \R$ if it has smooth density, the support of $p$ is convex, and $-\nabla^2\log \left(\frac{dp}{dx}\right)(x)	\succeq \cc \Id_d$ for all $x\in \inter(\supp(p))$. 
\end{definition}
%
Next we recall some basics on the classical Wiener space and the Malliavin calculus \cite{nualart2006malliavin}.
\subsection*{Wiener space}
Let $(\Omega,\mathcal F,\gamma)$ be the classical Wiener space: $\Omega=C_0([0,1];\R^d)$ is the set of continuous functions from $[0,1]$ to $\R^d$,  $\gamma$ is the Wiener measure, and $\mathcal F$ is the completion (with respect to $\gamma$) of the Borel sigma-algebra generated by the uniform norm $|\omega|_{\infty}:=\sup_{t\in [0,1]}|\omega_t|$ for $\omega\in \Omega$. In words, a path $\omega\in \Omega$ sampled according to $\gamma$ has the law of a Brownian motion in $\R^d$ running from time 0 to time 1. We set $W_t:=W(\omega)_t:=\omega_t$ for $t\in [0,1]$ and let $(\mathcal F_t)_{t\in [0,1]}$ be the sigma-algebra on $\Omega$ generated by $(W_t)_{t\in [0,1]}$ together with the null sets of $\mathcal F$. We say that a process $(u_t)_{t\in [0,1]}$ is adapted if $u_t:\Omega\to \R^d$ is $\mathcal F_t$-measurable for all $t\in [0,1]$. 
For the rest of the paper we define a probability measure $\mu$ on $\Omega$ by
\[
d\mu(\omega)=f(\omega_1)d\gamma(\omega). 
\]
An important Hilbert subspace in $\Omega$ is the Cameron-Martin space $\HH$ which is defined as follows: Let $H=L^2([0,1],\R^d)$ and given $g\in H$ set $i(g)\in \Omega$ by $i(g)_t:=\int_0^tg_sds$. Then $\HH:=\{i(g):g\in H\}$ and we often write $h_t=\int_0^t \dot{h}_sds$ for $\dot{h}\in H$ and $h\in \HH$.  The space $\HH$ has an inner product induced from the inner product of the Hilbert space $H$, namely, $\langle h,g\rangle_{\HH}:=\int_0^1\langle \dot{h}_s,\dot{g}_s\rangle ds$.  The significance of the Cameron-Martin space is that the measure of the process $W+h=(\omega_t+h_t(\omega))_{t\in [0,1]}$ is absolutely continuous with respect to $\gamma$ whenever $h(\omega)\in \HH$ $\gamma$-a.e. and $(h_t(\omega))_{t\in [0,1]}$ is adapted and regular-enough; this is a consequence of Girsanov's theorem. Given $\dot{h}\in H$ we set $W(\dot{h}):=\int _0^1\dot{h}_tdW_t$ where the integral is the stochastic It\^o integral; in this notation, $W_t=W(1_{[0,t]})$.



Next we define the notion of contraction which is compatible with the Cameron-Martin space. 
\begin{definition}
\label{def:contract}
A measurable map $T:\Omega\to \R^d$ is an \emph{almost-sure contraction} with constant $C$ if
\[
|T(\omega+h)-T(\omega)|\le C|h|_{\HH} \quad \forall \,h\in \HH  \quad\gamma\text{-a.e.}
\]
\end{definition}

%

In Euclidean space, a function is Lipschitz if and only if its derivative (which exists almost-everywhere) is bounded. In order to find the analogue of this result for our notion of contraction we need an appropriate definition of derivatives on the Wiener space. 
\subsection*{Malliavin calculus} 
The calculus on the Wiener space was developed by P. Malliavin in the 70's and it will play an important role in our proof techniques. The basic objects of analysis in this theory is the variation of a function $F:\Omega\to \R$ as the input $\omega\in \Omega$ is perturbed. In order for the calculus to be compatible with the Wiener measure only perturbations in the direction of the Cameron-Martin space $\HH$ are considered. We now sketch the construction of the Malliavin derivative and refer to \cite{nualart2006malliavin} for a complete treatment. The construction of derivatives of $F$ starts with the definition of the class $\mathcal S$ of smooth random variables: $F\in\mathcal S$ if there exists $m\in \mathbb Z_+$ and a smooth function $\eta:\R^m\to \R$ whose partial derivatives have polynomial growth such that 
\[
F=\eta(W(\dot{h}_1),\ldots, W(\dot{h}_m))
\]
for some $\dot{h}_1,\ldots,\dot{h}_m\in  H$. The Malliavin derivative of a smooth random variable $F$ is a map $DF:\Omega\to H$ defined by
\[
DF=\sum_{i=1}^m\partial_i\eta(W(\dot{h}_1),\ldots, W(\dot{h}_m))\dot{h}_i.
\]
To get some intuition for this definition observe that
\[
\langle DF(\omega),\dot{h}\rangle_H=\frac{d}{d\epsilon}F(\omega+\epsilon h)\big|_{\epsilon=0}
\]
for $\gamma$-\text{a.e.} $\omega\in\Omega$  and every $\dot{h}\in H$ with $\HH\ni h=\int_0\dot{h}$, is the G\^ateaux derivative of $F$ in the direction $\dot{h}$. The Malliavin derivative is then extended to a larger class of functions on the Wiener space: Given $p\ge 1$ we let $\mathbb D^{1,p}$ be the closure of the class $\mathcal S$ with respect to the norm
\[
\|F\|_{1,p}:=\left(\EE_{\gamma}\left[|F|^p\right]+\EE_{\gamma}\left[|DF|_{H}^p\right]\right)^{\frac{1}{p}}.
\]
In other words, $\mathbb D^{1,p}$ is the domain in $L^p(\Omega,\gamma)$ of the Malliavin derivative operator $D$. The value of $DF\in H$ at time $t\in [0,1]$ is denoted as $D_tF$.  

The notion of the Malliavin derivative allows us to define the appropriate notion of derivatives of transport maps $F:\Omega\to \R^k$. Let $F=(F^1,\ldots,F^k):\Omega\to \R^k$ with $F^i:\Omega\to\R$ a Malliavin differentiable random variable for $i\in [k]$, and let $D_tF$ be the $k\times d$ matrix given by $[D_tF]_{ij}=D_t^jF^i$ where we use the notation $D_tF^i=(D_t^1F^i,\ldots, D_t^dF^i)\in \R^d$ for $i\in [k]$; in other words, $D_t^jF^i$ is the $j$th coordinate of $D_tF^i$. For $\gamma$-a.e. $\omega\in \Omega$ we define the linear Malliavin derivative operator $\D_{\omega} F: H\to \R^d$ by
\[
\D_{\omega} F[\dot{h}]:=\int_0^1D_tF(\omega)\dot{h}_tdt = \langle DF(\omega),\dot{h}\rangle_H,\quad \dot{h}\in H.
\]
When no confusion arises we omit the subscript dependence on $\omega$ and write $\D F$.
The next result shows that almost-sure contraction is equivalent to the boundedness of the corresponding Malliavin derivative operator. In the following we denote by $\mathcal L(H,\R^d)$ the space of linear operators from $H$ to $\R^d$ equipped with the operator norm  $|\cdot|_{\mathcal L(H,\R^d)}$. For example, note that $\D F\in\mathcal L(H,\R^d)$ for $\gamma$-a.e. $\omega\in \Omega$.
\begin{lemma}
\label{lem:rad}
$~$
\\

\begin{itemize}
\item Suppose $T:\Omega\to \R^d$ is an almost-sure contraction with constant $C$. Then $\D T$ exists $\gamma$-a.e. and $|\D T|_{\mathcal L(H,\R^d)}\le C$ $\gamma$-a.e.
\\

\item Let $T:\Omega\to \R^d$ be such that there exists $q>1$ so that $\EE_{\gamma}[|T|^q]<\infty$ and $\D T$ exists $\gamma$-a.e. If $|\D T|_{\mathcal L(H,\R^d)}\le C$ $\gamma$-a.e. then there exists an almost-sure contraction $\tilde T:\Omega\to \R^d$ with constant $C$ such that $\gamma$-a.e. $\tilde T=T$.
\end{itemize}

\end{lemma}

\begin{proof}
The first part will follow from  \cite[Theorem 5.11.2(ii)]{bogachev1998gaussian} while the second part will follow from  \cite[Theorem 5.11.7]{bogachev1998gaussian} once we check that these results can be applied. We take the domain to be $\Omega$ ( a locally convex space) with the measure $\gamma$ (a centered Radon Gaussian measure). The space $\HH$ is the Cameron-Martin space while the image of $T$ is a subset of $\R^d$ (a separable Banach space  with the Radon-Nikodym property). It remains to check that the G\^ateaux derivative of $T$ along $\HH$ is equal to $\mathcal DT$. For smooth cylindrical maps $T$ \cite[p. 207]{bogachev1998gaussian} this is clear and the general result follows from \cite[Theorem 5.7.2]{bogachev1998gaussian}.
\end{proof}

\subsection*{The F\"ollmer process and the Brownian transport map} The history of the \emph{F\"ollmer process} goes back to the work of E. Schr\"odinger in 1932 \cite{leonard2013survey}, but it was H. F\"ollmer who formulated the problem in the language of stochastic differential equations \cite{follmer1988random}; see also the work of Dai Pra from the stochastic control approach \cite{dai1991stochastic}. Let $p=f d\gamma_d$ be our probability measure on $\R^d$ and let $(B_t)_{t\in [0,1]}$ be the standard Brownian motion in $\R^d$. The \emph{F\"ollmer drift} $v(t,x)$ is such that the solution $(X_t)_{t\in [0,1]}$ of the stochastic differential equation
 \begin{align}
 \label{eq:sde}
dX_t=v(t,X_t)dt+dB_t,\quad X_0=0,
\end{align}
satisfies $X_1\sim p$ and, in addition, $\int_0^1\EE_{\gamma}\left[|v(t,X_t)|^2\right]dt$ is minimal among all such drifts. It turns out that the F\"ollmer drift $v$ has an explicit solution: Let $(P_t)_{t\ge 0}$ be the heat semigroup on $\R^d$ acting on functions $\eta:\R^d\to \R$ by
\[
P_t\eta(x)=\int \eta(x+\sqrt{t}z)d\gamma_d(z),
\]
then, the  F\"ollmer  drift $v:[0,1]\times \R^d\to\R^d$ is given by
\[
v(t,x):=\nabla\log P_{1-t}f(x).
\]
That $X_1\sim p$ with the above $v$ can be seen, for example, from the Fokker-Planck equation of \eqref{eq:sde}. Further, as a consequence of Girsanov's theorem, the optimal drift satisfies
\begin{align}
\label{eq:entprop}
\h(p|\gamma_d)=\frac{1}{2}\int_0^1\EE_{\gamma}\left[|v(t,X_t)|^2\right]dt.
\end{align} 
We refer to \cite{follmer1988random, lehec2013representation} for more details. Specifically, the validity of \eqref{eq:entprop} is guaranteed in our setting by \cite[Theorem 3.1]{dai1991stochastic} (using the uniqueness of the solution to \eqref{eq:sde}).

The \emph{Brownian transport map} is defined as the map $X_1:\Omega\to\R^d$. This definition makes sense only if \eqref{eq:sde} has a strong solution which in particular is defined at $t=1$; we will address this issue in the next section. 

\section{Almost-sure contraction properties of Brownian transport maps}
\label{sec:causal}
In this section we show that the Brownian transport map is an almost-sure contraction in various settings. The following is the main result of this section and it covers the almost-sure contraction statements of Theorem \ref{thm:uniinformal} and Theorem \ref{thm:Gaussmix}. 
\begin{theorem}
\label{thm:causalmain}
$~$
\begin{enumerate}
\item
Suppose that either $p$ is $\cc$-log-concave for some $\cc> 0$, or that $p$ is $\cc$-log-concave for some $\cc\in \R$ and that $S<+\infty$. Then \eqref{eq:sde} has a unique strong solution for all $t\in [0,1]$. Furthermore,\\
%

\begin{enumerate}[(a)]
\item If \rev{$\cc S^2\ge 4$} then $X_1$ is an almost-sure contraction with constant $\frac{1}{\sqrt{\cc}}$; equivalently,
\[
|\D X_1|_{\mathcal L(H,\R^d)}^2\le \frac{1}{\cc} \quad\gamma\textnormal{-a.e.}
\]
\item If \rev{$\cc S^2<4$} then $X_1$ is an almost-sure contraction with constant \rev{$\left(\frac{e^{2-\frac{1}{2}\cc S^2}+1}{8}\right)^{1/2}S$}; equivalently,
\rev{\[
|\D X_1|_{\mathcal L(H,\R^d)}^2\le \left(\frac{e^{2-\frac{\cc}{2} S^2}+1}{8}\right)S^2  \quad\gamma\textnormal{-a.e.}
\]}
\end{enumerate}
\item Fix a probability measure $\nu$ on $\R^d$ \rev{with $R:=\diam(\supp(\nu))$,} and let $p:=\gamma_d\star \nu$. Then \eqref{eq:sde} has a unique strong solution for all $t\in [0,1]$. Furthermore, $X_1$ is an almost-sure contraction with constant \rev{$\left(e^{\frac{1}{2}R^2}-1\right)^{1/2}\frac{\sqrt{2}}{R}$}; equivalently,
\rev{\[
|\D X_1|_{\mathcal L(H,\R^d)}^2\le \frac{2\left(e^{\frac{1}{2}R^2}-1\right)}{R^2}  \quad\gamma\textnormal{-a.e.}
\]}
\end{enumerate}
\end{theorem}

\begin{remark}
The dichotomy  of \rev{$\cc S^2\ge 4$ versus $\cc S^2<4$} is just a convenient way of organizing the various cases we consider, i.e., $\cc$ nonpositive or nonnegative and $S$ finite or infinite. This dichotomy is ambiguous when $\cc=0$ and $S=\infty$ since we need to make a convention regarding $0\cdot \infty$. Either way, the bound provided by Theorem \ref{thm:causalmain}(1) is trivial, since it is equal to $\infty$, so when proving Theorem \ref{thm:causalmain}(1) we will ignore issues arising from this case. We'll come back to the case $\cc=0$ when proving  Theorem \ref{thm:KLS}. 
\end{remark}

The proof of Theorem \ref{thm:causalmain}, ignoring for now the issue of existence of solutions to  \eqref{eq:sde}, relies on the fact that the Malliavin derivative of the F\"ollmer process satisfies the following linear equation:
\[
D_rX_t=\Id_d+\int_r^t\nabla v(s,X_s)D_rX_sds\quad \forall \,r\le t\quad\text{and}\quad D_rX_t=0\quad\forall\, r> t.
\] 
Using this equation we show that
\[
|\D X_t|_{\mathcal L(H,\R^d)}^2\le \int_0^t e^{2\int_s^t\lambda_{\max}(\nabla v(r,X_r))dr} ds\quad\forall\, t\in[0,1]\quad\gamma\text{-a.e.}
\]
Hence, the proof of Theorem \ref{thm:causalmain} now boils down to estimating $\lambda_{\max}(\nabla v(r,X_r))$. In section \ref{subsec:cov} we express $\nabla v(r,X_r)$ as a covariance matrix which allows us to bound $\lambda_{\max}(\nabla v(r,X_r))$. In section \ref{subsec:derv_Follmer} we use those estimates to establish the existence and uniqueness of a strong solution to \eqref{eq:sde}. Consequently, we derive a differential equation for $\D X_t$, which together with the estimates on $\lambda_{\max}(\nabla v(r,X_r))$, allow us to bound $|\D X_t|_{\mathcal L(H,\R^d)}^2$. We complete the proof of Theorem \ref{thm:causalmain} in section \ref{subsec:proof_causalmain}.

\begin{remark}
As explained above, the key point behind the proof of Theorem \ref{thm:causalmain} is to upper bound $
\lambda_{\max}(\nabla v(r,X_r))=\lambda_{\max}(\nabla^2\log P_{1-r}(X_r))$. However, once a Hessian estimate on $\nabla^2\log P_{1-r}(X_r)$ is obtained, it can be used to prove functional inequalities without the usage of the Brownian transport map:

(a) The first way to do so is to work with the semigroup of $(X_t)$, and mimic the classical Bakry-\'Emery calculation (see \cite{bauerschmidt2023stochastic}). The downside of this approach is that it is well suited to functional inequalities such as the log-Sobolev inequality, but not to isoperimetric-type inequalities. In contrast, transport approaches, such as the Brownian transport map, can provide all of these functional inequalities in one streamlined framework. 

(b) The second way to apply the Hessian estimate  is to use it within the context of the heat flow transport map of Kim and Milman \cite{kim2012generalization}. This approach avoids the issues mentioned in part (a). On the other hand,  the usage of this transport map is only suitable if we want to prove \emph{pointwise} estimates on the Lipschitz constant of the transport map. In contrast, the Brownian transport map allows us to prove estimates on the Lipschitz constant of the transport map in \emph{expectation}, which is what is needed to make the connection with the  Kannan--Lov\'asz--Simonovits conjecture; cf. Theorem \ref{thm:klsintro}. (We remark however that the heat flow map has its own advantages, as explained in \cite[p.3]{MR4651223}.)
\end{remark}

\subsection{Covariance estimates}
\label{subsec:cov}

We begin by representing $\nabla v$ as a covariance matrix. Define the measure $p^{x,t}$ on $\R^d$, for fixed $t\in [0,1]$ and $x\in \R^d$, by
\begin{align}
\label{eq:pxt}
dp^{x,t}(y):= \frac{f(y)\varphi^{x,t}(y)}{P_tf(x)}dy
\end{align}
where $\varphi^{x,t}$ is the density of the $d$-dimensional Gaussian distribution with mean $x$ and covariance $t\Id_d$. 
\begin{claim*}
\begin{align}
\label{eq:vcov}
\nabla v(t,x)=\frac{1}{(1-t)^2}\Cov (p^{x,1-t})-\frac{1}{1-t}\Id_d\quad  \forall\, t\in [0,1]
\end{align}
and
\begin{align}
\label{eq:nablavlower}
-\frac{1}{1-t}\Id_d	\preceq \nabla v(t,x)\quad \forall\, t\in [0,1].
\end{align}
\end{claim*}

\begin{proof}
The estimate \eqref{eq:nablavlower} follows immediately from \eqref{eq:vcov}. To prove \eqref{eq:vcov} note that since
\[
P_{1-t}f(x)=\int f(y)\varphi^{x,1-t}(y)dy
\]
we have
\begin{align*}
&\nabla P_{1-t}f(x)=\frac{1}{1-t}\int (y-x)f(y)\varphi^{x,1-t}(y)dy,\\
&\nabla^2P_{1-t}f(x)=\frac{1}{(1-t)^2}\int (y-x)^{\otimes 2}f(y)\varphi^{x,1-t}(y)dy-\frac{1}{1-t}\left(\int f(y)\varphi^{x,1-t}(y)dy\right) \Id_d,
\end{align*}
and hence,
\begin{align*}
\nabla v(t,x)&=\nabla^2\log P_{1-t}f(x)=\frac{\nabla^2P_{1-t}f(x)}{P_{1-t}f(x)}-\left(\frac{\nabla P_{1-t}f(x)}{P_{1-t}f(x)}\right)^{\otimes 2}\\
&=\frac{1}{(1-t)^2}\int (y-x)^{\otimes 2}dp^{x,1-t}(y)-\frac{1}{(1-t)^2}\left(\int (y-x)dp^{x,1-t}(y)\right)^{\otimes 2}-\frac{1}{1-t} \Id_d\\
&=\frac{1}{(1-t)^2}\int y^{\otimes 2}dp^{x,1-t}(y)-\frac{1}{(1-t)^2}\left(\int ydp^{x,1-t}(y)\right)^{\otimes 2}-\frac{1}{1-t} \Id_d\\
&=\frac{1}{(1-t)^2}\Cov (p^{x,1-t})-\frac{1}{1-t}\Id_d. 
\end{align*}
\end{proof}

We start by using the representation \eqref{eq:vcov} to upper bound $\nabla v$. 
\begin{lemma} 
\label{lem:nablav}
Define the measure $dp=fd\gamma_d$ and let $S:=\diam (\supp(p))$. Then,
\begin{enumerate}
\item For every $t\in [0,1]$,
\[
\rev{\nabla v(t,x) 	\preceq \left(\frac{S^2}{4(1-t)^2}-\frac{1}{1-t}\right)\Id_d}.
\]

\item Let $\cc\in \R$ and suppose that $p$ is $\cc$-log-concave. Then, for any $t\in \left[\frac{\cc}{\cc-1}1_{\cc<0},1\right]$,
\[
\rev{\nabla v(t,x)\preceq  \frac{1-\cc}{\cc(1-t)+t}\Id_d}.
\]
\\
%

\item Fix a probability measure $\nu$ on $\R^d$ \rev{with $R :=\diam(\supp(\nu))$,} and let $p:=\gamma_d\star \nu$. Then,
\[
\rev{\nabla v(t,x)\preceq \frac{1}{4}R^2\Id_d}.
\]
\end{enumerate}
\end{lemma}

\begin{proof} $~$

\begin{enumerate}
\item 
By  \eqref {eq:vcov}, it suffices to show that \rev{$\Cov (p^{x,1-t})\preceq \frac{1}{4}S^2\Id_d$} which \rev{follows from the definition of $p^{x,1-t}$ and from Popoviciu's inequality}.\\

\item  If $p$ is $\cc$-log-concave then, for any $t\in [0,1)$, $p^{x,1-t}$ is $\left(\cc+\frac{t}{1-t}\right)$-log-concave because
\[
-\nabla^2\log \left(\frac{dp^{x,1-t}}{dy}\right)(y)=-\nabla^2\log \left(f(y)\varphi^{0,1}(y)\right)-\nabla^2\log\left(\frac{\varphi^{x,1-t}(y)}{\varphi^{0,1}(y)}\right)\succeq \cc \Id_d+\frac{t}{1-t}\Id_d
\]
where we used that $dp(y)=f(y)\varphi^{0,1}(y)dy$. 
If $t\in \left[\frac{\cc}{\cc-1}1_{\cc<0},1\right]$, then $\left(\cc+\frac{t}{1-t}\right)\ge 0$ so by the Brascamp-Lieb inequality \cite[Theorem 4.9.1]{bakry2013analysis}, applied to functions of the form $\R^d\ni x\mapsto \langle x,v\rangle$ for $v\in S^{d-1}$, we get
\[
\Cov(p^{x,1-t})\preceq \left(\cc+\frac{t}{1-t}\right)^{-1}\Id_d
\]
and the result follows by \eqref{eq:vcov}. \\

\item We have
\begin{align*}
&\frac{dp^{x,1-t}(y)}{dy}=\frac{(\gamma_d\star \nu)(y)}{\varphi^{0,1}(y)}\frac{\varphi^{x,1-t}(y)}{P_{1-t}\left(\frac{\gamma_d\star \nu}{\varphi^{0,1}}\right)(x)}=A_{x,t}\int\varphi^{z,1}(y) \varphi^{\frac{x}{t},\frac{1-t}{t}}(y)d \nu(z)
\end{align*}
for some constant $A_{x,t}$ depending only on $x$ and $t$.
Hence, 
\[
\frac{dp^{x,1-t}(y)}{dy}=\int \varphi^{(1-t)z+x,1-t}(y)\, d\tilde{\nu}(z)
\]
where $\tilde{\nu}$ is a probability measure which is a multiple of $\nu$ by a positive function. In particular, \rev{$\supp(\tilde{\nu}) = \supp(\nu)$}. Let $G$ be a standard Gaussian vector in $\R^d$ and $Z\sim \tilde{\nu}$ be independent. Then
\[
\sqrt{1-t}G+x+(1-t)Z\sim p^{x,1-t}
\]
so, by \rev{applying Popoviciu's inequality, this time for $Z$,}
\rev{\[
\Cov(p^{x,1-t})=(1-t)\Id_d+(1-t)^2\Cov(Z)\preceq (1-t)[1+(1-t)R^2/4]\Id_d. 
\]}
By \eqref{eq:vcov},
\rev{\[
\nabla v(t,x)\preceq \frac{1+(1-t)R^2/4}{1-t}\Id_d-\frac{1}{1-t}\Id_d=\frac{1}{4}R^2\Id_d.
\]}
%
\end{enumerate}
\end{proof}

\begin{remark}
\label{rem:rhoSBL}
In principle, we could use more refined Brascamp-Lieb inequalities  \cite[Theorem 3.3]{kolesnikov2016riemannian}, or use the results of \cite{courtade2020stability} (which imply a stronger Poincar\'e inequality; we omit the details of this implication), to improve Lemma \ref{lem:nablav}(2) and the subsequent results. However, the improvement will end up being not too significant at a cost of much more tedious computations so we omit the details. 
\end{remark}

The majority of this section focuses on part (1) of Theorem \ref{thm:causalmain} since once that part is settled, part (2) will follow easily. The next two corollaries combine the bounds of Lemma \ref{lem:nablav}(1,2) to obtain a bound on $\lambda_{\max}(\nabla v(t,x))$, as well as its exponential, which is needed to bound $\mathcal DX_t$. The first corollary handles the case $\cc \ge 0$ with no assumptions on $S$ while the second corollary handles the case $\cc<0$ under the assumption $S<\infty$. 

\begin{corollary}
\label{cor:betapos}
Define the measure $dp=fd\gamma_d$ with $S:=\diam (\supp(p))$ and suppose that $p$ is $\cc$-log-concave with $\cc\in [0,+\infty)$. 
\begin{itemize}
\item If \rev{$\cc S^2\ge 4$} then
\[
\lambda_{\max}(\nabla v(t,x))\le \theta_t:=\frac{1-\cc}{(1-\cc)t+\cc},\quad t\in [0,1]
\]
and
\begin{align*}
\int_0^te^{2\int_s^t\theta_rdr}ds=\frac{t((1-\cc)t+\cc)}{\cc},\quad t\in [0,1].
\end{align*}
\item If \rev{$\cc S^2<4$} then
\rev{\[
\lambda_{\max}(\nabla v(t,x))\le\theta_t:=
\begin{cases}
 \frac{t+S^2/4-1}{(1-t)^2}\quad \text{for } t\in \left[0, \frac{1-\cc S^2/4}{(1-\cc)S^2/4+1}\right],\\
\frac{1-\cc}{(1-\cc)t+\cc} \quad \text{for } t\in \left[\frac{1-\cc S^2/4}{(1-\cc)S^2/4+1},1\right],
\end{cases}
\]}
and, for \rev{$t\in\left[\frac{1-\cc S^2/4}{(1-\cc)S^2/4+1},1\right]$},\footnote{In order to simplify the computations this bound is proven only for $t$ large enough, which is all we will end up needing, rather than for all $t\in[0,1)$.}
\rev{\begin{align*}
\int_0^te^{2\int_s^t\theta_rdr}ds=\frac{1}{8S^2}((1-\cc)S^2t+\cc S^2)^2\left\{e^{2(1-\frac{\cc }{4}S^2)}-1\right\}+((1-\cc)t+\cc)\left(\cc \frac{S^2}{4}+t+(1-\cc)t\frac{S^2}{4}-1\right).
\end{align*}}
\end{itemize}
\end{corollary}

\begin{proof}
By Lemma \ref{lem:nablav}, the two upper bounds we can get on $\lambda_{\max}(\nabla v(t,x))$ are \rev{$\frac{t+S^2/4-1}{(1-t)^2}$} and $\frac{1-\cc}{\cc(1-t)+t}$. Simple algebra shows that 
\rev{\[
\frac{t+S^2/4-1}{(1-t)^2}\le \frac{1-\cc}{\cc(1-t)+t} \quad\text{if and only if}\quad  \left(\frac{S^2}{4}-\cc \frac{S^2}{4}+1\right) t\le 1-\frac{\cc S^2}{4}.
\]}
We consider two cases. 
%
%

\begin{itemize}
\item \rev{$\cc S^2\ge 4$}: By considering \rev{$\cc S^2=4$} we see that the bound \rev{$\left(\frac{S^2}{4}-\frac{\cc S^2}{4}+1\right) t\le 1-\frac{\cc S^2}{4}$} cannot hold so it is always advantageous to use the bound 
\[
\lambda_{\max}(\nabla v(t,x))\le \theta_t:=\frac{1-\cc }{\cc(1-t)+t}=\frac{1-\cc}{(1-\cc)t+\cc}.
\]
Next we will compute $\int_0^te^{2\int_s^t\theta_rdr}ds$ and we first check that the integral $\int_s^t\theta_rdr$ is well-defined. The only issue is if $(1-\cc)t+\cc=0$ which happens when $t_0:=\frac{\cc}{\cc-1}$. If $\cc\in (0,1)$ then $t_0<0$ so $\theta_t$ is integrable on $[0,1]$, and if $\cc \ge 1$, then $t_0> 1$ so again $\theta_t$ is integrable on $[0,1]$. The only issue is when $\cc=0$ in which case $t_0=0$. However, in that case we cannot have \rev{$\cc S^2\ge 4$} as $\cc=0$. Compute,
\begin{align*}
\int_s^t\theta_rdr=(1-\cc)\left\{\frac{1}{1-\cc}\log((1-\cc)r+\cc)\right\}\bigg|_s^t=\log\left(\frac{(1-\cc)t+\cc}{(1-\cc)s+\cc}\right)
\end{align*}
so
\begin{align*}
\int_0^te^{2\int_s^t\theta_rdr}ds&=((1-\cc)t+\cc)^2\int_0^t\frac{1}{((1-\cc)s+\cc)^2}ds\\
&=-\frac{((1-\cc)t+\cc)^2}{1-\cc}\left\{\frac{1}{(1-\cc)t+\cc}-\frac{1}{\cc}\right\}=\frac{t((1-\cc)t+\cc)}{\cc}.
\end{align*}

\item \rev{$\cc S^2<4$}: The condition \rev{$\left(\frac{S^2}{4}-\frac{\cc S^2}{4}+1\right) t\le 1-\frac{\cc S^2}{4}$} is equivalent to 
\rev{\[
t\le \frac{1-\cc S^2/4}{(1-\cc) S^2/4+1}
\]}
since the denominator is nonnegative as \rev{$\cc S^2<4$}. Hence we define
\rev{\[
\lambda_{\max}(\nabla v(t,x))\le\theta_t:=
\begin{cases}
\frac{t+S^2/4-1}{(1-t)^2}\quad \text{for } t\in \left[0, \frac{1-\cc S^2/4}{(1-\cc)S^2/4+1}\right],\\
\frac{1-\cc}{(1-\cc)t+\cc} \quad \text{for } t\in \left[\frac{1-\cc S^2/4}{(1-\cc)S^2/4+1},1\right].
\end{cases}
\]}
From now until the end of the proof we assume that \rev{$t\ge \frac{1-\cc S^2/4}{(1-\cc)S^2/4+1}$}. In order to compute $\int_s^t\theta_rdr$ we start by noting that  $\frac{r}{(1-r)^2}=\frac{d}{dr}\left[\frac{1}{1-r}+\log(1-r)\right]$. We also note that, following the discussion in the \rev{$\cc S^2\ge 4$} case, the denominator $(1-\cc)t+\cc$ does not vanish in the range where it is integrated.  For \rev{$s\in\left[0, \frac{1-\cc S^2/4}{(1-\cc)S^2/4+1}\right]$} we have
\rev{\begin{align*}
&\int_s^t\theta_rdr=\int_s^{\frac{1-\cc S^2/4}{(1-\cc)S^2/4+1}}\left(\frac{r+S^2/4-1}{(1-r)^2}\right)dr+\int_{\frac{1-\cc S^2/4}{(1-\cc)S^2/4+1}}^t\frac{1-\cc}{(1-\cc)r+\cc} dr\\
&=\left\{\frac{1}{1-r}+\log(1-r)\right\}\bigg|_s^{\frac{1-\cc S^2/4}{(1-\cc)S^2/4+1}}+\left\{\frac{S^2/4-1}{1-r}\right\}\bigg|_s^{\frac{1-\cc S^2/4}{(1-\cc)S^2/4+1}}+(1-\cc)\left\{\frac{1}{1-\cc}\log((1-\cc)r+\cc)\right\}\bigg |_{\frac{1-\cc S^2/4}{(1-\cc)S^2/4+1}}^t\\
&=\left\{\log(1-r)\right\}\bigg|_s^{\frac{1-\cc S^2/4}{(1-\cc)S^2/4+1}}+\frac{S^2}{4}\left\{\frac{1}{1-r}\right\}\bigg|_s^{\frac{1-\cc S^2/4}{(1-\cc)S^2/4+1}}+\left\{\log((1-\cc)r+\cc)\right\}\bigg |_{\frac{1-\cc S^2/4}{(1-\cc)S^2/4+1}}^t\\
&=\log\left(\frac{S^2/4}{(1-\cc)S^2/4+1}\right)-\log(1-s)+\frac{(1-\cc)S^2}{4}+1-\frac{S^2}{4(1-s)}+\log((1-\cc)t+\cc)-\log\left(\frac{1}{(1-\cc)S^2/4+1}\right)\\
&=\left\{\frac{(1-\cc)S^2}{4}+1+\log\left(\frac{(1-\cc)S^2t}{4}+ \frac{\cc S^2}{4}\right)\right\}-\left\{\log(1-s)+\frac{S^2}{4(1-s)}\right\}.
\end{align*}}
Hence,
\rev{\[
e^{2\int_s^t\theta_rdr}ds=\frac{1}{16}e^{\frac{1}{2}(1-\cc)S^2+2}((1-\cc)S^2t+\cc S^2)^2\frac{e^{-\frac{S^2}{2(1-s)}}}{(1-s)^2}.
\]}
For \rev{$s\in\left[\frac{1-\cc S^2/4}{(1-\cc)S^2/4+1},1\right]$} we have,
\begin{align*}
\int_s^t\theta_rdr=\int_s^t\frac{1-\cc}{(1-\cc)r+\cc} dr=\log\left(\frac{(1-\cc)t+\cc}{(1-\cc)s+\cc}\right)
\end{align*}
and so
\[
e^{2\int_s^t\theta_rdr}ds=\left(\frac{(1-\cc)t+\cc}{(1-\cc)s+\cc}\right)^2. 
\]
It follows that
\rev{\begin{align*}
&\int_0^te^{2\int_s^t\theta_rdr}ds\\
&= \frac{1}{16}e^{\frac{1}{2}(1-\cc)S^2+2}((1-\cc)S^2t+\cc S^2)^2\int_0^{\frac{1-\cc S^2/4}{(1-\cc)S^2/4+1}} \frac{e^{-\frac{S^2}{2(1-s)}}}{(1-s)^2}  ds+((1-\cc)t+\cc)^2\int_{\frac{1-\cc S^2/4}{(1-\cc)S^2/4+1}}^t \frac{1}{((1-\cc)s+\cc)^2}ds,
\end{align*}}
and we note that both integrals are finite because $(1-\cc)s+\cc$ does not vanish in the range of integration. The first integral reads
\rev{\begin{align*}
\int_0^{\frac{1-\cc S^2/4}{(1-\cc)S^2/4+1}} \frac{e^{-\frac{S^2}{2(1-s)}}}{(1-s)^2}  ds=\left\{-\frac{2}{S^2}e^{-\frac{S^2}{2(1-s)}}\right\}\bigg|_0^{\frac{1-\cc S^2/4}{(1-\cc)S^2/4+1}}=-\frac{2}{S^2}\left\{e^{-\frac{1}{2}(1-\cc)S^2-2}-e^{-\frac{1}{2}S^2}\right\}
\end{align*}}
so 
\rev{\[
\frac{1}{16}e^{\frac{1}{2}(1-\cc)S^2+2}((1-\cc)S^2t+\cc S^2)^2\int_0^{\frac{1-\cc S^2/4}{(1-\cc)S^2/4+1}} \frac{e^{-\frac{S^2}{2(1-s)}}}{(1-s)^2}  ds=\frac{1}{8S^2}((1-\cc)S^2t+\cc S^2)^2\left\{e^{2(1-\frac{\cc S^2}{4})}-1\right\}.
\]}
The second integral reads
\rev{\begin{align*}
\int_{\frac{1-\cc S^2/4}{(1-\cc)S^2/4+1}}^t \frac{1}{((1-\cc)s+\cc)^2}ds=-\frac{1}{1-\cc}\left\{\frac{1}{(1-\cc)s+\cc}\right\}\bigg|_{\frac{1-\cc S^2/4}{(1-\cc)S^2/4+1}}^t=\frac{-1}{1-\cc}\left\{\frac{1}{(1-\cc)t+\cc}-((1-\cc)\frac{S^2}{4}+1)\right\}
\end{align*}}
so
\rev{\begin{align*}
((1-\cc)t+\cc)^2\int_{\frac{1-\cc S^2/4}{(1-\cc)S^2/4+1}}^t \frac{1}{((1-\cc)s+\cc)^2}ds&=-\frac{(1-\cc)t+\cc}{1-\cc}\left\{1-((1-\cc)t+\cc)(\frac{(1-\cc)S^2}{4}+1)\right\}\\
&=((1-\cc)t+\cc)\left(\frac{\cc S^2}{4}+t+\frac{(1-\cc)tS^2}{4}-1\right).
\end{align*}}
%
%
%
\end{itemize}
Adding everything up gives the result. 
\end{proof}

\begin{corollary}
\label{cor:betaneg}
Define the measure $dp=fd\gamma_d$ and suppose that $S:=\diam (\supp(p))<\infty$ and that $p$ is $\cc$-log-concave with $\cc\in(-\infty, 0)$. We have
\rev{\[
\lambda_{\max}(\nabla v(t,x))\le\theta_t:=
\begin{cases}
 \frac{t+S^2/4-1}{(1-t)^2}\quad \text{for } t\in \left[0, \frac{1-\cc S^2/4}{(1-\cc)S^2/4+1}\right],\\
\frac{1-\cc}{(1-\cc)t+\cc} \quad \text{for } t\in \left[\frac{1-\cc S^2/4}{(1-\cc)S^2/4+1},1\right),
\end{cases}
\]}
and, for \rev{$t\in\left[\frac{1-\cc S^2/4}{(1-\cc)S^2/4+1},1\right]$},
\rev{\begin{align*}
\int_0^te^{2\int_s^t\theta_rdr}ds=\frac{1}{8S^2}((1-\cc)S^2t+\cc S^2)^2\left\{e^{2(1-\frac{\cc}{4} S^2)}-1\right\}+((1-\cc)t+\cc)\left(\frac{\cc S^2}{4}+t+\frac{(1-\cc)tS^2}{4}-1\right).
\end{align*}}
\end{corollary}

\begin{proof}
By Lemma \ref{lem:nablav}, the two upper bounds we can get on $\lambda_{\max}(\nabla v(t,x))$ are \rev{$\frac{t+S^2/4-1}{(1-t)^2}$}  for any $t\in [0,1]$ and $\frac{1-\cc}{\cc(1-t)+t}$ for $t\in \left[\frac{\cc}{\cc-1},1\right]$. Hence, for $t\in [0,\frac{\cc}{\cc-1})$, we must use the bound \rev{$\frac{t+S^2/4-1}{(1-t)^2}$}. Next we note that \rev{$0< \frac{\cc}{\cc-1}<\frac{1-\cc S^2/4}{(1-\cc)S^2/4+1}\le 1$} and that (using $\cc(1-t)+t\ge 0$ for $t\ge \frac{\cc}{\cc-1}$),
\rev{\[
\frac{t+S^2/4-1}{(1-t)^2}\le \frac{1-\cc}{\cc(1-t)+t} \quad\text{for }  t\in\left[\frac{\cc}{\cc-1},\frac{1-\cc S^2/4}{(1-\cc)S^2/4+1}\right]. 
\]}
We define
\rev{\[
\lambda_{\max}(\nabla v(t,x))\le \theta_t:=
\begin{cases}
\frac{t+S^2/4-1}{(1-t)^2}\quad \text{for } t\in \left[0, \frac{1-\cc S^2/4}{(1-\cc)S^2/4+1}\right],\\
\frac{1-\cc}{(1-\cc)t+\cc} \quad \text{for } t\in \left[\frac{1-\cc S^2/4}{(1-\cc)S^2/4+1},1\right].
\end{cases}
\]}
As in the proof of Corollary \ref{cor:betapos}, we have
\rev{\[
e^{2\int_s^t\theta_rdr}=
\begin{cases}
e^{\frac{1}{2}(1-\cc)S^2+2}((1-\cc)S^2t+\cc S^2)^2\frac{e^{-\frac{S^2}{2(1-s)}}}{16(1-s)^2}, & s\in\left[0, \frac{1-\cc S^2/4}{(1-\cc)S^2/4+1}\right]\\
\left(\frac{(1-\cc)t+\cc}{(1-\cc)s+\cc}\right)^2,& s\in\left[\frac{1-\cc S^2/4}{(1-\cc)S^2/4+1},1\right].
\end{cases}
\]}
Since \rev{$\frac{\cc}{\cc-1}<\frac{1-\cc S^2/4}{(1-\cc)S^2/4+1}$}, the above term can be integrated as in the proof of Corollary \ref{cor:betapos}. 
\end{proof}

\subsection{The Malliavin derivative of the F\"ollmer process}
\label{subsec:derv_Follmer}

The bounds provided by \eqref{eq:nablavlower} and Lemma \ref{lem:nablav} are only strong enough to establish the existence of a unique strong solution to \eqref{eq:sde} only until $t<1$ because at $t=1$ these bounds can blow up. For our purposes, however, it is crucial to have the solution well-defined at $t=1$ since we need $X_1\sim p$. We will proceed by first analyzing the behavior of the solution before time 1, which will then allow us to extend the solution and its Malliavin derivative to $t=1$; see Proposition \ref{prop:existsde}.

\begin{lemma}
\label{lem:existsde}
Let $dp=fd\gamma_d$ with $S:=\diam (\supp(p))$ and suppose that either $S<\infty$ or $p$ is $\cc$-log-concave with $\cc\ge 0$. The equation \eqref{eq:sde} has a unique strong solution $(X_t)$ for $t\in [0,1)$ satisfying
\[
D_rX_t=\Id_d+\int_r^t\nabla v(s,X_s)D_rX_sds\quad \forall \,r\le t\quad\text{and}\quad D_rX_t=0\quad\forall\, r> t, \quad \gamma\textnormal{-a.e.}
\] 
In addition,
\[
|\D X_t|_{\mathcal L(H,\R^d)}^2\le \int_0^t e^{2\int_s^t\lambda_{\max}(\nabla v(r,X_r))dr} ds,\quad\forall\, t\in[0,1)\quad\gamma\textnormal{-a.e.}
\]
%
%
%
\end{lemma}

\begin{proof}
Fix $T\in (0,1)$ and note that by \eqref{eq:nablavlower} and Lemma \ref{lem:nablav}, $v:[0,T]\times\R^d\to \R^d$ is uniformly Lipschitz in $x$ (with the Lipschitz constant depending on $T$). Writing $v(t,x)=v(t,0)+\int_0^1\nabla v(t,rx)x dr$ we see, again by \eqref{eq:nablavlower} and Lemma \ref{lem:nablav}, that $v$ is of linear growth. It is standard that under these conditions the equation \eqref{eq:sde} has a unique strong solution \cite[Lemma 2.2.1]{nualart2006malliavin}. The Malliavin differentiability of $(X_t)_{t\in [0,1)}$ and the formula for its derivative  follow from \cite[Theorem 2.2.1 and p. 121]{nualart2006malliavin}, as we now elaborate. According to \cite[Theorem 2.2.1]{nualart2006malliavin}, if $(X_t)$ is a solution to a stochastic differential equation
\[
dX_t=dB_t+\int_0^tb(s,X_s)ds,
\]
with $b$ globally Lipschitz and of linear growth, then 
\rev{\[
D_rX_t=\Id_d + \int_r^t\bar{b}(s,X_s)D_rX_sds\quad \forall \,r\le t\quad\text{and}\quad D_rX_t=0\quad\forall\, r> t, \quad \gamma\textnormal{-a.e.},
\]}
where $\bar b=\nabla b$ (see \cite[p. 121]{nualart2006malliavin}). As mentioned above, $b:=v$ is indeed globally Lipschitz and of linear growth, which implies the result.

Turning to the bound on $\mathcal DX_t$, fix $\dot{h}\in H$ and define $\alpha_{\dot{h}}:[0,1)\to \R^d$ by 
\[
\alpha_{\dot{h}}(t):=\D X_t[\dot{h}]=\int_0^tD_rX_t[\dot{h}_r]dr. 
\]
The equation for $DX_t$ and Fubini's theorem (which can be applied since $\nabla v$ is bounded and by using Gr\"onwall's inequality on any norm of $D_rX_t$) imply that
\begin{align*}
\alpha_{\dot{h}}(t)&= \int_0^t\Id_d[\dot{h}_r]dr+\int_0^t\int_r^t\nabla v(s,X_s)D_rX_s[\dot{h}_r]dsdr\\
&=h_t+\int_0^t\int_0^t\nabla v(s,X_s)D_rX_s[\dot{h}_r]dsdr\quad\text{(because $D_rX_s=0$ for $s<r$)}\\
&=h_t+\int_0^t\nabla v(s,X_s)\left(\int_0^tD_rX_s[\dot{h}_r]dr\right)ds\\
&=h_t+\int_0^t\nabla v(s,X_s)\left(\int_0^sD_rX_s[\dot{h}_r]dr\right)ds\quad\text{(because $D_rX_s=0$ for $s<r$)}
\\
&=h_t+\int_0^t\nabla v(s,X_s) \alpha_{\dot{h}}(s)ds.
\end{align*}
Hence
\[
\partial_t\alpha_{\dot{h}}(t)=\dot{h}_t+\nabla v(t,X_t) \alpha_{\dot{h}}(t) \quad\forall\, t\in [0,1). 
\]
Set $\lambda_t:=\lambda_{\max}(\nabla v(t,X_t))$. It follows from the Cauchy-Schwarz inequality that
\begin{align*}
\partial_t|\alpha_{\dot{h}}(t)|^2&=2\langle \partial_t\alpha_{\dot{h}}(t),\alpha_{\dot{h}}(t)\rangle=2\langle \dot{h}_t,\alpha_{\dot{h}}(t)\rangle+2\langle\alpha_{\dot{h}}(t), \nabla v(t,X_t)\alpha_{\dot{h}}(t)\rangle\\
&\le 2|\dot{h}_t|\sqrt{|\alpha_{\dot{h}}(t)|^2}+2\lambda_t|\alpha_{\dot{h}}(t)|^2
\end{align*}
so defining $y:[0,1)\to \R$ by $y(t):=|\alpha_{\dot{h}}(t)|^2$ we find
\[
\partial_ty(t)\le 2|\dot{h}_t|\sqrt{y(t)}+2\lambda_ty(t).
\]
In order to analyze $y(t)$ we note that the solution of the Bernoulli ordinary differential equation
\[
\partial_tz(t)=2|\dot{h}_t|\sqrt{z(t)}+2\lambda_tz(t), \quad z(0)=0
\]
can be verified to be
\[
z(t)=\left(e^{\int_0^t\lambda_s ds}\int_0^t e^{-\int_0^s\lambda_rdr}|\dot{h}_s|ds\right)^2.
\]
By the Cauchy-Schwarz inequality, 
\[
z(t)\le e^{2\int_0^t\lambda_s ds}\int_0^t e^{-2\int_0^s\lambda_rdr}ds\int_0^t|\dot{h}_s|^2ds= \int_0^t e^{2\int_s^t\lambda_rdr}ds\int_0^t|\dot{h}_s|^2ds
\]
so since $y(t)\le z(t)$ for all $t\in [0,1)$ we conclude that
\begin{align*}
|\D X_t|_{\mathcal L(H,\R^d)}^2=\sup_{\dot{h}\in H:|\dot{h}|_H=1}|\alpha_{\dot{h}}(t)|^2\le\int_0^t  e^{2\int_s^t\lambda_rdr}ds.
\end{align*}
\end{proof}

Combining Lemma \ref{lem:existsde} and Corollaries  \ref{cor:betapos}, \ref{cor:betaneg} we obtain:

\begin{corollary}
\label{cor:boundDX}
Let $dp=fd\gamma_d$ with $S:=\diam (\supp(p))$ and suppose that $p$ is $\cc$-log-concave for some $\cc\in \R$. Then, $\gamma$-a.e.,
\begin{enumerate}[(a)]
\item Suppose $\cc \ge 0$. 
\begin{itemize}
\item If \rev{$\cc S^2\ge 4$}: 
\[
|\D X_t|_{\mathcal L(H,\R^d)}^2\le \frac{t((1-\cc)t+\cc)}{\cc},\quad t\in [0,1).
\]
\item If \rev{$\cc S^2<4$}: For \rev{$t\in\left[\frac{1-\cc S^2/4}{(1-\cc)S^2/4+1},1\right)$}, 
\rev{\[
|\D X_t|_{\mathcal L(H,\R^d)}^2\le\frac{1}{8S^2}((1-\cc)S^2t+\cc S^2)^2\left\{e^{2(1-\frac{\cc}{4} S^2)}-1\right\}+((1-\cc)t+\cc)\left(\frac{\cc S^2}{4}+t+\frac{(1-\cc)tS^2}{4}-1\right).
\]}
\end{itemize}

\item Suppose $\cc\le 0$ and that $S<+\infty$. For \rev{$t\in\left[\frac{1-\cc S^2/4}{(1-\cc)S^2/4+1},1\right)$}, 
\rev{\[
|\D X_t|_{\mathcal L(H,\R^d)}^2\le\frac{1}{8S^2}((1-\cc)S^2t+\cc S^2)^2\left\{e^{2(1-\frac{\cc}{4} S^2)}-1\right\}+((1-\cc)t+\cc)\left(\frac{\cc S^2}{4}+t+\frac{(1-\cc)tS^2}{4}-1\right).
\]}
\end{enumerate}
\end{corollary}

We will now extend the solution $X$ and its Malliavin derivatives to $t=1$. 

\begin{proposition}
\label{prop:existsde}
Let $dp=fd\gamma_d$ with $S:=\diam (\supp(p))$  and suppose that either $p$ is $\cc$-log-concave for some $\cc> 0$, or that $p$ is $\cc$-log-concave for some $\cc\in \R$ and that $S<+\infty$. The equation \eqref{eq:sde} has a unique strong solution $(X_t)$ for $t\in [0,1]$ satisfying, $\gamma$-a.e.,
\[
\forall\, r\le t, ~~D_rX_t=\Id_d+\int_r^t\nabla v(s,X_s)D_rX_sds\quad\text{and}\quad D_rX_t=0~\forall\, r> t
\] 
and
\[
|\D X_t|_{\mathcal L(H,\R^d)}^2\le \int_0^t e^{2\int_s^t\lambda_{\max}(\nabla v(r,X_r))dr} ds\quad\forall\, t\in[0,1].
\]
In addition: 
\begin{enumerate}[(a)]
\item Suppose $\cc\ge 0$. 
\begin{itemize}
\item If \rev{$\cc S^2\ge 4$}: 
\[
|\D X_t|_{\mathcal L(H,\R^d)}^2\le \frac{t((1-\cc)t+\cc)}{\cc},\quad t\in [0,1].
\]
\item If \rev{$\cc S^2<4$}: For \rev{$t\in\left[\frac{1-\cc S^2/4}{(1-\cc)S^2/4+1},1\right]$}, 
\rev{\[
|\D X_t|_{\mathcal L(H,\R^d)}^2\le\frac{1}{8S^2}((1-\cc)S^2t+\cc S^2)^2\left\{e^{2(1-\frac{\cc}{4} S^2)}-1\right\}+((1-\cc)t+\cc)\left(\frac{\cc S^2}{4}+t+\frac{(1-\cc)tS^2}{4}-1\right).
\]}
\end{itemize}

\item Suppose $\cc\le 0$ and $S<+\infty$. For \rev{$t\in\left[\frac{1-\cc S^2/4}{(1-\cc)S^2/4+1},1\right]$}, 
\rev{\[
|\D X_t|_{\mathcal L(H,\R^d)}^2\le\frac{1}{8S^2}((1-\cc)S^2t+\cc S^2)^2\left\{e^{2(1-\frac{\cc}{4} S^2)}-1\right\}+((1-\cc)t+\cc)\left(\frac{\cc S^2}{4}+t+\frac{(1-\cc)tS^2}{4}-1\right).
\]}
\end{enumerate}
\end{proposition}


\begin{proof}
We start by establishing the solution to \eqref{eq:sde} all the way to $t=1$. Let $(X_t)_{t\in[0,1)}$ be the process given by Lemma \ref{lem:existsde}. For $k\in\mathbb Z_+$ define $t_k:=1-\frac{1}{k}$ and compute, for $l\ge k$,
\begin{align*}
\EE[|X_{t_l}-X_{t_k}|^2]&\rev{\le 2\EE[|B_{t_l}-B_{t_k}|^2]+2\int_{t_k}^{t_l}\EE[|v(s,X_s)|^2]ds\le 2d(t_l-t_k)+4 (t_l-t_k)\h(p|\gamma_d)}\\
&\rev{\le \frac{2d}{k}+\frac{4}{k}\h(p|\gamma_d)}
\end{align*}
where we used \eqref{eq:entprop} and $t_l-t_k\le 1-\left(1-\frac{1}{k}\right)=\frac{1}{k}$. Given $\epsilon>0$ let $N$ be such that \rev{$\frac{2d}{N}+\frac{4}{N}\h(p|\gamma_d)<\epsilon$} (which is possible as $\h(p|\gamma_d)<\infty$) to conclude that $\EE_{\gamma}[|X_{t_l}-X_{t_k}|^2]\le \epsilon$ for any $k,l\ge N$. Hence, $\{X_{t_k}\}$ is a Cauchy sequence in $L^2(\Omega,\R^d)$ which is complete; we denote the limit by $X_1$. Repeating the above argument on the right-hand side of \eqref{eq:sde} shows that $(X_t)_{t\in [0,1]}$ solves \eqref{eq:sde} for all $t\in [0,1]$.\\

To extend the derivative to $t=1$ we start by showing that $DX_1$ exists. Fix $w\in \R^d$ and take $\dot{h}\equiv w$ so that
\[
|\D X_t[\dot{h}]|^2=\left|\int_0^tD_rX_twdr\right|^2.
\]
Taking $w=e_j$, the $j$th element of the standard basis of $\R^d$, and using that $D_r^jX_t^i=0$ if $r>t$, we have 
\[
|\D X_t[\dot{h}]|^2=\sum_{j=1}^d\left(\int_0^tD_r^jX_t^idr\right)^2=\sum_{j=1}^d\left(\int_0^1D_r^jX_t^idr\right)^2\le \sum_{j=1}^d|D^jX_t^i|_{H}^2.
\]
By Corollary \ref{cor:boundDX}, it follows that $\sup_k |D^jX_{t_k}^i|_{H}^2<\infty$ for any $i,j\in [d]$, $\gamma$-a.e. Hence, by \cite[Lemma 1.2.3]{nualart2006malliavin}, for any $i,j\in [d]$, $D^jX_1^i$ exists and $D^jX_{t_k}^i$ converges to $D^jX_1^i$ in the weak topology of $L^2(\Omega,H)$. Hence, for a fixed $\dot{h}\in H$, we have that $\EE_{\gamma}[|\D X_{t_k}[\dot{h}]-\D X_1[\dot{h}]|^2]\to 0$ as $k\to\infty$. 	In particular, for a fixed $\dot{h}\in H$, $\D X_{t_k}[\dot{h}]$ converges to $\D X_{1}[\dot{h}]$ in probability. 

On the other hand, fix $\dot{h}\in H$ and recall the definition of $\alpha_{\dot{h}}:[0,1)\to \R^d$ from the proof of Lemma \ref{lem:existsde}. The definition of $\alpha_{\dot{h}}$ as an integral, and the fact that the integrand is bounded (since $\sup_k |D^jX_{t_k}^i|_{H}^2<\infty$ and $\dot{h}_r$ is in $L^2([0,1],\R^d)$), show that, $\gamma$-a.e., $\{\alpha_{\dot{h}}(t_k)\}$ is a Cauchy sequence so it converges to some limit denoted as $\alpha_{\dot{h}}(1)$. Hence, $\gamma$-a.e,  $\D X_{t_k}[\dot{h}]$ converges to $\alpha_{\dot{h}}(1)$ for any $\dot{h}\in H$. In particular, $\D X_{t_k}[\dot{h}]$ converges to $\alpha_{\dot{h}}(1)$ in probability. 

It follows that, $\gamma$-a.e., $\alpha_{\dot{h}}(1)=\D X_{1}[\dot{h}]$. Since $\alpha_{\dot{h}}(t)=\D X_{t}[\dot{h}]$ for all $t\in [0,1)$ we conclude that, $\gamma$-a.e., for any $\dot{h}\in H$, $\D X_{t}[\dot{h}]$ converges to $\D X_{1}[\dot{h}]$. By the Banach-Steinhaus theorem, $\gamma$-a.e., the limiting operator $\D X_{1}:H\to \R^d$ is linear and continuous with $|\D X_{1}|_{\mathcal L(H,\R^d)}\le \lim\inf_{t\uparrow 1} |\D X_t|_{\mathcal L(H,\R^d)}$. The proof is complete. 
\end{proof}

\subsection{Proof of Theorem \ref{thm:causalmain}}
\label{subsec:proof_causalmain}
We start by noting that Lemma \ref{lem:rad} applies in our setting because the moment assumption holds by either convexity or the boundedness of the support (including in the Gaussian mixture case). 

Part (1): Combining the results in Proposition \ref{prop:existsde} and plugging in $t=1$ we get:
\begin{enumerate}[(a)]
\item If \rev{$\cc S^2\ge 4$}: 
\[
|\D X_1|_{\mathcal L(H,\R^d)}^2\le \frac{1}{\cc}\quad \gamma\text{-a.e.}
\]
\item If \rev{$\cc S^2<4$}: 
\rev{\[
|\D X_1|_{\mathcal L(H,\R^d)}^2\le \left(\frac{e^{2-\frac{\cc}{2} S^2}+1}{8}\right)S^2 \quad \gamma\text{-a.e.}
\]}
\end{enumerate}
This completes the proof.

Part (2): By Lemma \ref{lem:nablav}(3),
\rev{\[
\nabla v(t,x)\preceq \frac{1}{4}R^2\Id_d
\]}
so the previous arguments of this section apply to show that \eqref{eq:sde} has a unique strong solution in the setting where $p$ is a mixture of Gaussians.  In addition, the bound \rev{$\nabla v(t,x)\preceq \frac{1}{4}R^2\Id_d$} implies that
\rev{\[
\lambda_{\max}(\nabla v(t,x))\le \theta_t:=\frac{1}{4}R^2\quad \forall t\in [0,1]. 
\]}
Hence, repeating the computations earlier in this section yields
\rev{\[
|\D X_1|_{\mathcal L(H,\R^d)}^2\le\int_0^1e^{2\int_s^1\theta_r dr}ds=2\frac{e^{\frac{1}{2}R^2}-1}{R^2}.
\]}

\section{Contraction properties of Brownian transport maps for log-concave measures}
\label{sec:contrctlc}
In this section we suppose that $p$ is an isotropic log-concave measure with compact support. Our main result, Theorem \ref{thm:KLS}, bounds the norms of the derivative of the Brownian transport map (Theorem \ref{thm:klsintro}). The proof of Theorem \ref{thm:KLS} relies on the result of \cite{Klartag} and the  technique of  \cite{chen2021almost}, which is based on the stochastic localization of Eldan; see also \cite{eldan2013thin, leeVempala, KlartagLehec}. 
\subsection*{Preliminaries} 
We start by explaining the connection between stochastic localization and the F\"ollmer process. Recall that the F\"ollmer process is the solution $(X_t)_{t\in [0,1]}$ to the stochastic differential equation \eqref{eq:sde}:
 \begin{align*}
dX_t=\nabla\log P_{1-t}f(X_t)dt+dB_t,\quad X_0=0
\end{align*}
and has the property that $X_1\sim p$ where $p=fd\gamma$.  We also recall the definition \eqref{eq:pxt}:
\[
dp^{x,t}(y)= \frac{f(y)\varphi^{x,t}(y)}{P_tf(x)}dy.
\]
Let us denote by $p_t$ the (random) law of $X_1|X_t$, that is, $\int_{\R^d} \eta dp_t=\EE_{\gamma}[\eta(X_1)|X_t]$ a.s. for all $\eta:\R^d\to \R$ continuous and bounded. The next lemma establishes the connection between the F\"ollmer process and stochastic localization. The proof is well-known and we provide it for completeness.
\begin{lemma}
\label{lem:Follmersotcloc}
For $t\in [0,1)$ the random law $p_t$ has a density with respect to the Lebesgue measure, denoted by $p_t(y)$, which satisfies $p_t(y)dy=dp^{X_t,1-t}(y)$. Further, given $y\in \R^d$ the random process $(p_t(y))_{t\in [0,1)}$ satisfies the stochastic differential equation 
\begin{align}
\label{eq:ptsde}
dp_t(y)=p_t(y)\left\langle \frac{y-\int zdp_t(z)}{1-t},dB_t\right\rangle.
\end{align}
\end{lemma} 
In stochastic localization (in its simplified setting), equation \eqref{eq:ptsde}, up to time-change ($t\mapsto \frac{1}{1-t}-1$), serves as the \emph{definition} of the process. We refer \cite{eldanICM}, \cite{lee2018kannan}, and \cite[section 4]{klartag2021spectral} for more information. 

\begin{proof}[Proof of Lemma \ref{lem:Follmersotcloc}]
Let $(X_t)_{t\in [0,1]}$ be the F\"ollmer process and let $\mu$ be its associated measure on the Wiener space $\Omega$: $\frac{d\mu}{d\gamma}(\omega)=f(\omega_1)$ for $\omega\in\Omega$. Then, for any $\eta:\R^d\to \R$ continuous and bounded, we have
\begin{align*}
\EE_{\gamma}[\eta(X_1)|X_t]&=\EE_{\mu}[\eta(\omega_1)|X_t]=\frac{\EE_{\gamma}\left[\frac{d\mu}{d\gamma}(\omega_1)\eta(\omega_1)\bigg|X_t\right]}{\EE_{\gamma}\left[\frac{d\mu}{d\gamma}(\omega_1)\bigg|X_t\right]}=\frac{\EE_{\gamma}[f(\omega_1)\eta(\omega_1)|X_t]}{\EE_{\gamma}[f(\omega_1)|X_t]}=\frac{P_{1-t}(f\eta)(X_t)}{P_{1-t}f(X_t)}\\
&=\frac{1}{P_{1-t}f(X_t)}\int_{\R^d}\eta(y)f(y)\frac{\exp\left(-\frac{|y-X_t|^2}{2(1-t)}\right)}{(2\pi(1-t))^{d/2}} dy=\int_{\R^d}\eta(y)p^{X_t,1-t}(y)dy.
\end{align*}
It follows that $p_t=p^{X_t,1-t}$ with density 
\[
p_t(y)=\frac{f(y)}{P_{1-t}f(X_t)}\frac{\exp\left(-\frac{|y-X_t|^2}{2(1-t)}\right)}{(2\pi(1-t))^{d/2}} 
\]
which is well-defined for all $t\in [0,1)$. Fix $y\in \R^d$ and let
\[\alpha(t,x):=\frac{1}{P_{1-t}f(x)}\quad\text{and} \quad \beta(t,x):=\frac{\exp\left(-\frac{|y-x|^2}{2(1-t)}\right)}{(2\pi(1-t))^{d/2}}
\]
so that $p_t(y)=\alpha(t,X_t)\beta(t,X_t)$. By the heat equation,
\begin{align*}
&\partial_t\alpha(t,x)=-\frac{\partial_t P_{1-t}f(x)}{P_{1-t}f(x)^2}=\frac{1}{2}\frac{\Delta P_{1-t}f(x)}{P_{1-t}f(x)^2},\quad \partial_t\beta(t,x)=\frac{\exp\left(-\frac{|y-x|^2}{2(1-t)}\right)}{(2\pi(1-t))^{d/2}}\left\{\frac{d}{2}\frac{1}{1-t}-\frac{|y-x|^2}{2(1-t)^2}\right\},\\
&\nabla \alpha(t,x)=-\frac{\nabla P_{1-t}f(x)}{P_{1-t}f(x)^2},\quad \nabla\beta(t,x)=\frac{\exp\left(-\frac{|y-x|^2}{2(1-t)}\right)}{(2\pi(1-t))^{d/2}}\frac{y-x}{1-t},\\
&\Delta \alpha(t,x)=-\frac{\Delta P_{1-t}f(x)}{P_{1-t}f(x)^2}+2\frac{|\nabla P_{1-t}f(x)|^2}{P_{1-t}f(x)^3}, \quad \Delta \beta(t,x)=\frac{\exp\left(-\frac{|y-x|^2}{2(1-t)}\right)}{(2\pi(1-t))^{d/2}}\left\{\frac{|y-x|^2}{(1-t)^2}-\frac{d}{1-t}\right\},
\end{align*}
and hence,
\begin{align*}
&\partial_t[\alpha(t,x)\beta(t,x)]=p_t(y)\left\{\frac{1}{2}\frac{\Delta P_{1-t}f(x)}{P_{1-t}f(x)}+\frac{d}{2}\frac{1}{1-t}-\frac{|y-x|^2}{2(1-t)^2}\right\},\\
&\nabla [\alpha(t,x)\beta(t,x)]=p_t(y)\left\{-\frac{\nabla P_{1-t}f(x)}{P_{1-t}f(x)}+\frac{y-x}{1-t}\right\},\\
&\frac{1}{2}\Delta [\alpha(t,x)\beta(t,x)]= p_t(y)\left\{-\frac{1}{2}\frac{\Delta P_{1-t}f(x)}{P_{1-t}f(x)}+\frac{|\nabla P_{1-t}f(x)|^2}{P_{1-t}f(x)^2}-\left\langle \frac{\nabla P_{1-t}f(x)}{P_{1-t}f(x)},\frac{y-x}{1-t}\right\rangle+\frac{|y-x|^2}{2(1-t)^2}-\frac{d}{2(1-t)}\right\}.
\end{align*}
It follows from It\^o's formula that
\begin{align*}
d[\alpha(t,X_t)\beta(t,X_t)]&= p_t(y)\left\{\frac{|\nabla P_{1-t}f(X_t)|^2}{P_{1-t}f(X_t)^2}-\left\langle \frac{\nabla P_{1-t}f(X_t)}{P_{1-t}f(X_t)},\frac{y-X_t}{1-t}\right\rangle\right\}dt\\
&+p_t(y)\left\langle -\frac{\nabla P_{1-t}f(X_t)}{P_{1-t}f(X_t)}+\frac{y-X_t}{1-t},dX_t\right\rangle\\
&=p_t(y)\left\langle -\frac{\nabla P_{1-t}f(X_t)}{P_{1-t}f(X_t)}+\frac{y-X_t}{1-t},dB_t\right\rangle.
\end{align*}
By integration by parts,
\[
\frac{\nabla P_{1-t}f(x)}{P_{1-t}f(x)}=\int \frac{z-x}{1-t}dp^{x,1-t}(z)\Longrightarrow -\frac{\nabla P_{1-t}f(X_t)}{P_{1-t}f(X_t)}=\frac{X_t-\int zdp_t(z)}{1-t},
\]
so
\[
dp_t(y)=p_t(y)\left\langle \frac{y-\int zdp_t(z)}{1-t},dB_t\right\rangle.
\]
\end{proof}

\subsection*{Moments of the derivative of the Brownian transport map}
Our next goal is to bound the moments of $\D X_t$. To this end, we will use the current best bounds in the Kannan--Lov\'asz--Simonovits conjecture. Let $k\ge 0$ be such that
\[
C_{\text{kls}}\le a d^k
\]
where $a>0$ is some dimension-free constant. If the Kannan--Lov\'asz--Simonovits conjecture is true,  we can take $k=\frac{1}{\log d}$ to get $C_{\text{kls}}\le a e$, which is a dimension-free constant. The result of  \cite{Klartag} is that we can take $k=\frac{\log\log d}{\log d}$, which then yields $C_{\text{kls}}\le a \log d$.

%
%
%

\begin{theorem}{\normalfont{(Isotropic log-concave measures)}}
\label{thm:KLS}
Let $p$ be an isotropic log-concave measure with compact support. Then \eqref{eq:sde} has a unique strong solution on $[0,1]$. Further, there exists a universal $\zeta$ such that, for any positive integer $m$,
\[
\EE_{\gamma}\left[|\D X_t|_{\mathcal L(H,\R^d)}^{2m}\right]\le \zeta^m (2m+1)! (\log d)^{12m}\quad\forall t\in [0,1].
\]
\end{theorem}
\begin{remark}
\label{rem:klsfinitesupp}
The assumption in Theorem \ref{thm:KLS} that $p$ has a compact support is not important for the application to the  Kannan--Lov\'asz--Simonovits  conjecture; see \cite[section 2.6]{chen2021almost}. In particular, the bounds in the theorem are independent of the size of the support of $p$.
\end{remark}

\begin{proof}
By Proposition \ref{prop:existsde}, there exists a unique strong solution $(X_t)$  to \eqref{eq:sde} for all $t\in [0,1]$ with $X_1\sim p$ and, for any $m>0$,
\begin{align}
\label{eq:mbound}
\EE_{\gamma}\left[|\D X_t|_{\mathcal L(H,\R^d)}^{2m}\right]\le \EE_{\gamma}\left[\left(\int_0^t e^{2\int_s^t\lambda_{\max}(\nabla v(r,X_r))dr} ds\right)^m\right]\quad\forall t\in[0,1].
\end{align}
Hence, our goal is  to upper bound the right-hand side of the inequality above. Given $\alpha>2$ define the stopping time 
\[
\tau:=r_0 \wedge \inf\{r\in [0,1]:\lambda_{\max}(\nabla v(r,X_r))\ge \alpha\}
\]
for some $r_0\in \left[0,\frac{t}{2}\right]$ to be chosen later. By Lemma \ref{lem:nablav}(2) (with $\cc=0$), we have $\lambda_{\max}(\nabla v(r,X_r))\le \frac{1}{r}$ for all $r\in [0,1]$ while, on the other hand, $\lambda_{\max}(\nabla v(r,X_r))\le \alpha$ for $r\in [0,\tau]$ by the definition of $\tau$. Hence,
\[
\int_s^t\lambda_{\max}(\nabla v(r,X_r))dr=\int_s^{\tau}\lambda_{\max}(\nabla v(r,X_r))dr+\int_{\tau}^t\lambda_{\max}(\nabla v(r,X_r))dr\le \alpha r_0+\int_{\tau}^t\frac{1}{r}dr=\alpha r_0+\log t-\log\tau
\]
so it follows that
\[
e^{2\int_s^t\lambda_{\max}(\nabla v(r,X_r))dr} \le e^{2\alpha r_0}\frac{t^2}{\tau^2}. 
\]
We conclude that 
\[
\EE_{\gamma}\left[\left(\int_0^t e^{2\int_s^t\lambda_{\max}(\nabla v(r,X_r))dr} ds\right)^m\right]\le e^{2m\alpha r_0}t^{2m}\EE_{\gamma}\left[\frac{1}{\tau^{2m}}\right],
\]
and hence, by \eqref{eq:mbound}, 
\begin{align}
\label{eq:malistop}
\EE_{\gamma}\left[|\D X_t|_{\mathcal L(H,\R^d)}^{2m}\right]\le  e^{2m\alpha r_0}\EE_{\gamma}\left[\frac{1}{\tau^{2m}}\right]t^{2m}. 
\end{align}
In light of \eqref{eq:malistop}, we need to choose $\alpha, r_0$ appropriately and show that $\EE_{\gamma}\left[\frac{1}{\tau^{2m}}\right]$ can be sufficiently bounded. The control of the moments of $\frac{1}{\tau}$ will rely on showing that this random variable has a sub-exponential tail.

\begin{lemma}
\label{lem:taumoment}
Suppose there exist nonnegative constants (possibly dimension dependent) $b_{\alpha},c_{\alpha}$ such that
\[
\PP_{\gamma}\left[\frac{1}{\tau}\ge r\right]\le c_{\alpha}e^{-b_{\alpha}r}\quad\forall r\in\left[\frac{1}{r_0},\infty\right].
\] 
Then,
\[
\EE_{\gamma}\left[\frac{1}{\tau^{m}}\right]\le \frac{1}{r_0^m}\left[1+\left(\frac{1}{b^m_{\alpha}}+1\right)m!c_{\alpha}me^{-\frac{b_{\alpha}}{r_0}}\right].
\]
\end{lemma}

\begin{proof}
We will apply the identity $\EE[Y^m]=m\int_0^{\infty} y^{m-1}\PP[ Y\ge y]dy$, for a nonnegative random variable $Y$, with $Y=\frac{1}{\tau}$. By the definition of $\tau$,  $\PP_{\gamma}\left[\frac{1}{\tau}\ge s\right]=1$ for $s\in \left[0,\frac{1}{r_0}\right]$ so, for any positive integer $m$,
\begin{align*}
\EE_{\gamma}\left[\frac{1}{\tau^{m}}\right]&=m\int_0^{\frac{1}{r_0}}r^{m-1}dr+m\int_{\frac{1}{r_0}}^{\infty} r^{m-1}\PP_{\gamma}\left[\frac{1}{\tau}\ge r\right]dr\le \frac{1}{r_0^m}+c_{\alpha}m\int_{\frac{1}{r_0}}^{\infty} r^{m-1} e^{-b_{\alpha}r}dr\\
&=\frac{1}{r_0^m}+\frac{c_{\alpha}m}{b_{\alpha}^m}\int_{\frac{b_{\alpha}}{r_0}}^{\infty} r^{m-1} e^{-r}dr=\frac{1}{r_0^m}+\frac{c_{\alpha}m(m-1)!}{b_{\alpha}^m}e^{-\frac{b_{\alpha}}{r_0}}\sum_{j=0}^{m-1}\frac{(b_{\alpha})^j}{r_0^jj!}
\end{align*}
where we used the incomplete Gamma function identity $\int_x^{\infty}r^{m-1}e^{-r}dr=(m-1)!e^{-x}\sum_{j=0}^{m-1}\frac{x^j}{j!}$ when $m$ is a positive integer. Using
\[
\frac{1}{j!}\le 1,\quad \quad b_{\alpha}^j\le b_{\alpha}^m+1, \quad\text{and}\quad\frac{1}{r_0^j}\le  \frac{1}{r_0^m} ~\text{ (as $r_0\in [0,1]$)}
\]
we have $\sum_{j=0}^{m-1}\frac{(b_{\alpha})^j}{r_0^jj!}\le m\frac{b_{\alpha}^m+1}{r_0^m}$ and hence
\[
\EE_{\gamma}\left[\frac{1}{\tau^{m}}\right]\le \frac{1}{r_0^m}\left[1+\left(\frac{1}{b_{\alpha}^m}+1\right)m!c_{\alpha}me^{-\frac{b_{\alpha}}{r_0}}\right].
\]
\end{proof}

In light of Lemma \ref{lem:taumoment}, our goal is to prove that $\frac{1}{\tau}$ has a sub-exponential tail, which requires a better understanding of the stopping time $\tau$. To simplify notation let $K_t:=\Cov(p^{X_t,1-t})$ and recall the representation \eqref{eq:vcov},
\[
\nabla v(t,X_t)=\frac{1}{(1-t)^2}K_t-\frac{1}{1-t}\Id_d.
\]
Hence,
\[
\lambda_{\max}(\nabla v(t,X_t))=\frac{\lambda_{\max}(K_t)}{(1-t)^2}-\frac{1}{1-t}
\]
and
\[
\tau=r_0\wedge \inf\left\{r\in [0,1]:\frac{\lambda_{\max}(K_r)}{(1-r)^2}-\frac{1}{1-r}\ge \alpha\right\}.
\]
The quantity $\lambda_{\max}(K_r)$ is difficult to control so we use the moment method and instead control $\Gamma_r:=\Tr[K_r^q]$, while noting that $\lambda_{\max}(K_r)\le \Gamma_r^{\frac{1}{q}}$ for any $q\ge 0$. The process $(\Gamma_t)_{t\in [0,1]}$ satisfies a stochastic differential equation
\[
d\Gamma_t=u_tdB_t+\delta_t dt
\]
for some vector-valued process $(u_t)_{t\in [0,1]}$ and a real-valued process $(\delta_t)_{t\in [0,1]}$. These processes can be derived using It\^o's formula and the stochastic differential equation satisfied by $(K_t)_{t\in [0,1]}$ (which itself can be derived using It\^o's formula). Next, we use the argument in \cite{chen2021almost} to control the processes $(u_t)$ and $(\delta_t)$.

\begin{lemma}
\label{lem:Chen}
Suppose $C_{\textnormal{kls}}\le a d^k$ for $k\ge 0$ and let $q:=\lceil\frac{1}{k}	\rceil+1$. Then, there exists a universal constant $c>0$ such that, for any $r\in \left[0,\frac{1}{2}\right]$, we have, a.s.,
\begin{align*}
&|u_r|\le cq\Gamma_r^{1+\frac{1}{2q}},\\
&\delta_r\le ca^2q^2(\log d)d^{2k-\frac{1}{q}}\Gamma_r^{1+\frac{1}{q}}.
\end{align*}
\end{lemma}

\begin{proof}
The statement of the lemma is essentially \cite[Lemma 6]{chen2021almost}, up to time-change. To make the connection with  \cite{chen2021almost} we recall that, by Lemma \ref{lem:Follmersotcloc}, 
\[
dp_r(x)=p_r(x)\left\langle \frac{x-a_r}{1-r},dB_r\right\rangle,
\]
where $a_r:=\int zdp_r(z)$, and that $\Gamma_r=\Tr\left[(\Cov(p_r))^q\right]$. On the other hand, the arguments of \cite{chen2021almost} use the measure-valued process \cite[Equation (13)]{chen2021almost}:
\begin{align}
\label{eq:sdeChen}
d\tilde p_r(x)=\tilde p_r(x)\langle x-\tilde a_r,dB_r\rangle,\quad \tilde p_0=p
\end{align}
with $\tilde a_r:=\int_{\R^d}zd\tilde p_r(z)$. The connection between $(p_r)$ and $(\tilde p_r)$ is via a time change: set $s(r):=\frac{1}{1-r}-1$ and note that 
\[
d\tilde p_{s(r)}(x)=\sqrt{s'(r)}\tilde p_{s(r)}(x)\langle x-\tilde a_{s(r)},dB_r\rangle=\frac{\tilde p_{s(r)}}{1-r}\langle x-\tilde a_{s(r)},dB_r\rangle.
\]
Since the equation \eqref{eq:sdeChen} has a unique strong solution \cite[Lemma 3]{chen2021almost}, it follows that $p_r=\tilde p_{s(r)}$ a.s. if the same driving Brownian motion is used. In particular, with $\tilde \Gamma_r:=\Tr\left[(\Cov(\tilde p_r))^q\right]$, we have  $\Gamma_r=\tilde \Gamma_{s(r)}$. By \cite[Lemma 6]{chen2021almost}, 
\[
d\tilde \Gamma_r=\tilde u_rdB_r+\tilde\delta_r dr
\]
for some vector-valued process $(\tilde u_r)_{r\in [0,1]}$ and a real-valued process $(\tilde \delta_r)_{r\in [0,1]}$ which satisfy
\begin{align*}
&|\tilde u_r|\le 16q\tilde \Gamma_r^{1+\frac{1}{2q}},\\
&\tilde \delta_r\le 64a^2q^2(\log d)d^{2k-\frac{1}{q}}\tilde \Gamma_r^{1+\frac{1}{q}}.
\end{align*}
Hence, as 
\[
d\Gamma_r=d\tilde \Gamma_{s(r)}=\sqrt{s'(r)}\tilde u_{s(r)}dB_r+s'(r)\tilde\delta_{s(r)}dr,
\]
we get $u_r=\sqrt{s'(r)}\tilde u_{s(r)}$ and $\delta_r=s'(r)\tilde\delta_{s(r)}$ a.s. The proof is complete by noting that $s'(r)$ and $\sqrt{s'(r)}$ are uniformly bounded on $\left[0,\frac{1}{2}\right]$. 
\end{proof}

Extending the analysis of \cite{chen2021almost}, we can use Lemma \ref{lem:Chen} to show that $\frac{1}{\tau}$ has a sub-exponential tail.
\begin{lemma}
\label{lem:subexp}
Suppose $C_{\textnormal{kls}}\le a d^k$ for $k\ge 0$ and let $q:=\lceil\frac{1}{k}	\rceil+1$. There exists a universal constant $c$ such that, with $\alpha=2d^{\frac{1}{q}}$, we have
\[
\PP_{\gamma}\left[\frac{1}{\tau}\ge r\right]\le c_{\alpha}e^{-b_{\alpha}r}\quad\forall r\in\left[\frac{1}{r_0},\infty\right],
\] 
with
\[
c_{\alpha}=\exp\left(\frac{2a^2q(\log d)d^{2k-\frac{1}{q}}}{c}\right)\quad \textnormal{and}\quad b_{\alpha}=\frac{1}{2c^2d^{\frac{1}{q}}}.
\]
\end{lemma}
\begin{proof}
For $s\le r_0$ we have,
\[
\PP_{\gamma}[\tau\le s]=\PP_{\gamma}\left[\sup_{r\in [0,s]}\lambda_{\max}(\nabla v(r,X_r))\ge \alpha\right]\le \PP_{\gamma}\left[\sup_{r\in [0,s]}\lambda_{\max}(K_r)\ge 2\alpha\right]
\]
where the last inequality uses that $r_0\le \frac{1}{2}$ and that $\alpha>2$. Recalling that $\lambda_{\max}(K_r)\le \Gamma_r^{\frac{1}{q}}$ we get,
\[
\PP_{\gamma}[\tau\le s]\le\PP_{\gamma}\left[\sup_{r\in [0,s]}\Gamma_r\ge (2\alpha)^q\right]=\PP_{\gamma}\left[\sup_{r\in [0,s]} \Theta_r\ge (2\alpha)^q\right]
\]
where $(\Theta_r)$ is the stopped process given by 
\[
\Theta_r:=1_{r<\theta}\Gamma_r+1_{r\ge \theta}(2\alpha)^q\quad\text{with}\quad \theta:=\inf\{r: \Gamma_r\ge (2\alpha)^q\}. 
\]
Let $\eta(x)=-x^{-\frac{1}{2q}}$ and note that $\eta$ is monotonically increasing on $(0,\infty)$ so
\[
\PP_{\gamma}[\tau\le s]\le \PP_{\gamma}\left[\sup_{r\in [0,s]}\eta( \Theta_r)\ge \eta((2\alpha)^q)\right]= \PP_{\gamma}\left[\sup_{r\in [0,s]}\eta(\Theta_r)\ge- \frac{1}{\sqrt{2\alpha}}\right].
\]
Moreover, since $\eta(\Theta_r)=-\frac{1}{\sqrt{2\alpha}}$ for $r\ge \theta$, we have 
\begin{align}
\label{eq:tauh}
\PP_{\gamma}[\tau\le s]\le \PP_{\gamma}\left[\sup_{r\in [0,\min(\theta,s)]}\eta(\Theta_r)\ge- \frac{1}{\sqrt{2\alpha}}\right].
\end{align}
Applying It\^o's formula to $\eta(\Theta_r)$, and using Lemma \ref{lem:Chen} as well as $\Gamma_r=\Theta_r$ for $r\le \theta$,  we get, for $r\le \theta$,
\begin{align*}
d\eta(\Theta_r)&=\frac{1}{2q\Gamma_r^{1+\frac{1}{2q}}}d\Gamma_r-\frac{1}{2}\frac{1}{2q}\left(\frac{1}{2q}+1\right)\frac{1}{\Gamma_r^{\frac{1}{2q}+2}}d[\Gamma]_r\le \frac{1}{2q\Gamma_r^{1+\frac{1}{2q}}}d\Gamma_r\\
&=\frac{1}{2q\Gamma_r^{1+\frac{1}{2q}}}u_rdB_r+\frac{1}{2q\Gamma_r^{1+\frac{1}{2q}}}\delta_r dr\\
&\le \frac{1}{2q\Gamma_r^{1+\frac{1}{2q}}}u_rdB_r+ \frac{1}{2}ca^2q(\log d)d^{2k-\frac{1}{q}}\Gamma_r^{\frac{1}{2q}}dr\\
&= \frac{1}{2q\Theta_r^{1+\frac{1}{2q}}}u_rdB_r+ \frac{1}{2}ca^2q(\log d)d^{2k-\frac{1}{q}}\Theta_r^{\frac{1}{2q}}dr.
\end{align*}
Define the martingale $M_s:=\int_0^s\frac{1}{2q\Theta_r^{1+\frac{1}{2q}}}u_rdB_r$ and note that, since $p$ is isotropic, we have $\eta(\Theta_0)=\eta(\Gamma_0)=\eta(d)=-d^{-\frac{1}{2q}}$. Hence,
\[
\eta(\Theta_s)\le -d^{-\frac{1}{2q}}+M_s+\int_0^s\frac{1}{2}ca^2q(\log d)d^{2k-\frac{1}{q}}\Theta_r^{\frac{1}{2q}}dr\le -d^{-\frac{1}{2q}}+M_s+\frac{1}{2}sca^2q(\log d)d^{2k-\frac{1}{q}}\sqrt{2}\sqrt{\alpha},
\]
where the last inequality holds by the definition of $(\Theta_r)$. Plugging this estimate into \eqref{eq:tauh} yields
\[
\PP_{\gamma}[\tau\le s]\le \PP_{\gamma}\left[\sup_{r\in [0,\min(\theta,s)]}M_r\ge- \frac{1}{\sqrt{2\alpha}}+d^{-\frac{1}{2q}}-\frac{\sqrt{2}}{2}sca^2q(\log d)d^{2k-\frac{1}{q}}\sqrt{\alpha}\right].
\]
By the Dubins-Schwarz theorem we have $M_s=Z_{[M]_s}$ with $(Z_s)$ a standard Brownian motion in $\R$, and by Lemma \ref{lem:Chen},
\[
[M_s]=\int_0^s\frac{1}{4q^2\Theta_r^{2+\frac{2}{2q}}}|u_r|^2dr\le \frac{c^2}{4}s.
\]
Hence,
\begin{align*}
&\PP_{\gamma}\left[\sup_{r\in [0,\min(\theta,s)]}M_r\ge- \frac{1}{\sqrt{2\alpha}}+d^{-\frac{1}{2q}}-\frac{\sqrt{2}}{2}sca^2q(\log d)d^{2k-\frac{1}{q}}\sqrt{\alpha}\right]\\
&=\PP_{\gamma}\left[\sup_{r\in [0,\min(\theta,s)]}Z_{[M]_r}\ge- \frac{1}{\sqrt{2\alpha}}+d^{-\frac{1}{2q}}-\frac{\sqrt{2}}{2}sca^2q(\log d)d^{2k-\frac{1}{q}}\sqrt{\alpha}\right]\\
&\le \PP_{\gamma}\left[\sup_{r\in [0,\frac{c^2}{4}\min(\theta,s)]}Z_r\ge- \frac{1}{\sqrt{2\alpha}}+d^{-\frac{1}{2q}}-\frac{\sqrt{2}}{2}sa^2q(\log d)d^{2k-\frac{1}{q}}\sqrt{\alpha}\right]\\
&\le \PP_{\gamma}\left[\sup_{r\in [0,\frac{c^2}{4}s]}Z_r\ge- \frac{1}{\sqrt{2\alpha}}+d^{-\frac{1}{2q}}-\frac{\sqrt{2}}{2}sca^2q(\log d)d^{2k-\frac{1}{q}}\sqrt{\alpha}\right].
\end{align*}
Applying Doob's maximal inequality for Brownian motion we get 
\begin{align*}
& \PP_{\gamma}\left[\sup_{r\in [0,\frac{c^2}{4}s]}Z_r\ge- \frac{1}{\sqrt{2\alpha}}+d^{-\frac{1}{2q}}-\frac{\sqrt{2}}{2}sca^2q(\log d)d^{2k-\frac{1}{q}}\sqrt{\alpha}\right]\\
 &\le\exp\left(-2\frac{\left[- \frac{1}{\sqrt{2\alpha}}+d^{-\frac{1}{2q}}-\frac{\sqrt{2}}{2}sca^2q(\log d)d^{2k-\frac{1}{q}}\sqrt{\alpha}\right]^2}{c^2s}\right).
\end{align*}
Now let $\alpha:=2d^{\frac{1}{q}}$ so that 
\begin{align*}
&\left[- \frac{1}{\sqrt{2\alpha}}+d^{-\frac{1}{2q}}-\frac{\sqrt{2}}{2}sa^2q(\log d)d^{2k-\frac{1}{q}}\sqrt{\alpha}\right]^2=\left[\frac{1}{2d^{\frac{1}{2q}}}-sca^2q(\log d)d^{2k-\frac{1}{2q}}\right]^2\\
&=\frac{1}{4d^{\frac{1}{q}}}-sca^2q(\log d)d^{2k-\frac{1}{q}}+s^2c^2a^4q^2(\log d)^2d^{4k-\frac{1}{q}}.
\end{align*}
Omitting the (positive) last term above we get
\begin{align*}
& \PP_{\gamma}\left[\sup_{r\in [0,\frac{c^2}{4}s]}Z_r\ge- \frac{1}{\sqrt{2\alpha}}+d^{-\frac{1}{2q}}-\frac{\sqrt{2}}{2}sca^2q(\log d)d^{2k-\frac{1}{q}}\sqrt{\alpha}\right]\\
 &\le\exp\left(-\frac{1}{2c^2sd^{\frac{1}{q}}}\right)\exp\left(\frac{2a^2q(\log d)d^{2k-\frac{1}{q}}}{c}\right).
\end{align*}
\end{proof}
We now complete the proof of the theorem. By Lemma \ref{lem:taumoment} and Lemma \ref{lem:subexp}, 
\[
\EE_{\gamma}\left[\frac{1}{\tau^{m}}\right]\le \frac{1}{r_0^m}\left[1+\left(\frac{1}{b^m_{\alpha}}+1\right)m!m\exp\left(\frac{2a^2q(\log d)d^{2k-\frac{1}{q}}}{c}\right)\exp\left(-\frac{1}{2c^2d^{\frac{1}{q}}r_0}\right)\right].
\]
We will choose $r_0\in \left[0,\frac{t}{2}\right]$ such that the two exponentials cancel each other. Setting 
\[
r_0=\frac{t}{4qca^2(\log d)d^{2k}}
\]
we get
\[
\EE_{\gamma}\left[\frac{1}{\tau^{m}}\right]\le \left(\frac{4qca^2(\log d)d^{2k}}{t}\right)^m\left[1+\left(\frac{1}{b^m_{\alpha}}+1\right)m!m\right].
\]

By  \cite[Theorem 1.2]{Klartag}, we may take $k=\frac{\log\log d}{\log d}$, and hence, $q=\lceil\frac{1}{k}\rceil+1=c'\frac{\log d}{\log \log d}$ for some $c'$. By increasing $c'$, we may assume that $2k=\frac{2}{q-1}$. Hence, using $\frac{1}{b^m_{\alpha}}+1\le \frac{2}{b^m_{\alpha}}=2^{m+1}c^{2m}d^{\frac{m}{q}}$ and $\left[1+\left(\frac{1}{b^m_{\alpha}}+1\right)m!m\right]\le  2\left(\frac{1}{b^m_{\alpha}}+1\right)m!m\le 2^{m+2}c^{2m}m!md^{\frac{m}{q}}$, we get 
\begin{align*}
\EE_{\gamma}\left[\frac{1}{\tau^{m}}\right]&\le  \left(\frac{4qca^2(\log d)d^{2k}}{t}\right)^m2^{m+2}c^{2m}m!md^{\frac{m}{q}} \le \frac{1}{t^m}(32)^mc^{3m}a^{2m}m!m[q(\log d)d^{\frac{2}{q-1}+\frac{1}{q}}]^m\\
&\le \frac{1}{t^m}(32)^mc^{3m}a^{2m}m!m[(\log d)qd^{\frac{4}{q-1}}]^m.
\end{align*}
We have 
\begin{align*}
&q^m=(c')^m\left(\frac{\log d}{\log \log d}\right)^m\le (c')^m (\log d)^m,\\
&d^{\frac{4m}{q-1}}=d^{4mk}=(\log d)^{4m},
\end{align*}
so
\[
\EE_{\gamma}\left[\frac{1}{\tau^{m}}\right]\le m! m[16c^3a^2]^m\frac{(\log d)^{6m}}{t^m},
\]
for some constant $a>0$. By \eqref{eq:malistop},
\begin{align*}
\EE_{\gamma}\left[|\D X_t|_{\mathcal L(H,\R^d)}^{2m}\right]&\le  e^{2m\alpha r_0}(2m)! 2m[32c^3a^2]^{2m}(\log d)^{12m}\\
&=\exp\left(2m\frac{2d^{\frac{1}{q}}t}{4qca^2(\log d)d^{\frac{2}{q-1}}}
\right)(2m)! 2m[32c^3a^2]^{2m}(\log d)^{12m}\\
&\le \exp\left(m\frac{1}{cqa^2(\log d)}
\right)(2m)! 2m[32c^3a^2]^{2m}(\log d)^{12m}\le  e^{c''m}(2m)! 2m[32c^3a^2]^{2m}(\log d)^{12m},
\end{align*}
where we used that $d^{\frac{1}{q}-\frac{2}{q-1}}<1$, and that $\exp\left(m\frac{1}{cqa^2(\log d)}
\right)\le e^{c''m}$ for some constant $c''$, as $d\to\infty$. Taking
\[
\zeta:=[32c^3a^2e^{c''}]^2,
\] 
and using $2m(2m)!\le (2m+1)!$, completes the proof.
\end{proof}

\section{Functional inequalities}
\label{sec:funct}
The contraction properties provided by Theorem \ref{thm:causalmain} and Theorem \ref{thm:KLS} allow us to prove functional inequalities for measures in Euclidean spaces. The main  goal of this section is to demonstrate the power of the contraction machinery developed in this paper, rather than be exhaustive, so we focus only on a number of functional inequalities. As a consequence of the almost-sure contraction of Theorem \ref{thm:causalmain}, we will prove $\Psi$-Sobolev inequalities (Theorem \ref{thm:PHIsob}), $q$-Poincar\'e inequalities (Theorem \ref{thm:qpoinc}), and isoperimetric inequalities (Theorem \ref{thm:isoper}). As a consequence of the contraction in expectation of Theorem \ref{thm:KLS},  we will construct Stein kernels and prove central limit theorems (Theorem \ref{thm:SteinIntro} and Corollary \ref{cor:CLTIntro}). 

We start with almost-sure contractions; the next lemma describes the behavior of derivatives under such contractions.

\begin{lemma}
\label{lem:functcontract}
Let $\Upsilon:\Omega\to \R^d$ be an almost-sure contraction with constant $C$ and let $\eta:\R^d\to \R$ be a continuously differentiable Lipschitz function. Then,
\[
D(\eta\circ\Upsilon)=(\mathcal D\Upsilon)^*\nabla\eta(\Upsilon)\quad \gamma\textnormal{-a.e.}
\]
where $(\mathcal D\Upsilon)^*:\R^d\to H$ is the adjoint of $\mathcal D\Upsilon$. Further,
\[
|D(\eta\circ\Upsilon)|_H\le C |\nabla\eta(\Upsilon)|\quad\gamma\textnormal{-a.e.}
\]
\end{lemma}
\begin{proof}
To compute $D(\eta\circ\Upsilon)$ we note that, by duality, it can be viewed as the operator $\mathcal D(\eta\circ\Upsilon):H\to \R$ acting on $h\in H$ by $\mathcal D(\eta\circ\Upsilon)[h]=\langle  D(\eta\circ\Upsilon),h\rangle_H$. By the chain rule  \cite[Proposition 1.2.3]{nualart2006malliavin},
\[
\langle  D(\eta\circ\Upsilon),h\rangle_H=\int_0^1(\nabla\eta(\Upsilon))^*D_t\Upsilon\dot{h}_tdt=\langle \nabla\eta(\Upsilon),\mathcal D\Upsilon[h]\rangle=\langle  (\mathcal D\Upsilon)^*\nabla\eta(\Upsilon),h\rangle_H
\]
so $  D(\eta\circ\Upsilon)= (\mathcal D\Upsilon)^*\nabla\eta(\Upsilon)$. 
Next, using
\[
D(\eta\circ\Upsilon)=(\mathcal D\Upsilon)^*\nabla\eta(\Upsilon)\quad \gamma\textnormal{-a.e.},
\]
the bound $|(\mathcal D\Upsilon)^*|_{\mathcal L(\R^d,H)}=|\mathcal D\Upsilon|_{\mathcal L(H,\R^d)}\le C$ (Lemma \ref{lem:rad}) implies
\[
|D(\eta\circ\Upsilon)|_H\le C |\nabla\eta(\Upsilon)|\quad\gamma\textnormal{-a.e.}
\]
\end{proof}

With Lemma \ref{lem:functcontract} in hand we can now start the proofs of the functional inequalities which follow from Theorem \ref{thm:causalmain}. We begin with the $\Psi$-Sobolev inequalities \cite{chafai2002entropies}. 
\begin{definition}
\label{def:PHIsob}
Let $\mathcal I$ be a closed interval (possibly unbounded) and let $\Psi:\mathcal I\to \R$ be a twice-differentiable function. We say that $\Psi$ is a \emph{divergence} if each of the functions $\Psi, \Psi'', -\frac{1}{\Psi''}$ is convex. Given a probability measure $\nu$ on $\R^d$  and a function $\eta:\R^d\to\mathcal I$, such that $\int \eta\, d\nu\in\mathcal I$, we define
\[
\Ent_{\nu}^{\Psi}(\eta):=\int_{\R^d}\Psi(\eta)d\nu-\Psi\left(\int_{\R^d}\eta\,d\nu\right).
\]
\end{definition}
Some classical examples of divergences are $\Psi:\R\to \R$ with $\Psi(x)=x^2$ (Poincar\'e inequality) and $\Psi:\R_{\ge 0}\to \R$ with $\Psi(x)=x\log x$ (log-Sobolev inequality). 
\begin{theorem}{\textnormal{($\Psi$-Sobolev inequalities)}}
\label{thm:PHIsob}
 Let $\Psi:\mathcal I\to \R$ be a divergence. 
\vspace{0.1in}

\begin{enumerate}
\item Let $p$ be a $\cc$-log-concave measure with $S:=\textnormal{diam}(\supp (p))$ and let $\eta:\R^d\to \mathcal I$ be any continuously differentiable Lipschitz function such that \rev{$\int \eta\, dp\in\mathcal I$}. 
\vspace{0.1in}

\begin{itemize}
\item If \rev{$\cc S^2\ge 4$} then 
\[
\Ent_p^{\Psi}(\eta)\le \frac{1}{2\cc}\int_{\R^d}\Psi''(\eta)|\nabla \eta|^2 dp. 
\]

\item If \rev{$\cc S^2< 4$} then
\rev{\[
\Ent_{p}^{\Psi}(\eta)\le \frac{e^{2-\frac{\cc}{2} S^2}+1}{16}S^2\int_{\R^d}\Psi''(\eta)|\nabla \eta|^2 dp.
\]}
\end{itemize}

\item Fix a probability measure $\nu$ on $\R^d$ \rev{with $R:=\diam(\supp(\nu))$} and let $p:=\gamma_d^{a,\Sigma}\star \nu$ where $\gamma_d^{a,\Sigma}$ is the Gaussian measure on $\R^d$ with mean $a$ and covariance $\Sigma$. Set $\lambda_{\min}:=\lambda_{\min}(\Sigma)$ and $\lambda_{\max}:=\lambda_{\max}(\Sigma)$. Then, for any $\eta:\R^d\to \mathcal I$, a continuously differentiable Lipschitz function such that \rev{$\int \eta\, dp\in\mathcal I$}, we have
\rev{\begin{align*}
\Ent_{p}^{\Psi}(\eta)\le \frac{\lambda_{\min}\lambda_{\max}}{R^2}\left(e^{ \frac{R^2}{2\lambda_{\min}}}-1\right)\int_{\R^d}\Psi''(\eta)|\nabla \eta|^2 d p.
\end{align*}}
\end{enumerate} 
\end{theorem}

\begin{proof}
$~$

\begin{enumerate}
\item We will use the fact \cite[Theorem 4.2]{chafai2002entropies} that $\Psi$-Sobolev inequalities hold for the Wiener measure $\gamma$:
\[
\Ent_{\gamma}^{\Psi}(F)\le \frac{1}{2}\EE_{\gamma}\left[\Psi''(F)|DF|_{H}^2\right]
\]
for any $F:\Omega\to \mathcal I$ which is $L^2$-integrable with respect to $\gamma$. Let $(X_t)_{t\in [0,1]}$ be the F\"ollmer process associated to $p$, so that $X_1\sim p$, and suppose that $X_1:\Omega\to \R^d$ is an almost-sure contraction with constant $C$. 
Given $\eta$ let $F(\omega):=(\eta\circ X_1)(\omega)$.
Then, by Lemma \ref{lem:functcontract} and \cite[Proposition 1.2.4]{nualart2006malliavin},
\begin{align*}
\Ent_{p}^{\Psi}(\eta)&=\Ent_{\gamma}^{\Psi}(F)\le \frac{1}{2}\EE_{\gamma}\left[\Psi''(F)|DF|_H^2\right]\le  \frac{C^2}{2}\EE_{\gamma}\left[\Psi''(\eta\circ X_1)|\nabla\eta(X_1)|^2\right]=\frac{C^2}{2}\int_{\R^d}\Psi''(\eta)|\nabla \eta|^2 dp.
\end{align*}
The proof is complete by Theorem \ref{thm:causalmain}.\\
%

\item Let $Y\sim \nu$, let $\tilde \nu$ be the law $\Sigma^{-1/2}Y$, and  define $\tilde p:=\gamma_d\star \tilde \nu$. Set $\lambda_{\min}:=\lambda_{\min}(\Sigma)$ and $\lambda_{\max}:=\lambda_{\max}(\Sigma)$. The argument of part (1) gives,
\rev{\[
\Ent_{\tilde p}^{\Psi}(\eta)\le\frac{e^{\frac{1}{2}\lambda_{\min}^{-1} R^2}-1}{\lambda_{\min}^{-1} R^2}\int_{\R^d}\Psi''(\eta)|\nabla \eta|^2 d\tilde p.
\]}
Let $p=\gamma_d^{a,\Sigma}\star \nu$ and let $\tilde X\sim \tilde p$ so that $\Sigma^{1/2}\tilde X+a\sim p$. Given $\eta$ let $\tilde \eta:=\eta(\Sigma^{1/2}x+a)$ so that
\rev{\begin{align*}
\Ent_{p}^{\Psi}(\eta)=\Ent_{\tilde p}^{\Psi}(\tilde \eta)\le \frac{e^{\frac{1}{2}\lambda_{\min}^{-1} R^2}-1}{\lambda_{\min}^{-1} R^2}\int_{\R^d}\Psi''(\tilde \eta)|\nabla \tilde \eta|^2 d\tilde p.
\end{align*}}
Since $\nabla \tilde \eta(x)=\Sigma^{1/2}\nabla \eta(\Sigma^{1/2}x+a)$ we have $|\nabla \tilde \eta(x)|^2\le \lambda_{\max} |\nabla \eta(\Sigma^{1/2}x+a)|^2$ so
\rev{\begin{align*}
\Ent_{p}^{\Psi}(\eta)\le \frac{\lambda_{\min}\lambda_{\max}}{R^2}\left(e^{ \frac{R^2}{2\lambda_{\min}}}-1\right)\int_{\R^d}\Psi''(\eta)|\nabla \eta|^2 d p.
\end{align*}}
\end{enumerate}
\end{proof}

\begin{theorem}{\textnormal{($q$-Poincar\'e inequalities)}}
\label{thm:qpoinc}
Let $q\in [1,\infty)$ and set $c_q:=1_{q\in [1,2)}\frac{\pi}{2}+1_{q\in [2,\infty)}\sqrt{q-1}$. Let $\eta:\R^d\to \mathcal \R$ be any continuously differentiable Lipschitz function such that $\int \eta\,dp=0$ and $\eta\in L^q(\gamma_d)$.

\begin{enumerate}
\item Let $p$ be a $\cc$-log-concave measure with $S:=\textnormal{diam}(\supp (p))$. 
\vspace{0.1in}

\begin{itemize}
\item If \rev{$\cc S^2\ge 4$} then 
\[
\EE_p[\eta^q]\le \frac{1}{\cc^{q/2}}c_q^q\EE_p[|\nabla \eta|^q].
\]

\item If \rev{$\cc S^2< 4$} then
\rev{\[
\EE_p[\eta^q]\le\left(\frac{e^{2-\frac{\cc}{2} S^2}+1}{8}\right)^{q/2}S^qc_q^q\EE_p[|\nabla \eta|^q].
\]}
\end{itemize}

\item Fix a probability measure $\nu$ on $\R^d$ \rev{with $R:=\diam(\supp(\nu))$} and let $p:=\gamma_d^{a,\Sigma}\star \nu$ where $\gamma_d^{a,\Sigma}$ is a the Gaussian measure on $\R^d$ with mean $a$ and covariance $\Sigma$. Then, with $\lambda_{\min}:=\lambda_{\min}(\Sigma)$ and $\lambda_{\max}:=\lambda_{\max}(\Sigma)$, we have
\rev{\[
\EE_{p}[\eta^q]\le    2^{q/2}c_q^q \frac{(\lambda_{\min}\lambda_{\max})^{q/2}}{R^q}(e^{\frac{R^2}{2\lambda_{\min}}}-1)^{q/2}\EE_{p}[|\nabla \eta|^q].
\]}
\end{enumerate} 
\end{theorem}

\begin{proof}
$~$
\begin{enumerate}
\item We will use the fact \cite[Theorem 2.6]{addona2022equivalence} (see \cite[Proposition 3.1(3)]{nourdin2009second}  for an earlier result) that the $q$-Poincar\'e inequality holds for the Wiener measure $\gamma$: For  $q\in [1,\infty)$ and $F\in \mathbb D^{1,q}$ with $\EE_{\gamma}[F]=0$, we have 
\[
\EE_{\gamma}[F^q]\le c_q^q\EE_{\gamma}[|DF|_{H}^q]. 
\]
Let $(X_t)_{t\in [0,1]}$ be the F\"ollmer process associated to $p$, so that $X_1\sim p$, and suppose that $X_1:\Omega\to \R^d$ is an almost-sure contraction with constant $C$. Given $\eta$ let $F(\omega):=(\eta\circ X_1)(\omega)$. Then, by Lemma \ref{lem:functcontract} and \cite[Proposition 1.2.4]{nualart2006malliavin},
\begin{align*}
&\EE_p[\eta^q]=\EE_{\gamma}[F^q]\le  c_q^q\EE_{\gamma}[|DF|_{H}^q]\le C^qc_q^q\EE_{\gamma}[|\nabla\eta(X_1)|^q]=C^qc_q^q\EE_p[|\nabla\eta|^q]. 
\end{align*}
The proof is complete by Theorem \ref{thm:causalmain}. \\

\item Let $Y\sim \nu$, let $\tilde \nu$ be the law $\Sigma^{-1/2}Y$, and  define $\tilde p:=\gamma_d\star \tilde \nu$. Set $\lambda_{\min}:=\lambda_{\min}(\Sigma)$ and $\lambda_{\max}:=\lambda_{\max}(\Sigma)$. \rev{Since $\diam(\supp(\tilde \nu))\leq \frac{R}{\sqrt{\lambda_{\min}}}$}, the argument of part (1) gives,
\rev{\[
\EE_{\tilde p}[\eta^q]\le \left(2\frac{e^{\frac{1}{2}\lambda_{\min}^{-1}R^2}-1}{\lambda_{\min}^{-1}R^2}\right)^{q/2}c_q^q\EE_{\tilde p}[|\nabla \eta|^q]. 
\]}
Let $p=\gamma_d^{a,\Sigma}\star \nu$ and let $\tilde X\sim \tilde p$ so that $\Sigma^{1/2}\tilde X+a\sim p$. Given $\eta$ let $\tilde \eta:=\eta(\Sigma^{1/2}x+a)$ so that
\rev{\begin{align*}
\EE_{p}[\eta^q]=\EE_{\tilde p}[\tilde \eta^q]\le \left(2\frac{e^{\frac{1}{2}\lambda_{\min}^{-1}R^2}-1}{\lambda_{\min}^{-1}R^2}\right)^{q/2}c_q^q\EE_{\tilde p}[|\nabla \tilde \eta|^q].
\end{align*}}
Since $\nabla \tilde \eta(x)=\Sigma^{1/2}\nabla \eta(\Sigma^{1/2}x+a)$ we have $|\nabla \tilde \eta(x)|^q\le \lambda_{\max}^{q/2} |\nabla \eta(\Sigma^{1/2}x+a)|^q$ so
\rev{\begin{align*}
\EE_{p}[\eta^q]\le    2^{q/2}c_q^q \frac{(\lambda_{\min}\lambda_{\max})^{q/2}}{R^q}(e^{\frac{R^2}{2\lambda_{\min}}}-1)^{q/2}\EE_{p}[|\nabla \eta|^q].
\end{align*}}

\end{enumerate}
\end{proof}

\begin{theorem}{\textnormal{(Isoperimetric inequalities)}}
\label{thm:isoper}
Let $\Phi$ be the cumulative distribution function of $\gamma_1$ and let $B_d\subset \R^d$ be the unit ball. 
\begin{enumerate}
\item Let $p$ be a $\cc$-log-concave measure with $S:=\textnormal{diam}(\supp (p))$ and let
\rev{\[
C:=
\begin{cases}
\frac{1}{\sqrt{\cc}}&\text{if }\cc S^2\ge 4\\
\left(\frac{e^{2-\frac{\cc}{2} S^2}+1}{8}\right)^{1/2}S&\text{if }\cc S^2< 4.
\end{cases}
\]}
Then, for any Borel set $A\subset \R^d$ and $r\ge 0$,
\[
p[A+r B_d]\ge\Phi\left(\Phi^{-1}(p[A])+\frac{r}{C}\right).
\]
\item Fix a probability measure $\nu$ on $\R^d$ \rev{with $R:=\diam(\supp(\nu))$} and let $p:=\gamma_d^{a,\Sigma}\star \nu$ where $\gamma_d^{a,\Sigma}$ is a the Gaussian measure on $\R^d$ with mean $a$ and covariance $\Sigma$. Set $\lambda_{\min}:=\lambda_{\min}(\Sigma),~ \lambda_{\max}:=\lambda_{\max}(\Sigma)$, and 
\rev{\[
C:= (2\lambda_{\min}\lambda_{\max})^{1/2}\frac{(e^{\frac{R^2}{2\lambda_{\min}}}-1)^{1/2}}{R}.
\]}
Then,
\[
p[A+rB_d]\ge \Phi\left(\Phi^{-1}(p[A])+\frac{r}{C}\right).
\]
\end{enumerate} 
\end{theorem}

\begin{proof}
$~$
\begin{enumerate}

\item Let $B_{\HH}$ be the unit ball in $\HH$. We will use the fact \cite[Theorem 4.3]{ledoux1996isoperimetry} that the Wiener measure $\gamma$ satisfies the isoperimetric inequality:
\[
\gamma[K+r B_{\HH}]\ge \Phi(\Phi^{-1}(\gamma(K))+r)
\]
for any Borel measurable set $K\subset \Omega$ and $r\ge 0$; see the discussion following \cite[Theorem 4.3]{ledoux1996isoperimetry} for measurability issues.

Let $(X_t)_{t\in [0,1]}$ be the F\"ollmer process associated to $p$ so that $X_1\sim p$. Suppose that $X_1:\Omega\to \R^d$ is an almost-sure  contraction with constant $C$, so in particular, 
\[
|X_1(\omega+h)-X_1(\omega)|\le C|h|_{\HH}\quad\forall h\in \HH, \quad \gamma\text{-a.e.}
\]
Let $M\subset \R^d$ be a Borel measurable set. We will show that
\begin{align}
\label{eq:incl}
X_1^{-1}(M)+\frac{r}{C}B_{\HH}\subset X_1^{-1}(M+rB_d).
\end{align}
Then, by the isoperimetric inequality for $\gamma$ and \eqref{eq:incl},
\begin{align*}
p[M+rB_d]=\gamma[X_1^{-1}(M+rB_d)]&\ge \gamma\left[X_1^{-1}(M)+\frac{r}{C}B_{\HH}\right]\\
&\ge \Phi\left(\Phi^{-1}\left(\gamma[X_1^{-1}(M)]\right)+\frac{r}{C}\right)=\Phi\left(\Phi^{-1}(p[M])+\frac{r}{C}\right). 
\end{align*}
The proof is then complete by  Theorem \ref{thm:causalmain}.

In order to prove \eqref{eq:incl} it suffices to show that 
\[
X_1\left(X_1^{-1}(M)+\frac{r}{C}B_{\HH}\right)\subset M+rB_d,
\]
or, in other words, that $\omega \in X_1^{-1}(M)+\frac{r}{C}B_{\HH}\Rightarrow X_1(\omega)\in  M+rB_d$. Fix $\omega \in X_1^{-1}(M)+\frac{r}{C}B_{\HH}$ so that $\omega=\theta+\frac{r}{C}h$ for some $\theta\in X_1^{-1}(M)$ and $h\in \HH$ with $|h|_{\HH}=1$. Then $X_1(\omega-\frac{r}{C}h)\in M$ and, as $X$ is an almost-sure contraction with a constant $C$, $\left|X_1\left(\omega-\frac{r}{C}h\right)-X_1(\omega)\right|\le r$ so $X_1(\omega)\in M+rB_d$ as desired.\\

\item  Let $Y\sim \nu$, let $\tilde \nu$ be the law $\Sigma^{-1/2}Y$, and  define $\tilde p:=\gamma_d\star \tilde \nu$. Set $\lambda_{\min}:=\lambda_{\min}(\Sigma)$ and $\lambda_{\max}:=\lambda_{\max}(\Sigma)$. \rev{Since $\diam(\supp(\tilde \nu))\leq \frac{R}{\sqrt{\lambda_{\min}}}$}, the argument of part (1) gives, for any Borel set $M\subset \R^d$ and $r\ge 0$,
\[
\tilde p [M+rB_d]\ge \Phi\left(\Phi^{-1}(\tilde p[M])+\frac{r}{C}\right)
\]
with \rev{$C:=\left(2\frac{e^{\frac{1}{2}\lambda_{\min}^{-1}R^2}-1}{\lambda_{\min}^{-1}R^2}\right)^{1/2}$}. Let $p=\gamma_d^{a,\Sigma}\star \nu$ and let $\tilde X\sim \tilde p$ so that $\Sigma^{1/2}\tilde X+a\sim p$. Then, for any Borel set $M\subset \R^d$ and $r\ge 0$,
\[
p[M+rB_d]=\tilde p[\Sigma^{-1/2}(M-a)+\Sigma^{-1/2}rB_d]\ge \tilde p[\Sigma^{-1/2}(M-a)+\lambda_{\max}^{-1/2}rB_d].
\]
Hence,
\[
p[M+rB_d]\ge \Phi\left(\Phi^{-1}\left(\tilde p\left[\Sigma^{-1/2}(M-a)\right]\right)+\frac{r\lambda_{\max}^{-1/2}}{C}\right).
\]
The proof is complete by noting that $\tilde p\left[\Sigma^{-1/2}(M-a)\right]=p[M]$. 
\end{enumerate}
\end{proof}

\subsection*{Stein kernels} We now turn to the applications of the contraction in expectation, as in Theorem \ref{thm:KLS}. Specifically, we shall prove Theorem \ref{thm:SteinIntro}, from which Corollary \ref{cor:CLTIntro} follows, as explained in the introduction. We first establish the connection between the Brownian transport map and Stein kernels. Given Malliavin differentiable functions $F,G:\Omega\to \R^k$ we denote
\[
\left( DF,DG\right)_H:=\int_0^1D_tF(D_tG)^*dt.
\]
Note that, as outlined in section \ref{sec:prem}, for every fixed $t \in [0,1]$, $D_tF$ is a $k \times d$ matrix, and so $\left(DF,DG\right)_H$ takes values in the space of $k\times k$ matrices.
The construction of the Stein kernel relies on the Ornstein-Uhlenbeck operator $\mathcal L$, as defined, as in \cite[section 2.8.2]{nourdin2012normal}. To define the operator, let $\delta$ stand for the adjoint of the Malliavin derivative $D$, also called the Skorokhod integral. For our purposes we shall only use $\delta$ on  matrix-valued paths $DF$ and $DG$, where $F$ and $G$ are as above. In this case, $\delta$ acts on the rows of $DG$, and $\delta DG$ takes values in $\R^k$. Formally, the adjoint property of $\delta$ is given by
$$\rev{\EE\left[\left(DF,DG\right)_H\right] = \EE\left[\langle F, \delta D G \rangle\right]},$$
where the inner product on the right hand side is the Euclidean one in $\R^k$.

We can now define the Ornstein-Uhlenbeck operator as $\mathcal L:= - \delta D$. By construction, if $G:\Omega\to \R^k$, then $\mathcal L G:\Omega \to \R^k$ as well. A useful property of $\mathcal L$ is that it is invertible on the subspace of functions $G$, satisfying $\EE_\gamma\left[G\right]=0$, and we denote the pseudo-inverse by $\mathcal L^{-1}$; see \cite[section 2.8.2]{nourdin2012normal} for more details, and in particular \cite[Proposition 2.8.11]{nourdin2012normal}. Now, given a Malliavin differentiable function $F:\Omega\to \R^k$ such that, $\EE_{\gamma}[F]=0$,  we define the $k\times k$ matrix-valued map
\[
\tau(x):=\EE_{\gamma}\left[(-D\mathcal L^{-1}F,DF)_H|F=x\right].
\]
Above, the expression $\EE_{\gamma}\left[\cdot|F=x\right]$ is the expectation of the regular conditional probability on the fibers of the map $F^{-1}(x)$, which is well-defined for almost every $x\in \R^k$, with respect to the law of $F$, \cite{hoffman1971existence}.

\begin{lemma}
\label{lem:steinkernel}
Let $F: \Omega \to \R^k$ be Malliavin differentiable and satisfy $\EE_\gamma\left[F\right]= 0$. Then, the map $\tau$ is a Stein kernel for $F_*\gamma$. 
\end{lemma}
\begin{proof} 
The proof follows the argument in \cite[Lemma 1]{Mikulincer}. Let $\eta:\R^k\to \R^k$ be a continuously differentiable and Lipschitz function and let $Y\sim F_*\gamma$. We need to show that
\[
\EE[\langle \nabla\eta(Y),\tau(Y)\rangle_{\text{HS}}]=\EE[\langle \eta(Y),Y\rangle]
\]
where $\langle \cdot,\cdot\rangle_{\text{HS}}$ is the Hilbert-Schmidt inner product. We recall that $\mathcal L=-\delta D$ where $\delta$ is the adjoint to the Malliavin derivative $D$. Compute,
\begin{align*}
&\EE[\langle \nabla\eta(Y),\tau(Y)\rangle_{\text{HS}}]=\EE[\langle \nabla\eta(Y),\EE_{\gamma}\left[(-D\mathcal L^{-1}F,DF)_H|F=Y\right]\rangle_{\text{HS}}]\\
&=\EE_{\gamma}[\langle \nabla\eta(F),(-D\mathcal L^{-1}F,DF)_H\rangle_{\text{HS}}]\\
&=\EE_{\gamma}[\Tr[(-D\mathcal L^{-1}F,D(\eta\circ F),)_H]]\quad (\text{chain rule})\\
&=\EE_{\gamma}[\langle \eta\circ F,-\delta D\mathcal L^{-1}F\rangle]\quad (\text{$\delta$ is adjoint to $D$})\\
&=\EE_{\gamma}[\langle \eta\circ F,\mathcal L\mathcal L^{-1}F\rangle]\quad (\mathcal L=-\delta D)\\
&=\EE_{\gamma}[\langle \eta\circ F, F\rangle]\quad (\text{when }\EE_{\gamma}[F]=0)\\
&=\EE[\langle \eta(Y),Y\rangle]. 
\end{align*}
\end{proof}
\stein*
\begin{remark}
As will become evident from the proof of 
Theorem \ref{thm:SteinIntro}, the result holds provided that $\chi\circ X_1$ is a Malliavin differentiable random vector where $(X_t)$ is the F\"ollmer process associated to $p$. By \cite[Proposition 1.2.3]{nualart2006malliavin}, this condition holds if $\chi$ is a continuously differentiable function with bounded partial derivatives. 
\end{remark}
\begin{proof}
Let $F=\chi \circ X_1$ and let $\tau$ be the Stein kernel constructed above so that $\tau_q:=\tau$ is a Stein kernel of $q$. Let $(\mathcal P_t)$ be the Ornstein-Uhlenbeck semigroup on the Wiener space\footnote{We will write $\mathcal P_s(DF)$ for the operator acting coordinate-wise on the matrix-valued function $DF$, and similarly, we write $\mathcal P_s(D_rF)$ for the operator acting coordinate-wise on the random matrix $D_rF$.} and recall that it is a contraction \cite[Proposition 2.8.6]{nourdin2012normal}. By \cite[Proposition 2.9.3]{nourdin2012normal},
\[
\EE_q[|\tau_q|_{\textnormal{HS}}^2]\leq\EE_{\gamma}[|(DF,D\mathcal L^{-1}F)_H|_{\textnormal{HS}}^2]=\EE_{\gamma}\left[\left|\int_0^{\infty}e^{-s}(DF,\mathcal P_s(DF))_Hds\right|_{\textnormal{HS}}^2\right]\le \sup_{s\in[0,\infty)}\EE_{\gamma}\left[\left|(DF,\mathcal P_s(DF))_H\right|_{\textnormal{HS}}^2\right].
\]
\rev{Note that for any $r \in[0,1]$, by the chain rule $D_rF = \nabla \chi(X_1)D_rX_1$. Hence, for a fixed $s\geq 0$, and using $(\mathcal P_s(DF))_r=\mathcal P_s(D_rF)$,
	$$(DF,\mathcal P_s(DF))_H = \nabla \chi(X_1)\int_0^1D_rX_1\mathcal P_s(D_rF)^*dr \implies \mathrm{rank}\left((DF,\mathcal P_s(DF))_H\right) \leq \min\{k,d\}.$$
	Thus, if $u,v\in \R^k$ are two vectors,
	\begin{align*}
	\langle u, (DF,\mathcal P_s(DF))_Hv\rangle^2 &= \left(\int_0^1\langle (D_rF)^*u,\mathcal P_s(D_rF)^*v\rangle dr\right)^2 \leq \int_0^1 |(D_rF)^*u|^2dr\int_0^1|\mathcal P_s(D_rF)^*v|^2dr\\
	&=\langle u,(DF,DF)_Hu\rangle\langle v,(\mathcal{P}_s(DF),\mathcal{P}_s(DF))_Hv\rangle\\
	 &\leq |u|^2|v|^2|(DF,DF)_H|_{\textnormal{op}}|(\mathcal{P}_s(DF),\mathcal{P}_s(DF))_H|_{\textnormal{op}}.
	\end{align*}
	Combining the above and taking supremum over unit vectors $u$ and $v$ we get
	$$\left|(DF,\mathcal P_s(DF))_H\right|_{\textnormal{HS}}^2 \leq \min\{k,d\}\left|(DF,\mathcal P_s(DF))_H\right|_{\textnormal{op}}^2\leq \min\{k,d\}\left|(DF,DF)_H\right|_{\textnormal{op}}\left|(\mathcal P_s(DF),\mathcal P_s(DF))_H\right|_{\textnormal{op}}.$$
	For the second term, for $v \in \R^k$ Jensen's inequality gives
	$$\langle v,(\mathcal P_s(DF),\mathcal P_s(DF))_Hv\rangle =\int_0^1|\mathcal P_s(D_rF)^*v|^2dr \leq  \int_0^1\mathcal{P}_s|\mathcal (D_rF)^*v|^2dr = \mathcal{P}_s\langle v,(DF,DF)_Hv\rangle,$$
	which implies $\left|(\mathcal P_s(DF),\mathcal P_s(DF))_H\right|_{\textnormal{op}} \leq \mathcal{P}_s\left(\left|(DF,DF)_H\right|_{\textnormal{op}}\right)$.
	So, by applying Cauchy–Schwarz and using the fact that $\mathcal P$ is a contraction,
\begin{align*}
		\EE_q[|\tau_q|_{\textnormal{HS}}^2] &\leq \min\{k,d\}\sup_{s\in[0,\infty)}\EE_\gamma\left[\left|(DF,DF)_H\right|_{\textnormal{op}}\mathcal{P}_s\left(\left|(DF,DF)_H\right|_{\textnormal{op}}\right)\right]\\
		&\leq \min\{k,d\}\sup_{s\in[0,\infty)}\sqrt{\EE_\gamma\left[\left|(DF,DF)_H\right|^2_{\textnormal{op}}\right]\EE_\gamma\left[\mathcal{P}_s\left(\left|(DF,DF)_H\right|_{\textnormal{op}}\right)^2\right]}\\
		&\leq \min\{k,d\}\EE_\gamma\left[\left|(DF,DF)_H\right|^2_{\textnormal{op}}\right].
\end{align*}
To finish the proof we again apply the chain rule to $F$ and obtain
\begin{align*}
		 \EE_\gamma\left[\left|(DF,DF)_H\right|^2_{\textnormal{op}}\right]&=\EE_\gamma\left[\left|\nabla \chi(X_1)(DX_1,DX_1)_H\nabla \chi(X_1)^*\right|^2_{\textnormal{op}}\right]\\
		&\leq \EE_\gamma\left[\left|\nabla\chi(X_1)\right|^4_{\textnormal{op}}\left|(DX_1,DX_1)_H\right|^2_{\textnormal{op}}\right] \leq \EE_\gamma\left[\left|\nabla\chi(X_1)\right|^4_{\textnormal{op}}\left|\mathcal{D}X_1\right|^4_{\mathcal{L}(H,\R^d)}\right]\\
		&\leq \sqrt{\EE_{p}\left[\left|\nabla \chi\right|_{\textnormal{op}}^8\right]}\sqrt{\EE_{\gamma}\left[\left|\mathcal{D}X_1\right|^8_{\mathcal{L}(H,\R^d)}\right]}.
\end{align*}
To see the second inequality take a unit vector $v$, so that,
$$\left|(DX_1,DX_1)_Hv\right|^2 = \left(\int_0^1 D_tX_1 ((D_tX_1)^*v)dt\right)^2\leq |\mathcal{D}X_1|^2_{\mathcal{L}(H,\R^d)}|\langle DX_1,v\rangle_H|^2 \leq |\mathcal{D}X_1|^4_{\mathcal{L}(H,\R^d)}.$$
The proof now concludes by applying Theorem \ref{thm:KLS}.
}
\end{proof}

\section{Cameron-Martin contractions}
\label{sec:CMcontract}
The notion of contraction we considered up until now was the appropriate one when the target measures were measures on $\R^d$. If, however, we are interested in transportation between measures on the Wiener space itself, then we need a stronger notion of contraction. 

\begin{definition}
\label{def:transport}
A measurable map $T:\Omega\to \Omega$ is a  \emph{Cameron-Martin transport map} if $T(\omega)=\omega+\xi(\omega)$ for some measurable map $\xi:\Omega\to \HH$; we write $\xi(\omega)=\int_0^{\cdot}\dot{\xi}(\omega)$ for some measurable map $\dot{\xi}:\Omega\to H$. We set $(T_t)_{t\in [0,1]}:=( W_t\circ T)_{t\in [0,1]}$. A Cameron-Martin transport map $T:\Omega\to \Omega$ is a \emph{Cameron-Martin contraction} with constant $C$ if, $\gamma$-a.e., 
\[
|T(\omega+h)-T(\omega)|_{H^1}\le C|h|_{\HH}\quad \forall ~h\in \HH. 
\]
\end{definition}

\begin{claim*}
Let $T:\Omega\to \Omega$ be a Cameron-Martin contraction with constant $C$. Then, for any $t\in [0,1]$, $T_t:\Omega\to \R^d$ is an almost-sure contraction with constant $C$. 
\end{claim*}
\begin{proof}
Let $T:\Omega\to \Omega$ be a Cameron-Martin contraction with constant $C$. Fix $h\in \HH,\omega\in \Omega,$ and define $q\in \HH$ by $q_t=T_t(\omega+h)-T_t(\omega)$ for $t\in [0,1]$; note that $q$ is indeed an element of $\HH$ since $T$ is a Cameron-Martin transport map. Since
\[
\sup_{t\in [0,1]}|q_t|=\sup_{t\in [0,1]}\left|\int_0^t\dot{q}_sds\right|,
\]
it follows from Jensen's inequality that
\[
\sup_{t\in[0,1]}|T_t(\omega+h)-T_t(\omega)|^2=\sup_{t\in [0,1]}\left|\int_0^t\dot{q}_sds\right|^2\le \sup_{t\in [0,1]}t\int_0^t|\dot{q}_s|^2ds\le \int_0^1|\dot{q}_s|^2ds=|T(\omega+h)-T(\omega)|_{H^1}^2.
\]
\end{proof}

We see that a Cameron-Martin contraction is a stronger notion than an almost-sure contraction. Given a measure $\mu$ on $\Omega$, a Cameron-Martin contraction between $\gamma$ and $\mu$ would transfer functional inequalities from $\gamma$ to $\mu$  where the functions are allowed to depend on the entire path $\{\omega_t\}_{t\in [0,1]}$. For the rest of the paper, we focus on the question of whether either the causal optimal transport map or the optimal transport map is a Cameron-Martin contraction when the target measure enjoys convexity properties. The notion of convexity we use is compatible with the Cameron-Martin space \cite{feyel2000notion}:
\begin{definition}
\label{def:convex}
A measurable map $V:\Omega\to \R\cup\{\infty\}$ is \emph{Cameron-Martin convex} if, for any $ h,g\in \HH$ and $\alpha\in [0,1]$, it holds that
\[\rev{
V(\omega+\alpha h+(1-\alpha)g)\le \alpha V(\omega+ h)+(1-\alpha) V(\omega+ g) \quad \gamma\textnormal{-a.e.}}
\]
\end{definition}

\begin{remark}
\label{rem:convex}
An important example of a Cameron-Martin convex function $V$ on $\Omega$ is $V(\omega)=\eta(\omega_1)$ with $\eta:\R^d\to \R$ a convex function. 
\end{remark}

The precise question we consider is the following: Suppose $\mu$ is a probability measure on $\Omega$ of the form $d\mu(\omega)=e^{-V(\omega)}d\gamma(\omega)$, where $V:\Omega\to \R$ is a Cameron-Martin convex function. Let $A,O:\Omega\to\Omega$ be the causal optimal transport map and optimal transport map from $\gamma$ to $\mu$, respectively. Is it true that either $A$ or $O$ is a Cameron-Martin contraction with any constant $C$?

In order to answer this question our first task is to construct a suitable notion of derivative for Cameron-Martin transport maps $T:\Omega\to \Omega$ so that, in analogy with Lemma \ref{lem:rad}, we can establish a correspondence between being a  Cameron-Martin contraction and having a bounded derivative. The Malliavin derivative was defined for real-valued functions $F:\Omega\to \R$ but it can be defined for $H$-valued functions $\dot{\xi}:\Omega\to H$ as well \cite[p. 31]{nualart2006malliavin}. We start with the class $\mathcal S_{H}$ of $H$-valued smooth random variables: $\dot{\xi}\in\mathcal S_{H}$ if $\dot{\xi}=\sum_{i=1}^mF_i\dot{h}_i$ where $F_i\in \mathcal S$ (the class of smooth random variables, cf. section \ref{sec:prem}), $\dot{h}_i\in H$ for $i\in [m]$ for some $m\in \mathbb Z_+$. For $\dot{\xi}\in \mathcal S_{H}$ we define
\[
H\otimes H\ni D\dot{\xi}:=\sum_{i=1}^m DF_i\otimes \dot{h}_i
\]
and let $\mathbb D^{1,p}(H)$ be the completion of $\mathcal S_{H}$ under the norm
\[
\|\dot{\xi}\|_{1,p,H}=\left(\EE_{\gamma}\left[|\dot{\xi}|_{H}^p\right]+\EE_{\gamma}\left[|D\dot{\xi}|_{H\otimes H}^p\right]\right)^{\frac{1}{p}}.
\]
Since $D\dot{\xi}\in H\otimes H$ we may also view it as a linear operator $D\dot{\xi}:H\to H$ and we denote its operator norm by $|D\dot{\xi}|_{\mathcal L(H)}$.

\begin{definition}
\label{def:Girder}
Let  $T:\Omega\to\Omega$ be a measurable map of the form $T(\omega)=\omega+\xi(\omega)$ where $\dot{\xi} \in \mathbb D^{1,p}(H)$ for some $p\ge 1$. For any $\omega\in \Omega$ we define $DT(\omega):H\to H$ by
\[
DT(\omega)[\dot{h}]=\dot{h}+D\dot{\xi}(\omega)[\dot{h}], \quad\forall ~\dot{h}\in H.
\]
The operator norm of $DT(\omega): H\to H$, with $\omega\in \Omega$ fixed, is denoted by $|DT(\omega)|_{\mathcal L(H)}$.
\end{definition}

For the purpose of this work, we focus on measures $\mu$ on $\Omega$ of the form  $d\mu(\omega):=f(\omega_1)d\gamma(\omega)$. In addition, comparing the following lemma to Lemma \ref{lem:rad}, we see that it provides only one direction of the correspondence between Cameron-Martin contractions and bounded derivatives. A more general theory could be developed, at least in principle, but our goal in this work is to highlight key differences between causal optimal transport and optimal transport on the Wiener space, to which end the following suffices. 

\begin{lemma}
\label{lem:strongcontractder}
Let  $\mu$ be a measure on $\Omega$ of the form $d\mu(\omega):=f(\omega_1)d\gamma(\omega)$ and let $T:\Omega\to \Omega$ be a transport map from $\gamma$ to $\mu$ of the form $T(\omega)=\omega+\xi(\omega)$ where $\xi:=\int_0^{\cdot}\dot{\xi}$ for some $\dot{\xi}\in\mathbb D^{1,p}(H)$.  If $T$ is a Cameron-Martin contraction with constant $C$ then $|DT|_{\mathcal L(H)}\le C$ $\gamma$-a.e.
\end{lemma}

\begin{proof}
We first note that the Malliavin differentiability of $\xi$, as well as \cite[Theorem 5.7.2]{bogachev1998gaussian}, imply that, $\gamma$-a.e., for any $h,g\in \HH$,
\begin{align}
\label{eq:tder}
\lim_{\epsilon\downarrow 0}\frac{1}{\epsilon}\langle T(\omega+\epsilon h)-T(\omega),g\rangle_{\HH}&=\lim_{\epsilon\downarrow 0}\frac{1}{\epsilon}\langle \epsilon h+\xi(\omega+\epsilon h)-\xi(\omega),g\rangle_{\HH}=\langle [\Id_{\Omega}+D\dot{\xi}(\omega)][\dot{h}],\dot{g}\rangle_{H}\nonumber\\
&=\langle DT(\omega)[\dot{h}],\dot{g}\rangle_{H}.
\end{align}
Suppose now that $T$ is a Cameron-Martin contraction with constant $C$ so that, for a fixed $h\in \HH$, and any $\epsilon>0$,
\[
\sup_{g\in \HH, |g|_{\HH}=1}\frac{1}{\epsilon}\langle T(\omega+\epsilon h)-T(\omega),g\rangle_{\HH}=\frac{1}{\epsilon}|T(\omega+\epsilon h)-T(\omega)|_{\HH}\le C|h|_{\HH}.
\]
Taking $\epsilon\downarrow 0$ and using \eqref{eq:tder} shows that
\[
\sup_{\dot{g}\in H, |\dot{g}|_H=1}\langle DT(\omega)[\dot{h}],\dot{g}\rangle_{H}=|DT(\omega)[\dot{h}]|_{H}\le C
\]
so it follows that
\[
\sup_{\dot{h}\in H, |\dot{h}|_{H}=1}|DT(\omega)[\dot{h}]|_{H}=|DT(\omega)|_{\mathcal L(H)}\le C.  
\]
\end{proof}

\section{Causal optimal transport}
\label{sec:causal_strong}
In this section we answer in the negative, for causal optimal transport maps, the question raised in section \ref{sec:CMcontract}, thus proving the second part of Theorem \ref{thm:strongInformal}. In particular, we will construct a strictly log-concave function $f:\R^d\to \R$ such that the causal optimal transport map from $\gamma$ to  $d\mu(\omega):=f(\omega_1)d\gamma(\omega)$  is not a Cameron-Martin contraction, with any constant $C$. This indeed provides a negative answer in light of Remark \ref{rem:convex}. Our concrete example is the case where $fd\gamma_d$ is the measure of a one-dimensional Gaussian random variable conditioned on being positive. More precisely, fix a constant $\sigma> 0$ and let $f:\R\to \R$ be given by 
\[
f(x)=\frac{1_{[0,+\infty)}(x)e^{-\frac{x^2}{2\sigma}}}{\int_{\R}1_{[0,+\infty)}(y)e^{-\frac{y^2}{2\sigma}}d\gamma_1(y)}.
\]
The measure $f(x)d\gamma_1(x)$ is the measure on $\R$ of a centered Gaussian, whose variance is smaller than one, conditioned on being positive. We define a measure $\mu$ on $\Omega$ by setting
\[
d\mu(\omega):=f(\omega_1)d\gamma(\omega)
\]
and note that $f$ is strictly log-concave for all $\sigma>0$, and that $f$ is log-concave as $\sigma\to \infty$. To simplify computations we will take $\sigma\ge 1$. Finally, note that the assumptions of Proposition \ref{prop:existsde} hold in this case ($\cc\ge 0$).

\begin{remark}
\label{rem:nontrivcexam}
Given $d\mu(\omega)=f(\omega_1)d\gamma(\omega)$ let $p$ be the measure on $\R$ given by $p:=fd\gamma_1$ and let $\gamma_1^{a,\sigma}$ be the Gaussian measure on $\R$ with mean $a$ and variance $\sigma$. The natural examples for testing whether the causal optimal transport map $A$, between $\gamma$ to $\mu$,  is a Cameron-Martin contraction are $p=\gamma_1^{a,1}$, for some $a\in \R$, and $p=\gamma_1^{0,\sigma}$, for some $\sigma> 0$. We can expect these examples to show that $A$ is not a Cameron-Martin contraction since they saturate the bounds in Lemma \ref{lem:nablav}: When $p=\gamma_1^{a,1}$ we have $\nabla v(t,x)=0$ (saturation of Lemma \ref{lem:nablav}(2) under the assumption $\cc \ge 1$) and when $p=\gamma_1^{0,\sigma}$ we have that, in the limit $\sigma\downarrow 0$, $\nabla v(t,x)=-\frac{1}{1-t}$ (saturation of \eqref{eq:nablavlower}). Since in these cases $|\nabla v|$ is the largest, we can expect that $A$ will not be a Cameron-Martin contraction since its derivative will blow up. However, explicit calculation shows that $A$ is in fact a Cameron-Martin contraction for $p=\gamma_1^{a,1}$ and $p=\gamma_1^{0,\sigma}$. Hence, we require the construction of a more sophisticated example which we obtain by considering Gaussians conditioned on being positive. 
\end{remark}

In order to prove that $A$ is not a Cameron-Martin contraction we will use Lemma \ref{lem:strongcontractder}. We will show that with the example above, with positive probability, the derivative can be arbitrary large so that $A$ cannot be a Cameron-Martin contraction with any constant $C$. 

As mentioned already, the map $A$ is nothing but the F\"ollmer process $X$ \cite{lassalle2013causal}. This allows for the following convenient representation of the derivative of $A$. 
\begin{lemma}
\label{lem:derrep}
Let $A$ be the causal optimal transport map from $\gamma$ to $\mu$ and let $X$ be the solution of \eqref{eq:sde}. Fix $0<\epsilon<1$. For any $\dot{h}\in H$,
\[
(DA[\dot{h}])_t=\partial_t\langle DX_t,\dot{h}\rangle_H\quad\forall ~t\in [0,1-\epsilon]. 
\]
In addition,
\[
|DA[\dot{h}]|_{H}^2\ge \int_0^{1-\epsilon}\left(\dot{h}_t+\nabla v(t,X_t)\int_0^te^{\int_s^t\nabla v(r,X_r)dr}\dot{h}_sds\right)^2dt.
\]
\end{lemma}

\begin{proof}
We have $A=\Id_{\Omega}+\xi$ where $\dot{\xi}_t(\omega)=v(t,X_t(\omega))$ with the drift $v$ as in \eqref{eq:sde}. To show that
\[
(DA[\dot{h}])_t=\partial_t\langle DX_t,\dot{h}\rangle_H\quad\forall ~t\in [0,1-\epsilon]
\]
we start by noting that Proposition \ref{prop:existsde} gives
\[
DX_t=1_{[0,t]}+\int_0^t\nabla v(s,X_s)DX_sds,
\]
so
\begin{align}
\label{eq:ode}
\partial_t\langle DX_t,\dot{h}\rangle_H=\dot{h}_t+\nabla v(t,X_t)\langle DX_t,\dot{h}\rangle_H \quad \forall t\in [0,1]. 
\end{align}
Hence, our goal is to show that
\[
(DA[\dot{h}])_t=\dot{h}_t+\nabla v(t,X_t)\langle DX_t,\dot{h}\rangle_H\quad\forall ~t\in [0,1-\epsilon].
\]
To establish the above identity it suffices to show that 
\begin{align*}
\langle Dv[\dot{h}],\dot{g}\rangle_{H}=\int_0^1\nabla v(t,X_t)\langle DX_t,\dot{h}\rangle_H\dot{g}_tdt
\end{align*}
for every $h,g\in \HH$ with $\dot{g}_t=0$ for all $t\in[1-\epsilon,1]$. This indeed holds since, for such $g$ and $h$,
\begin{align*}
\langle Dv[\dot{h}],\dot{g}\rangle_{H}&=\lim_{\delta\downarrow 0}\frac{1}{\delta}\langle \xi(\omega+\delta h)-\xi(\omega),g\rangle_{\HH}=\int_0^1 \lim_{\delta\downarrow 0}\frac{v(t,X_t(\omega+\delta h))-v(t,X_t(\omega))}{\delta}\dot{g}_tdt\\
&=\int_0^1\nabla v(t,X_t)\langle DX_t,\dot{h}\rangle_H\dot{g}_tdt
\end{align*}
where in the second equality the integral and the limit were exchanged by the use of the dominated convergence theorem and as $v$ is Lipschitz on $[0,1-\epsilon]$ (Equation \eqref{eq:nablavlower} and Lemma \ref{lem:nablav}), while the third equality holds by the chain rule which can be applied as $v$ is Lipschitz \cite[Proposition 1.2.4]{nualart2006malliavin}.

The proof of the second part of the lemma follows by noting that the solution to the ordinary differential equation \eqref{eq:ode}, with initial condition $0$ at $t=0$, is
\[
\langle DX_t,\dot{h}\rangle_H=\int_0^te^{\int_s^t\nabla v(r,X_r)dr}\dot{h}_sds.
\]
\end{proof}
The next theorem is the main result of this section, showing that, with positive probability, $|DA|_{\mathcal L(H)}$ can be arbitrarily large, thus proving the second part of Theorem \ref{thm:strongInformal}.
\begin{theorem}
\label{thm:strong} 
Let $\ell\in \HH$ be given by $\ell(t)=t$ for $t\in [0,1]$. There exists a constant $c>0$ such that, for any $0<\epsilon <c$, there exists a measurable set $\mathcal E_{\epsilon}\subset \Omega$ satisfying
\[
\gamma[\mathcal E_{\epsilon}]>0
\]
and
\[
\gamma[|DA[\dot{\ell}]|_{H}> c\log (1/\epsilon)~|~\mathcal E_{\epsilon}]=1.
\]
\end{theorem}

The upshot of Theorem \ref{thm:strong} is that there exists a unit norm $\dot{h}\in H$, specifically $h=\ell$, such that, for any $b
>0$,  the event $\{|DA[\dot{h}]|_{H}>b\}$ has positive probability (possibly depending on $b$). Since 
\[
|DA|_{\mathcal L(H)}=\sup_{\dot{h}\in H:|\dot{h}|=1}|DA[\dot{h}]|_{H}
\]
we conclude that $A$ cannot be a Cameron-Martin contraction, with any constant $C$. 

Next we describe the idea behind the proof of Theorem \ref{thm:strong}. Fix $0<\epsilon<1$. By Lemma \ref{lem:derrep}, and as $\nabla v(t,X_t)=\partial_{xx}^2\log P_{1-t}f(X_t)$, we have
\begin{align}
\label{eq:DTh}
|DA[\dot{h}]|_{H}^2\ge \int_0^{1-\epsilon}\left(\dot{h}_t+\partial_{xx}^2\log P_{1-t}f(X_t)\int_0^te^{\int_s^t\partial_{xx}^2\log P_{1-r}f(X_r)dr}\dot{h}_sds\right)^2dt.
\end{align}
The idea of the proof is to construct a function $\eta_{\epsilon}:[0,1-\epsilon]\to \R$ and a constant $b>0$, such that
\[
\partial_{xx}^2\log P_{1-t}f(\eta_{\epsilon}(t))\approx -\frac{b}{1-t} \quad\forall ~t\in [\epsilon,1-\epsilon].
\]
If we choose $h=\ell$, and substitute $\eta_{\epsilon}(t)$ for $X_t$ in \eqref{eq:DTh}, then a computation shows that $|DA[\dot{h}]|_{H}$ is large. The final step is to show that, with positive probability, 
\[
\partial_{xx}^2\log P_{1-t}f(X_t)\approx \partial_{xx}^2\log P_{1-t}f(\eta_{\epsilon}(t)) \quad\forall ~t\in [\epsilon,1-\epsilon]. 
\]
This implies that, with positive probability, we can make $|DA[\dot{h}]|_{H}$ arbitrary large. We now proceed to make this idea precise. We start with the construction of the function $\eta_{\epsilon}$.

\begin{lemma}
\label{lem:eta}
For every $0<\epsilon<1$ there exists an absolutely continuous function $\eta_{\epsilon}:[0,1-\epsilon]\to \R$, with $\eta_{\epsilon}(0)=0$, such that
\[
 -\frac{1}{2}\frac{1}{1-t} -\frac{1}{1-t+\sigma}\le\partial_{xx}^2\log P_{1-t}f(\eta_{\epsilon}(t))\le -\frac{1}{8}\frac{1}{1-t}\quad \forall t\in [\epsilon,1-\epsilon].
 \]
Furthermore, for any $\epsilon>0$ and $\eta_{\epsilon}$ as above, there exists $\delta(\epsilon)>0$ such that, if $\tilde \eta:[0,1]\to \R$ satisfies
\[
\sup_{t\in [\epsilon,1-\epsilon]}|\tilde\eta(t)-\eta_{\epsilon}(t)|<\delta(\epsilon),
\] 
then 
\[
 -\frac{1}{2}\frac{1}{1-t} -\frac{1}{1-t+\sigma}\le\partial_{xx}^2\log P_{1-t}f(\tilde\eta(t))\le -\frac{1}{8}\frac{1}{1-t} \quad \forall t\in [\epsilon,1-\epsilon]
 \]
as well. 
\end{lemma}

\begin{proof}
Fix $0<\epsilon<1$ and let $Z:=\int_{\R}1_{[0,+\infty)}(y)e^{-\frac{y^2}{2\sigma}}d\gamma_1(y)$ so that $f(x)=Z^{-1}1_{[0,+\infty)}(x)e^{-\frac{x^2}{2\sigma}}$. Let $\varphi$ be the density of the standard Gaussian measure on $\R$ and  let $\Phi(x):=\int_{-\infty}^x\varphi(y)dy$ be its cumulative distribution function. Making the change of variables $y\mapsto \frac{y-x}{\sqrt{t}}$ we get 
\begin{align*}
P_tf(x)&=\frac{1}{Z}\int_{\R}1_{[0,+\infty)}(x+\sqrt{t}y)e^{-\frac{(x+\sqrt{t}y)^2}{2\sigma}}\frac{e^{-\frac{y^2}{2}}}{\sqrt{2\pi}}dy=\frac{1}{Z\sqrt{t}\sqrt{2\pi}}\int_0^{\infty}e^{-\frac{y^2}{2\sigma}}e^{-\frac{(y-x)^2}{2t}}dy.
\end{align*}
Since
\begin{align*}
&\frac{y^2}{\sigma}+\frac{(y-x)^2}{t}=\left(\frac{1}{\sigma}+\frac{1}{t}\right)y^2-2\frac{xy}{t}+\frac{x^2}{t}=\frac{t+\sigma}{t\sigma}\left[y-\frac{\sigma}{t+\sigma}x\right]^2+\left[\frac{1}{t}-\frac{\sigma}{t(t+\sigma)}\right]x^2\\
&=\frac{t+\sigma}{\sigma t}\left[y-\frac{\sigma}{t+\sigma}x\right]^2+\frac{x^2}{t+\sigma},
\end{align*}	
we have
\begin{align*}
\frac{1}{\sqrt{t}Z\sqrt{2\pi}}\int_0^{\infty}e^{-\frac{y^2}{2\sigma}}e^{-\frac{(y-x)^2}{2t}}dy=\frac{\sqrt{\frac{\sigma }{t+\sigma}}}{Z}e^{-\frac{1}{2}\frac{x^2}{t+\sigma}}\int_0^{\infty}\frac{\exp\left[-\frac{1}{2\frac{\sigma t}{t+\sigma}}\left[y-\frac{\sigma}{t+\sigma}x\right]^2\right]}{\sqrt{2\pi \frac{\sigma t}{t+\sigma}}}dy.
\end{align*}
The cumulative distribution function of a Gaussian with mean $\frac{\sigma}{t+\sigma}x$ and variance $\frac{\sigma t}{t+\sigma}$ is $y\mapsto \Phi\left(\frac{y-\frac{\sigma}{t+\sigma}x}{\sqrt{\frac{\sigma t}{t+\sigma}}}\right)$ so,
\begin{align*}
\int_0^{\infty}\frac{\exp\left[-\frac{1}{2\frac{\sigma t}{t+\sigma}}\left[y-\frac{\sigma}{t+\sigma}x\right]^2\right]}{\sqrt{2\pi \frac{\sigma t}{t+\sigma}}}dy=1- \Phi\left(\frac{-\frac{\sigma}{t+\sigma}x}{\sqrt{\frac{\sigma t}{t+\sigma}}}\right)=1- \Phi\left(-\sqrt{\frac{\sigma}{t(t+\sigma)}}x\right).
\end{align*}
Since $\Phi(-y)=1-\Phi(y)$ we conclude that
\[
 P_tf(x)=\frac{\sqrt{\frac{\sigma }{t+\sigma}}}{Z}e^{-\frac{1}{2}\frac{x^2}{t+\sigma}}\Phi\left(\sqrt{\frac{\sigma}{t(t+\sigma)}}x\right)\quad \forall t\in [0,1]. 
\]
Let $\sigma_t:=\sqrt{\frac{\sigma}{(1-t)(1-t+\sigma)}}$ and let $m(x):=\frac{\varphi(-x)}{\Phi(-x)}$ be the inverse Mills ration, where $\varphi$ is the density of the standard Gaussian on $\R$. Then, for $t\in [0,1]$,
\begin{align*}
\partial_x\log P_{1-t}f(x)=\sigma_t\frac{\varphi(\sigma_t x)}{\Phi(\sigma_t x)}-\frac{x}{1-t+\sigma}=\sigma_tm(-\sigma_t x)-\frac{x}{1-t+\sigma},
\end{align*}
and using the readily verified relation
\[
m'(x)=m^2(x)+xm(x),
\]
we get
\[
\partial_{xx}^2\log P_{1-t}f(x)=-\sigma_t^2 m'(-\sigma_t x)-\frac{1}{1-t+\sigma}=-\sigma_t^2[m^2(-\sigma_t x)-\sigma_t xm(-\sigma_t x)]-\frac{1}{1-t+\sigma}. 
\]

Set $\eta_{\epsilon}(t):=-\sigma_t^{-1}c$ for $t\in [\epsilon,1-\epsilon]$ with $c$ a constant to be determined shortly, and continue $\eta_{\epsilon}$ to $[0,\epsilon]$ in such a way that $\eta_{\epsilon}$ is absolutely continuous with derivative in $L^2([0,1-\epsilon])$, and $\eta_{\epsilon}(0)=0$. Then,
\[
\partial_{xx}^2\log P_{1-t}f(\eta_{\epsilon}(t))=-\sigma_t^2[m^2(c)+cm(c)]-\frac{1}{1-t+\sigma}\quad \forall~ t\in [\epsilon,1-\epsilon].
\]
Since  $m^2(0)+0m(0)=\frac{2}{\pi}>\frac{1}{2}$, and as $\lim_{x\to-\infty}[m^2(x)+xm(x)]=0$,  the continuity of $m$ implies that there exists $c<0$ such that $m^2(c)+cm(c)=\frac{1}{3}$. With this choice of $c$ we have $\frac{1}{4}<m^2(c)+cm(c)<\frac{1}{2}$ so 
\[
-\frac{\sigma_t^2}{2}-\frac{1}{1-t+\sigma}\le \partial_{xx}^2\log P_{1-t}f(\eta_{\epsilon}(t))\le -\frac{\sigma_t^2}{4}-\frac{1}{1-t+\sigma}\quad \forall~ t\in [\epsilon,1-\epsilon].
\]
Whenever $\sigma\ge 1$, 
\[
\frac{1}{2}\frac{1}{1-t}\le\sigma_t^2=\frac{\sigma}{(1-t)(1-t+\sigma)}\le \frac{1}{1-t}
\]
so 
\[
 -\frac{1}{2}\frac{1}{1-t}-\frac{1}{1-t+\sigma}\le\partial_{xx}^2\log P_{1-t}f(\eta_{\epsilon}(t))\le -\frac{1}{8}\frac{1}{1-t}-\frac{1}{1-t+\sigma}.
\]
Since $ -\frac{1}{1-t+\sigma}\le 0$ we get
\[
 -\frac{1}{2}\frac{1}{1-t} -\frac{1}{1-t+\sigma}\le\partial_{xx}^2\log P_{1-t}f(\eta_{\epsilon}(t))\le -\frac{1}{8}\frac{1}{1-t}.
\]
This completes the proof of the first part of the lemma. 

For the second part of the lemma, given $\epsilon$ and $\eta_{\epsilon}$ as above, use the continuity of $m$, and that $m^2(c)+cm(c)=\frac{1}{3}$, to choose $\delta'>0$ such that $|c'-(-c)|<\delta'\Rightarrow \frac{1}{4}<m^2(-c')-c'm(-c')< \frac{1}{2}$. Now let $\delta(\epsilon):=\frac{\delta'}{\sigma_{1-\epsilon}}$  and let $\tilde \eta:[0,1]\to \R$ be any function such that $\sup_{t\in [\epsilon,1-\epsilon]}|\eta_{\epsilon}(t)-\tilde\eta(t)|<\delta(\epsilon)$. Then, 
\begin{align*}
|\sigma_t\tilde \eta(t)-(-c)|=|\sigma_t\tilde \eta(t)-\sigma_t\eta_{\epsilon}(t)|\le |\sigma_{1-\epsilon}\tilde \eta(t)-\sigma_{1-\epsilon}\eta_{\epsilon}(t)|<\delta' \quad\forall~t\in [\epsilon,1-\epsilon]
\end{align*}
so $\frac{1}{4}<m^2(-\sigma_t\tilde \eta(t))-(-\sigma_t\tilde \eta(t))m(-\sigma_t\tilde \eta(t))< \frac{1}{2}$. It follows, as above, that
\[
-\frac{\sigma_t^2}{2}-\frac{1}{1-t+\sigma}\le \partial_{xx}^2\log P_{1-t}f(\tilde\eta(t))\le -\frac{\sigma_t^2}{4}-\frac{1}{1-t+\sigma}\quad \forall~ t\in [\epsilon,1-\epsilon],
\]
and we continue as above to complete the proof.
\end{proof}
Next we show that if we take $h=\ell$, and substitute for $X$ in \eqref{eq:DTh} a function $\tilde\eta$ which is close to the function $\eta_{\epsilon}$ constructed in Lemma \ref{lem:eta}, then $|DA[\dot{h}]|_{H}$ is large. 

\begin{lemma}
\label{lem:DTh_large}
There exists a constant $c>0$ with the following properties. Fix $0<\epsilon <c$ and let $\eta_{\epsilon}$ and $\delta:=\delta(\epsilon)$ be as in Lemma \ref{lem:eta}. Let $\tilde \eta:[0,1]\to\R$ be any function such that $\sup_{t\in [\epsilon,1-\epsilon]}|\tilde \eta(t)-\eta_{\epsilon}(t)|<\delta$. Then,
\[
\int_0^{1-\epsilon}\left(1+\partial_{xx}^2\log P_{1-t}f(\tilde \eta(t))\int_0^te^{\int_s^t\partial_{xx}^2\log P_{1-r}f(\tilde \eta(r))dr}ds\right)^2dt\ge c\log(1/\epsilon).
\]
\end{lemma}
\begin{proof}
By equation \eqref{eq:nablavlower} and Lemma \ref{lem:nablav}(2) (with $\cc=1$),
\[
-\frac{1}{1-t}\le\partial_{xx}^2\log P_{1-t}f(x)\le 0
\]
so, for $s\le \epsilon$,
\[
\int_s^{\epsilon}\partial_{xx}^2\log P_{1-r}f(\tilde \eta(r))dr\ge \int_0^{\epsilon}\partial_{xx}^2\log P_{1-r}f(\tilde \eta(r))dr\ge \int_0^{\epsilon}-\frac{1}{1-r}dr=\log (1-\epsilon).
\]
It follows that, for $\epsilon<t$,
\begin{align*}
\int_0^te^{\int_s^t\partial_{xx}^2\log P_{1-r}f(\tilde \eta(r))dr}ds&\ge \int_{\epsilon}^te^{\int_s^t\partial_{xx}^2\log P_{1-r}f(\tilde \eta(r))dr}ds+(1-\epsilon)\int_0^{\epsilon}e^{\int_{\epsilon}^t\partial_{xx}^2\log P_{1-r}f(\tilde \eta(r))dr}ds\\
&\ge \frac{1}{2}\int_{\epsilon}^te^{\int_s^t\partial_{xx}^2\log P_{1-r}f(\tilde \eta(r))dr}ds+\frac{1}{2}\int_0^{\epsilon}e^{\int_{\epsilon}^t\partial_{xx}^2\log P_{1-r}f(\tilde \eta(r))dr}ds\\
&=\frac{1}{2} \int_0^te^{\int_{\epsilon\vee s}^t\partial_{xx}^2\log P_{1-r}f(\tilde \eta(r))dr}ds.
\end{align*}
Using the lower bound of Lemma \ref{lem:eta} we get,
\[
\int_0^te^{\int_s^t\partial_{xx}^2\log P_{1-r}f(\tilde \eta(r))dr}ds\ge \frac{1}{2}\int_0^te^{\int_{\epsilon\vee s}^t\left[ -\frac{1}{2}\frac{1}{1-r} -\frac{1}{1-r+\sigma}\right]dr}ds \quad \forall t\in [\epsilon,1-\epsilon],
\]
and using the upper bound of Lemma \ref{lem:eta} we conclude that, for $t\in [\epsilon,1-\epsilon]$,
\begin{align*}
&\partial_{xx}^2\log P_{1-t}f(\tilde \eta(t))\int_0^te^{\int_s^t\partial_{xx}^2\log P_{1-r}f(\tilde \eta(r))dr}ds\le  -\frac{1}{16}\frac{1}{1-t} \int_0^te^{\int_{\epsilon\vee s}^t\left[ -\frac{1}{2}\frac{1}{1-r} -\frac{1}{1-r+\sigma}\right]dr}ds\\
&\le \rev{ -\frac{1}{16}\frac{1}{1-t} \int_0^te^{\int_{s}^t\left[ -\frac{1}{2}\frac{1}{1-r} -\frac{1}{1-r+\sigma}\right]dr}ds=-\frac{1}{16}\frac{1}{1-t} \left\{\int_0^t \left[\frac{\sqrt{1-t}}{\sqrt{1-s}}\cdot\frac{1-t+\sigma}{1-s+\sigma}\right]ds\right\}}\\
&\rev{\leq\frac{1}{16}\frac{\sqrt{1-t}-1}{\sqrt{1-t}},}
\end{align*}
since \rev{$\sigma\geq 1$ implies $\frac{1-t+\sigma}{1-s+\sigma} \geq \frac{1}{2}$}. 
%
In particular, letting $\rev{t_0:=\frac{288}{289}}$ \rev{for which $1+\frac{1}{16}\frac{\sqrt{1-t_0}-1}{\sqrt{1-t_0}}=0$}, we get
\[
1+\partial_{xx}^2\log P_{1-t}f(\tilde \eta(t))\int_0^te^{\int_s^t\partial_{xx}^2\log P_{1-r}f(\tilde \eta(r))dr}ds\le  0\quad \forall t\in[t_0,1-\epsilon].
\]
It follows that 
\begin{align*}
&\int_0^{1-\epsilon}\left(1+\partial_{xx}^2\log P_{1-t}f(\tilde \eta(t))\int_0^te^{\int_s^t\partial_{xx}^2\log P_{1-r}f(\tilde \eta(r))dr}ds\right)^2dt\\
&\ge \int_{t_0}^{1-\epsilon}\left(1+\partial_{xx}^2\log P_{1-t}f(\tilde \eta(t))\int_0^te^{\int_s^t\partial_{xx}^2\log P_{1-r}f(\tilde \eta(r))dr}ds\right)^2dt\\
&\rev{\ge \int_{t_0}^{1-\epsilon}\left(1-\frac{1}{16}\frac{1-\sqrt{1-t}}{\sqrt{1-t}}\right)^2dt\ge \frac{1}{512}\int_{t_0}^{1-\epsilon}\frac{1}{1-t}dt-\frac{1}{8}\int_{t_0}^{1-\epsilon}\frac{1}{\sqrt{1-t}}dt}\\
&\rev{\ge \frac{1}{512}\log(1/\epsilon)+\frac{1}{512}\log(1/289)-\frac{1}{68}},
\end{align*}
which completes the proof. 
\end{proof}

It remains to show that, with positive probability, $X$ is close to $\eta_{\epsilon}$.
\begin{lemma}
\label{lem:bridgesup}
Fix $0<\epsilon <c$, with $c$ as in Lemma \ref{lem:DTh_large}, and let $\delta:=\delta(\epsilon)$ be as in Lemma \ref{lem:eta}. Then, the set 
\[
\mathcal E_{\epsilon,\delta}:=\left\{\tilde \eta\in \Omega:\tilde \eta(0)=0~\text{ and }\sup_{t\in [\epsilon,1-\epsilon]}|\tilde \eta(t)-\eta_{\epsilon}(t)|<\delta\right\}\subset \Omega
\]
is measurable and has positive probability.
\end{lemma}
\begin{proof}
The measurability of $\mathcal E_{\epsilon,\delta}$ follows as the sigma-algebra $\mathcal F$ is generated by the Borel sets of $\Omega$ with respect to the uniform norm. To show that $\mathcal E_{\epsilon,\delta}$ has positive probability it will be useful to note that the F\"ollmer process $X$ is a mixture of Brownian bridges in the following sense. Let $(\tilde \Omega, \tilde{\mathcal F},\tilde \PP)$ be any probability space which supports a Brownian motion $\tilde B=(\tilde B_t)_{t\in [0,1]}$ and a random vector $ Y\sim p$, independent of $\tilde B$. Define the process $Z$ by $Z_t:=\tilde B_t-t(\tilde B_1- Y)$ for $t\in [0,1]$ so, conditioned on $Y$, $Z$ is a Brownian bridge starting at $0$ and terminating at $Y$. Given a set $\mathcal B\in \mathcal F$ we have \cite{follmer1988random},
\[
\gamma[X\in\mathcal B]=\tilde \PP[\{\tilde \omega\in\tilde \Omega: Z(\tilde \omega)\in\mathcal B\}].
\]
Since 
\[
\tilde \PP[\{\tilde \omega\in\tilde \Omega: Z(\tilde \omega)\in\mathcal B\}]=\EE[\tilde \PP[\{\tilde \omega\in\tilde \Omega: Z(\tilde \omega)\in\mathcal B\}|Y]],
\]
it will suffice to show that, for any $b\in \R$ and $Z^b_t:=\tilde B_t-t(\tilde B_1-b)$, we have 
\[
\tilde \PP[\{\tilde \omega\in\tilde \Omega: Z^b(\tilde \omega)\in\mathcal E_{\epsilon,\delta}\}]>0.
\]
This is equivalent to the following statement: Fix $b\in \R$ and let $\eta_{\epsilon}$ be as in Lemma \ref{lem:eta}. Then, for any $\epsilon\in (0,1)$ and $\delta>0$, 
\begin{align}
\label{eq:bridge}
\tilde \PP\left[\sup_{t\in [\epsilon,1-\epsilon]}|Z_t^b-\eta_{\epsilon}(t)|< \delta \right]>0. 
\end{align}
To prove \eqref{eq:bridge} define the function $h:[0,1]\to \R$ by
\[
h_t=
\begin{cases}
\eta_{\epsilon}(t),&t\in [0,1-\epsilon]\\
\frac{b-\eta_{\epsilon}(1-\epsilon)}{\epsilon}t+\frac{\eta_{\epsilon}(1-\epsilon)-(1-\epsilon)b}{\epsilon},&t\in(1-\epsilon, 1],
\end{cases}
\]
and note that the construction of $\eta_{\epsilon}$ ensures that $h\in H^1$. Then, for any $\delta>0$,
\[
\tilde \PP\left[\sup_{t\in [0,1]}|\tilde B_t-h_t|<\frac{\delta}{2}\right]>0,
\]
because $H^1$ is dense in $C_0([0,1])$. It follows that 
\begin{align*}
&\tilde \PP\left[\sup_{t\in [0,1-\epsilon]}|Z^b_t-\eta_{\epsilon}(t)|<\delta\right]\ge \tilde \PP\left[\sup_{t\in [0,1-\epsilon]}|\tilde B_t-\eta_{\epsilon}(t)|<\frac{\delta}{2}, ~|\tilde B_1-b|<\frac{\delta}{2}\right]\\
&=\tilde \PP\left[\sup_{t\in [0,1-\epsilon]}|\tilde B_t-h_t|\le\frac{\delta}{2},~ |\tilde B_1-h_1|<\frac{\delta}{2}\right]\ge \tilde \PP\left[\sup_{t\in [0,1]}|\tilde B_t-h_t|<\frac{\delta}{2}\right]>0.
\end{align*}
\end{proof}

We can now complete the proof of Theorem \ref{thm:strong} by combining the above lemmas. 

\begin{proof}{\textnormal{(of Theorem \ref{thm:strong})}}
Fix $0<\epsilon<c$ and let $\eta_{\epsilon}$ and $\delta:=\delta(\epsilon)$ be as in Lemma \ref{lem:eta}. Let $\mathcal E_{\epsilon}:=\mathcal E_{\epsilon,\delta}$ be as in Lemma \ref{lem:bridgesup} so $\gamma[\mathcal E_{\epsilon}]>0$. Conditioned on $\mathcal E_{\epsilon}$, Lemma \ref{lem:DTh_large} implies that there exists $c>0$ such that
\[
|DA|[\dot{\ell}]|_H\ge\int_0^{1-\epsilon}\left(1+\partial_{xx}^2\log P_{1-t}f(X_t)\int_0^te^{\int_s^t\partial_{xx}^2\log P_{1-r}f(X_r)dr}ds\right)^2dt\ge c\log(1/\epsilon),
\]
where the first inequality holds by \eqref{eq:DTh}. It follows that
\[
\gamma[|DA[\dot{\ell}]|_{H}> c\log (1/\epsilon)~|~\mathcal E_{\epsilon}]=1.
\]
\end{proof}

\section{Optimal transport}
\label{sec:opt}

The fundamental results of optimal transport on the Wiener space are due to Feyel and {\"U}st{\"u}nel \cite{feyel2004monge}. One of their results, pertaining to our setting, is the following analogue of a theorem of Brenier in Euclidean spaces \cite[Theorem 2.12]{villani2003topics}.

\begin{theorem}{\textnormal{(\cite[Theorem 4.1]{feyel2004monge})}}
\label{thm:OTfu}
Let $\mu$ be a measure on $\Omega$ defined by $\frac{d\mu}{d\gamma}(\omega)=F(\omega)$, where $F:\Omega \to \R$ is a positive function $\gamma$-a.e., such that $W_2(\gamma,\mu)<\infty$. Then, there exists a unique (up to a constant) convex map $\phi:\Omega\to \R$, such that $\mu$ is the pushforward of $\gamma$ under $\nabla\phi\in \HH$, where $(\nabla\phi(\omega))_t=\int_0^tD_s\phi(\omega)ds$ with $D\phi$ being the Malliavin derivative of $\phi$, and $\EE_{\gamma}\left[| \omega-\nabla\phi(\omega)|_{\HH}^2\right]=W_2^2(\gamma,\mu)$. 
\end{theorem}
The main result in this section, Theorem \ref{thm:OT}, establishes a Cameron-Martin contraction for $\phi$ in the case where $F(\omega)=f(\omega_1)$ is $(\cc-1)$-log-concave for $\cc\ge 0$; informally, $\mu$ is $\cc$-log-concave for $\cc\ge 0$. This proves the first part of Theorem \ref{thm:strongInformal}. Our motivation is primarily to show that there are settings where the optimal transport map is a Cameron-Martin contraction, contrary to the causal optimal transport map (second part of Theorem \ref{thm:strongInformal} via Theorem \ref{thm:strong}).

\begin{theorem}
\label{thm:OT}
Let $\mu$ be a probability measure on $\Omega$ given by $d\mu(\omega)=f(\omega_1)d\gamma(\omega)$ where $f:\R^d\to \R$ is $(\cc-1)$-log-concave for some $\cc\ge 0$. Suppose, in addition, that the optimal transport map in $\R^d$ from $\gamma_d$ to $p:=f\gamma_d$ is twice-differentiable\footnote{See \cite[Theorem 1]{cordero2019regularity} for conditions under which such regularity holds.}. Then, the optimal transport map in $\Omega$ from $\gamma$ to $\mu$ is a Cameron-Martin contraction with constant $\max\left\{\frac{1}{\sqrt{\cc}},1\right\}$. 
\end{theorem}

Recall that in the Euclidean setting, if $p=fd\gamma_d$ with  $f:\R^d\to \R$ $(\cc-1)$-log-concave for some $\cc\ge 0$, then the optimal transport map is a contraction with constant $\frac{1}{\sqrt{\cc}}$  \cite[Theorem 2.2]{kolesnikov2011mass}. However, the fact that the analogous result of Theorem \ref{thm:OT} gives the constant  $\max\left\{\frac{1}{\sqrt{\cc}},1\right\}$ is no coincidence. If the constant of the Cameron-Martin contraction of the optimal transport map on the Wiener space was in fact $\frac{1}{\sqrt{\cc}}$, then, arguing as in section \ref{sec:funct}, we would conclude that the measure $\mu$ on $\Omega$ given by $d\mu(\omega)=f(\omega_1)d\gamma(\omega)$ satisfies a Poincar\'e inequality with constant $\frac{1}{\cc}$. But as argued in \cite[Remark 2.7]{KK}, we expect that if $\mu$ is equivalent to $\gamma$ then its Poincar\'e constant must be greater or equal to 1. Hence, we may expect Theorem \ref{thm:OT} to be the optimal result.\footnote{If we allow $\mu$ to be singular with respect to $\gamma$ then we can in fact get a Poincar\'e constant smaller than 1 \cite[Theorem 6.1]{feyel2000notion}. The prototypical example is when $\mu$ is the measure of the process $\sigma W$ where $W$ is a Brownian motion and $\sigma\neq 1$. Then $\mu$ is singular with respect to $\gamma$ and $\mu$ satisfies the conditions of \cite[Theorem 6.1]{feyel2000notion}. In this sense, the Wiener space is fundamentally different than the Euclidean space where the measure of any (non-zero) multiple of the  standard Gaussian is equivalent to $\gamma_d$.} (We note that the question of the contraction properties of the optimal transport map on Wiener space was addressed in \cite[\S 6]{MR2200166}.)

\begin{proof} 
We start by explicitly computing the optimal transport map $O$. Define the measure $p$ on $\R^d$ by $\frac{dp}{d\gamma_d}(x):=f(x)$ and let $\phi_d:\R^d\to \R$ be the convex function such that $\nabla\phi_d$ is the optimal transport map from $\gamma_d$ to $p$. Define $\dot{\xi}\in H$ by $\dot{\xi}_t:=\nabla \phi_d(\omega_1)-\omega_1$, for all $t\in [0,1]$, and let $O':\Omega\to \Omega$ be given by $O'(\omega)_t:=\omega_t+\xi_t$ (where $\xi_t:=\int_0^t\dot{\xi}_sds$) for $t\in [0,1]$. We claim that $O=O'$. Indeed, by the uniqueness part of Theorem \ref{thm:OTfu}, and as $O'$ is convex (according to Definition \ref{def:convex} and Remark \ref{rem:convex}), it suffices to show that the pushforward of $\gamma$ by $O'$ is $\mu$, and that $W_2^2(\gamma,\mu)=\EE_{\gamma}\left[|\omega-O'(\omega) |_{\HH}^2\right]$. The fact that $\mu$ is the pushforward of $\gamma$ by $O'$ follows by construction: Let $Z^b$ be a Brownian bridge on $[0,1]$ starting at $0$ and terminating at $b\in\R^d$. Then, for any $\eta:\Omega\to \R$ measurable continuous and bounded, we have
\begin{align*}
\EE_{\gamma}[\eta(O'(\omega))]=\EE_{\gamma}[\EE_{\gamma}[\eta(O'(\omega))]|\omega_1]=\int_{\R^d} \EE[\eta(Z^{\nabla\phi_d(x)})]d\gamma_d(x)=\EE_{Y\sim p}[\EE[\eta(Z^Y)]]=\EE_{\mu}[\eta(\omega)]. 
\end{align*}

To see that $O'$ is in fact the actual optimal transport map we compute
\begin{align*}
\EE_{\gamma}\left[|\omega-O'(\omega) |_{\HH}^2\right]=\EE_{\gamma}\left[|\xi|_{\HH}^2\right]=\EE_{\gamma}\left[\int_0^1 |\omega_1-\nabla\phi_d(\omega_1)|^2dt\right]=W_2^2(\gamma_d,p)\le W_2^2(\gamma,\mu),
\end{align*}
which shows that $O'=O$. Next we show that $O$ is a Cameron-Martin contraction. Fix $h\in \HH$ and compute
\begin{align*}
|O(\omega+h)-O(\omega)|_{\HH}^2=\int_0^1 |\dot{h}_t+\dot{\xi}_t(\omega+h)-\dot{\xi}_t(\omega)|^2dt=\int_0^1 |\nabla\phi_d(\omega_1+h_1)-\nabla\phi_d(\omega_1)+\dot{h}_t-h_1|^2 dt.
\end{align*}
Since 
\[
\nabla\phi_d(\omega_1+h_1)-\nabla\phi_d(\omega_1)=\left[\int_0^1\nabla^2\phi_d(\omega_1+rh_1)dr\right]h_1,
\]
we may write
\begin{align*}
|O(\omega+h)-O(\omega)|_{\HH}^2&=\int_0^1 \left|\left[\int_0^1\nabla^2\phi_d(\omega_1+rh_1)dr-\Id_d\right]h_1+\dot{h}_t\right|^2 dt=:\int_0^1 \left|Mh_1+\dot{h}_t\right|^2 dt\\
&=\int_0^1\left\{|Mh_1|^2+2\langle Mh_1,\dot{h}_t\rangle+|\dot{h}_t|^2\right\}dt=|Mh_1|^2+2\langle Mh_1,h_1\rangle+|h|_{\HH}^2. 
\end{align*}
Since $p$ is $\cc$-log-concave, we have $0\preceq	\nabla^2\phi_d\preceq	 \frac{1}{\sqrt{\cc}}\Id_d$ \cite[Theorem 2.2]{kolesnikov2011mass}, and hence, $-\Id_d \preceq  M\preceq \left(\frac{1}{\sqrt{\cc}}-1\right)\Id_d$. 
There are now two cases: $\cc\ge 1$ and $\cc<1$. Suppose $\cc\ge 1$. We claim that $|Mh_1|^2+2\langle Mh_1,h_1\rangle\le 0$. Indeed, the latter is equivalent to $|[M+\Id_d]h_1|^2\le |h_1|^2$, which is true since $0\preceq M+\Id_d \preceq\frac{1}{\sqrt{\cc}}\Id_d\preceq \Id_d$. This shows that
\[
\cc \ge 1\quad \Longrightarrow \quad |O(\omega+h)-O(\omega)|_{\HH}\le  |h|_{\HH}\quad\forall h\in \HH.
\]
Suppose now that $\cc<1$. We claim that $|Mh_1|^2+2\langle Mh_1,h_1\rangle\le \left(\frac{1}{\cc}-1\right)|h|_{\HH}^2$. Indeed, since $|[M+\Id_d]h_1|^2\le \frac{1}{\cc}|h_1|^2$, we get $|Mh_1|^2+2\langle Mh_1,h_1\rangle\le  \left(\frac{1}{\cc}-1\right)|h_1|^2\le \left(\frac{1}{\cc}-1\right)|h|_{\HH}^2$ where we used $|h_1|^2\le |h|_{\HH}^2$ by the Cauchy-Schwarz inequality. This shows that 
\[
\cc < 1\quad \Longrightarrow \quad |O(\omega+h)-O(\omega)|_{\HH}\le  \frac{1}{\sqrt{\cc}} |h|_{\HH}\quad\forall h\in \HH.
\]
\end{proof}

\begin{remark}
The example of a one-dimensional Gaussian conditioned on being positive, constructed in section \ref{sec:causal_strong}, does not exactly satisfy the assumptions of Theorem \ref{thm:OT} since the second-derivative of the transport map between $\gamma_1$ and the conditioned Gaussian $p=f\gamma_1$ does not exist at every point in $\R$. Nonetheless, the statement of Theorem \ref{thm:OT} still holds true in this case. In the example of section \ref{sec:causal_strong}, the optimal transport map is explicit,  $\nabla\phi_1=F_p^{-1}\circ F_{\gamma_1}$, where $F_p$ and $F_{\gamma_1}$ are the cumulative distribution functions of $p$ and $\gamma_1$, respectively.  Computing the derivatives of this map we see that $\phi_1$ is twice-differentiable everywhere. Hence, the proof of Theorem \ref{thm:OT} still goes through since $\nabla^2\phi_1$ must exists everywhere on the line $\omega_1+rh_1$.
\end{remark}

\begin{remark}
The proof of Theorem \ref{thm:OT} shows that the optimal transport map $O$ between $\gamma$ and $\mu(d\omega)=f(\omega_1)\gamma(d\omega)$ is essentially the optimal transport map in $\R^d$ between $\gamma_d$ and $f\gamma_d$. This explains why we cannot use the optimal transport map on Wiener space instead of the Brownian transport map, since the desired contraction properties for the optimal transport maps in $\R^d$ are still unknown. 
\end{remark}

\begin{remark}
\label{rem:thetaToTheta}
Inspection of the proof of Theorem \ref{thm:OT} reveals that if $\theta:\R^d\to \R^d$ is a contraction of $\gamma_d$ onto a target measure $p$, then one can construct a Cameron-Martin contraction $\Theta:\Omega\to \Omega$ of $\gamma$ onto $d\mu(\omega):=p(\omega_1)d\gamma(\omega)$. Indeed, define $\Theta_t(\omega):=\omega_t+t\theta(\omega_1)-t\omega_1$, for $t\in [0,1]$, and repeat the computation in the proof of Theorem \ref{thm:OT}. Note that, in general, $\Theta$ will not be the optimal transport map on the Wiener space, unless $\theta$ is the optimal transport map on the Euclidean space.
\end{remark}

\bibliographystyle{abbrv}
\bibliography{refBMTarxiv}

\end{document}